\documentclass[11pt]{article}

\usepackage{geometry,showlabels, fullpage, authblk}
\usepackage{xr-hyper}
\usepackage{times,enumitem}
\usepackage{hyperref}[]
\hypersetup{
	colorlinks=true,
	linkcolor=blue,
	filecolor=magenta,      
	urlcolor=cyan,
	citecolor=blue,
}
\usepackage{ulem}
\usepackage{subcaption}
\usepackage{url}

\usepackage{bigints}
\usepackage{amsfonts,amsmath,amssymb,amsthm, bbm}
\usepackage{wrapfig}
\usepackage{verbatim,float,url}
\usepackage{graphicx}
\usepackage{subcaption}
\usepackage[round]{natbib}   
\usepackage[english]{babel}
\usepackage{cancel}
\usepackage{color}
\usepackage[dvipsnames]{xcolor}
\usepackage{cleveref}
\usepackage{multirow}
\usepackage{booktabs}

\usepackage{mathtools}

\newtheorem{theorem}{Theorem}
\newtheorem{lemma}[theorem]{Lemma}

\usepackage{mdwlist}

\theoremstyle{definition}
\newtheorem{definition}{Definition}
\newtheorem{example}{Example}

\usepackage{tikz}  
\usetikzlibrary{trees}
\usetikzlibrary{arrows.meta}
\usetikzlibrary{decorations.pathreplacing}

\def\holder{H\"{o}lder }
\def\pct[#1]{\ensuremath{#1^{\text{th}}}}
\def\pctst[#1]{\ensuremath{#1^{\text{st}}}}

\graphicspath{{figs/}}

\usepackage{algorithm}
\usepackage{algpseudocode}

\algnewcommand\algorithmicinput{\textbf{INPUT:}}
\algnewcommand\INPUT{\item[\algorithmicinput]}
\algnewcommand\algorithmicoutput{\textbf{OUTPUT:}}
\algnewcommand\OUTPUT{\item[\algorithmicoutput]}

\usepackage{mathtools}
\DeclarePairedDelimiter\ceil{\lceil}{\rceil}
\DeclarePairedDelimiter\floor{\lfloor}{\rfloor}

\theoremstyle{plain}
\newtheorem{assumption}{Assumption}
\allowdisplaybreaks
\usepackage{bm}

\usepackage{cleveref}

\title{Quantile  regression  with ReLU  Networks: Estimators and minimax rates}

\author[1]{Oscar Hernan Madrid Padilla}
\author[2]{Wesley Tansey}

\author[3]{Yanzhen Chen}

\affil[1]{\small Department of Statistics, University of California, Los Angeles}
\affil[2]{\small Department of Epidemiology and Biostatistics, Memorial Sloan Kettering}
\affil[3]{\small Department of ISOM, Hong Kong University of Science and Technology }

\begin{document}
	\maketitle
	
	\begin{abstract}
		
		Quantile regression is the task of estimating a specified percentile response, such as the median (\pct[50] percentile), from a collection of known covariates.
		We study quantile regression with rectified linear unit (ReLU) neural networks as the chosen model class. We derive an upper bound on the expected mean squared error of a ReLU network used to estimate any quantile conditioning on a set of covariates. This upper bound only depends on the best possible approximation error, the number of layers in the network, and the number of nodes per layer.
		We further show  upper bounds that are tight for two large classes of functions: compositions of \holder functions and members of a Besov space. These tight bounds imply ReLU networks with  quantile regression achieve minimax rates for  broad collections of function types.
		Unlike existing  work, the theoretical results hold under minimal assumptions and apply to general error distributions, including heavy-tailed distributions. Empirical simulations on a suite of synthetic response functions demonstrate the theoretical results translate to practical implementations of ReLU networks.
		Overall, the theoretical and empirical results provide insight into the strong performance of ReLU neural networks for quantile regression across a broad range of  function classes and error distributions. All code for this paper is publicly available at \url{https://github.com/tansey/quantile-regression}.
		

	\end{abstract}
	
	
	\section{Introduction}
	\label{sec:intro}
	
	
	The standard task in regression is to predict the mean response of some variable $Y$, conditioned on a set of known covariates $X$. Typically, this is done by minimizing the mean squared error,
	\begin{equation}
		\label{eqn:mse}
		\hat{f}^{\textrm{(mse)}} \,=\, \underset{ f \in \mathcal{F}}{\arg \min}\, \frac{1}{n}\sum_{i=1}^{n } (y_i - f(x_i))^2, \, 
	\end{equation}
	where $\mathcal{F}$ is a function class.
	
	In many scenarios, this may not be the desired estimand. For instance, if the data contain outliers or the noise distribution of $Y$ is heavy-tailed, \cref{eqn:mse} will be an unstable. The median is then often a more prudent quantity to estimate, even under a squared error risk metric. Alternatively, fields such as quantitative finance and precision medicine are often concerned with extremal risk as well as expected risk. In these domains, one may wish to estimate tail events such as the \pct[5] or \pct[95] percentile outcome. Estimating the \pct[5], \pct[50] (median), \pct[95], or any other response percentile conditional on covariates is the task of quantile regression.
	
	
	
	The goal of quantile regression is to estimate a \textit{quantile function}. Formally, given independent measurements  $\{(x_i,y_i)\}_{i=1}^n \subset \mathbb{ R}^d  \times  \mathbb{ R}$,  from the  random vector  $(X,Y) \in \mathbb{ R}^d \times  \mathbb{ R}$, the goal is to estimate $f_{\tau}^* \,:\,  \mathbb{ R}^d   \rightarrow \mathbb{ R}$ given as 
	\[
	f_{\tau}^*(x)  \, =\,  F_{Y|X=x}^{-1}(\tau),\,\,\,\,\,  x\in \mathbb{ R}^d,
	\]
	where  $\tau \in (0,1)$ is a quantile level,  $ F_{Y|X=x}$  is   the distribution of  $Y$ conditioned on $X=x$, and $f_{\tau}^*(\cdot)$ is the quantile function for the \pct[\tau] quantile.   For example, when  $\tau =0.5 $ the function  $	f_{\tau}^*(x)  $  becomes the conditional median  of  $Y$ given  $X=x$. More generally, the quantile level $\tau$ corresponds to the $(\tau \times 100)^{\text{th}}$ percentile response.
	
	As an estimator   for  $f_{\tau}^*$,  we consider     $\hat{f}$  of the form
	\begin{equation}
		\label{eqn:estimator1}\hat{f} \,=\, \underset{  f \in  \mathcal{F} }{\arg \min}\,   \sum_{i=1}^{n } \rho_{\tau}(y_i - f(x_i)),
	\end{equation}
	where $\mathcal{F}$ is a class of neural network  models, and  $\rho_{\tau}(x) = \max\{\tau x,(\tau-1)x\}$ is the quantile loss  function as in \cite{koenker1978regression}.
	Neural network models optimizing \cref{eqn:estimator1} have been proposed in previous contexts (see Section \ref{sec:previous_work}). These models have shown strong empirical performance, but a theoretical understanding of neural quantile regression remains absent. In this paper, we lay the groundwork for the theoretical foundations of quantile regression with neural networks. Below we briefly summarize our contributions.

	\subsection{Summary of results}
	
	We establish statistical guarantees for quantile regression with multilayer neural networks built with the rectified linear unit  (ReLU) activation function ($\phi(x) =\max\{x,0\}$). Specifically, we make the following contributions:
	
	
	
	\begin{itemize}
		\item For the class of  ReLU neural networks $\mathcal{F}$ with $W$ parameters,    $U $  nodes,   and  $L$ layers, we provide an upper bound on the expected value of the mean squared error for estimating the quantile function $f_{\tau}^*(\cdot)$ at the design points  $x_1,\ldots,x_n$. The upper bound requires  no assumptions  about  the  function $f_{\tau}^*(\cdot)$ though it depends on the best performance possible, under the quantile loss, for functions in the class $\mathcal{F}$.

		\item Suppose that $f_{\tau}^*(\cdot)$ can be written as the composition of functions whose coordinates are  H\"{o}lder  functions  (see  \cite{schmidt2017nonparametric}). We show that there exists a sparse ReLU neural network class $\mathcal{F}$  such that the corresponding quantile regression estimator  attains minimax rates, under squared error loss, for estimating $f_{\tau}^*(\cdot)$.
		This result  holds under minimal assumptions on the distribution of $(X,Y)$.  As a result,  quantile regression with ReLU networks can be directly applied to models with  heavy-tailed error distributions.		
		
		\item Suppose that $f_{\tau}^*(\cdot)$ belongs to a Besov space  $B^{s}_{p,q}([0,1]^d)$,  where  $0<p,q\leq \infty$, and  $s>  d/p$.  We show that under mild conditions, there exists a ReLU neural network structure $\mathcal{F}$  such that  quantile  regression constrained to $\mathcal{F}$   attains  the  rate $n^{- \frac{2s}{2s+d} }$  under the squared error loss. The  resulting rate is minimax in balls of the space  $B^{s}_{p,q}([0,1]^d)$. Thus,  our work advances the nonparametric regression  results from \cite{suzuki2018adaptivity}  to the quantile regression setting where the distribution of the errors can be arbitrary distributions.
		
		

	\end{itemize}
	
	\subsection{Previous work}
	\label{sec:previous_work}
	
	This paper lies at the intersection of nonparametric function estimation theory and quantile regression. The theory we develop draws on two well-established, though mostly-independent, lines of research in estimation theory: (i) universality and convergence rates for neural networks, and (ii) minimax rates for estimating functions in a Besov space and a space based on compositions of H\"{o}lder functions. We merge and extend results from these fields to analyze neural quantile regression. This provides a theoretical foundation for a number of proposed neural quantile methods with strong empirical performance but no prior theoretical motivation. We briefly outline relevant work in each of these areas and situate this paper within these different lines of work.
	
	Neural networks have been shown to have attractive theoretical properties in many scenarios. \cite{hornik1989multilayer} showed that regardless of the activation function, single-layer feedforward networks  can approximate any measurable function; more thorough  descriptions of approximation theory for neural networks are given in \cite{	white1989learning,barron1993universal,barron1994approximation,hornik1994degree,anthony2009neural}.  In the statistical theory literature, \cite{mccaffrey1994convergence} proved  convergence rates for single-layer feedforward networks. \cite{kohler2005adaptive} proved convergence rates for estimating a regression function with a shallow network using sigmoid activation functions; \cite{hamers2006nonasymptotic}  developed  risk bounds in a similar framework. \cite{ klusowski2016risk} developed  risk bounds for high-dimensional ridge function combinations that include neural networks.  	\cite{klusowski2016uniform}  studied uniform approximations by neural network models. 

	More recently, an emerging line of research explores approximation theory for ReLU  neural networks.
	These works are motivated by the empirical successes of ReLU networks, which often outperform neural networks with other activation functions \citep[e.g.][]{nair2010rectified,glorot2011deep} and have achieved state-of-the-art performance in a number of domains \citep[e.g.][]{krizhevsky2012imagenet,devlin:naacl19}.
	\cite{yarotsky2017error}  provided approximation results for Sobolev spaces, which were exploited by  \cite{farrell2018deep} for semiparametric  inference in causality related problems. 
	\cite{liang2016deep} and \cite{petersen2018optimal} provided approximation results for    piecewise smooth functions.  Additional approximation results were also established in \cite{schmidt2017nonparametric}  for classes of functions constructed from H\"{o}lder functions. For such classes, the corresponding approximation results  were exploited by \cite{schmidt2017nonparametric} to obtain minimax rates for nonparametric regression with ReLU networks.  More recently,  
	\cite{nakada2020adaptive} proved minimax rates for nonparametric estimation with ReLU networks in settings with intrinsict low dimension of the data. \cite{bauer2019deep}  studied theoretical properties of nonparametric regression  with neural networks with sigmoid activation function.
	
	
	Separate from neural network theory, related work has investigated regression when the true function is in a Besov space. Such classes of functions are widely used in statistical modeling due to  their  ability to capture spatial inhomogeneity in smoothness. In addition, Besov spaces include more traditional smoothness spaces such as Hölder and Sobolev spaces.   Some statistical works involving Besov spaces include \cite{donoho1998minimax} and  \cite{suzuki2018adaptivity} which provided  minimax results on Besov spaces in the context of regression based on wavelet and neural network estimators respectively.  \cite{brown2008robust}  studied the one dimensional median regression setting when the median function belongs to a Besov space. \cite{uppal2019nonparametric} considered the context of density estimation and  convergence of generative adversarial networks.   A more mathematically generic treatment of Besov spaces can be found in  \cite{devore1988interpolation}  and \cite{lindenstrauss2013classical}.
	
	
	In a line of empirical work, quantile regression with neural networks has been shown to be a powerful nonparametric tool for modeling complex data sets. Successful applications of neural quantile regression include precipitation downscaling and wind power \citep{cannon2011quantile,hatalis2017smooth}, credit portfolio analysis \citep{feng2010robust}, value at risk \citep{xu2016quantile}, financial returns \citep{taylor2000quantile,zhang2019extending},  electrical industry forecasts \citep{zhang2018improved}, and transportation problems \citep{rodrigues2020beyond}.

	On the theoretical side of quantile regression with neural networks,  \cite{white1992nonparametric}  proved convergence in probability results for shallow networks. \cite{chen1999improved}  developed  theory for estimation with a general loss and with single hidden layer  neural network architectures based on a smooth activation function. For target functions in the  Barron class \citep{barron1993universal,barron1994approximation,hornik1994degree},    \cite{chen1999improved}   proved convergence rates better than $n^{-1/4} $  rate in root-mean-square error metric for time series nonparametric quantile regression. Similarly,  Example 3.2.2 in  \cite{chen2007large}  also  established a faster than $n^{-1/4}$ rate in root-mean-square error metric for nonparametric quantile regression in the Sobolev space $W_1^1([0,1]^d)$ ($\ell_1$-integrable functions with domain $[0,1]^d$  and $\ell_1$-integrable first order partial derivatives).  In a related work, \cite{chen2020efficient} considered quantile treatment effect estimation.  Despites all these notable efforts, the results for quantile regression with neural networks are not known to be minimax optimal. We fill this gap by  considering  quantile regression with deep ReLU neural network architectures,  showing minimax rates for general classes of functions.


	


	

	
	
	
	
	


	\section{Neural quantile regression with ReLU networks}
	
	\subsection{Univariate response quantile regression}
	
	For a  vector  $v \in \mathbb{R}^r$
	we define  the function  $\phi_v \,:\,\mathbb{R}^r \rightarrow \mathbb{R}^r$ as 
	\[
	\phi_v\left(   \begin{array}{l}
		a_1\\
		\vdots\\
		a_r
	\end{array}    \right)   = \left(   \begin{array}{l}
		\phi(a_1 -v_1)\\
		\,\,\,\,\,\,\,\,\,\,\,\,\,\, \vdots\\
		\phi(a_r -v_r)
	\end{array}    \right) , 
	\]
	where  $\phi \,:\,\mathbb{ R} \rightarrow  \mathbb{ R}$   given as  $\phi(x) =\max\{x,0\}$  is the ReLU activation function. By convention, when $v =0$ we write $\phi$ to denote  $\phi_v$.
	With this notation, we consider   neural network functions  $f  \,:\, \mathbb{R}^{p_0}  \rightarrow   \mathbb{R}^{p_{L+1}} $ of  the form
	\begin{equation}
		\label{eqn:form1}
		f(x) =   A^{(L)}   \phi_{V_L}  \circ   A^{(L-1)}    \phi_{   V_{L-1}}  \circ   \cdots   \circ  A^{(1)}   \phi_{V_1} \circ A^{(0)}    x, 
	\end{equation}
	where $\circ $ denotes the composition of functions, and  $A^{(i)}  \in  \mathbb{R}^{   p_{i+1}  \times p_{i}   }$, $V_i  \in \mathbb{R}^{p_i}$,  $p_0,\ldots,p_{L+1} \in \mathbb{ N}$ for  $i\in\{0,1,\ldots,L+1\}$.  Here  the matrices $\{A^{(i)}\}$ are the weights in the network, $L$ is the number of layers, and  $(p_0,\ldots,p_{L+1})^{\top}\in \mathbb{ R}^{L+2}$ the width vector.  In this section we assume that $p_{L+1}=1$.
	
	

	Since we focus on quantile regression restricted to neural networks  with ReLU activation functions, we briefly review how joint estimation of quantiles can be achieved. Specifically,  if multiple  quantile levels  are given in a set $\Lambda  \subset (0,1)$, then it is natural to estimate  the  quantile functions $\{ f_{\tau}^*(\cdot)\}_{\tau \in   \Lambda}$  by solving the  problem 
	\begin{equation}
		\label{eqn:multiple_qauntiles}
		\{	\hat{f}_{\tau}\}_{\tau \in \Lambda}    =   \begin{array}{ll}
			\underset{  \{f_{\tau}\}_{\tau \in \Lambda}  \subset  \mathcal{F}  }{\arg \min}& \displaystyle  \,\, \sum_{\tau \in \Lambda}  \sum_{i=1}^{n}\rho_{\tau}(y_i- f_{\tau}(x_i))\\
			\text{subject to} &  f_{\tau}(x_i)\leq  f_{\tau^{\prime}}(x_i) \,\,\,\,\forall  \tau <\tau^{\prime}, \,\,\,\tau, \tau^{\prime} \in \Lambda,\,\,\,i=1,\ldots,n.
		\end{array}
	\end{equation}
	The  constraints in (\ref{eqn:multiple_qauntiles}) are noncrossing  restrictions that are meant to ensure the monotonicity of quantiles. However, due to the nature of stochastic subgradient descent, the monotonicity constraints in (\ref{eqn:multiple_qauntiles}) can make finding a solution to this problem  challenging. To address this, letting $\tau_0 < \ldots < \tau_m$ be the elements of $\Lambda$, we solve 
	\begin{equation}
		\label{eqn:multiple_qauntiles2}
		\{	\hat{h}_{\tau}\}_{\tau \in \Lambda}    =   \begin{array}{ll}
			\underset{ \{h_{\tau}\}_{\tau \in \Lambda}  \subset  \mathcal{F}  }{\arg \min}& \displaystyle  \,\,      \sum_{i=1}^{n}\rho_{\tau}(y_i- h_{\tau_0}(x_i))     +    \sum_{j=1}^m  \sum_{i=1}^{n}\rho_{\tau}\left\{y_i-  h_{\tau_0}(x_i)-   \sum_{l=1}^j  \log\left( 1+ e^{  h_{  \tau_l }(x_i)    } \right)        \right\}\\
		\end{array}
	\end{equation}
	and set
	\[
	\hat{f}_{\tau_0}(x) =\hat{h}_{\tau_0}(x), \,\,\,\,\text{and} \,\,\,\,    \hat{f}_{\tau_j}(x)      =     \hat{h}_{\tau_0}(x)+  \sum_{l=1}^j  \log\left( 1+ e^{  \hat{h}_{  \tau_l }(x)    } \right)     \,\,\,\,\text{for}\,\,\,\,\,j=1,\ldots,m.
	\]
	By construction, (\ref{eqn:multiple_qauntiles2})  implies that the quantile functions  $\{	\hat{f}_{\tau}\}_{\tau \in \Lambda}   $  satisfy the monotonicity constraint in (\ref{eqn:multiple_qauntiles}).  We find  this approach to be numerically stable as compared to other choices such as replacing the terms  $\log(1+e^{  \hat{h}_{  \tau_l }(x_i)    }  )$  with $e^{  \hat{h}_{  \tau_l }(x_i)    } $.
	A different alternative is to estimate the quantile functions separately and then to order their output  as in \cite{chernozhukov2010quantile} and  \cite{zhang2019extending}. In this paper, we will focus on solving (\ref{eqn:multiple_qauntiles2}) which we find to be better in practice.

	\subsection{Extension to multivariate response}
	\label{sub_sec:multivariate}
	The framework that we have considered so far restricts the outcome variable  to be univariate. However, in many machine learning problems where  neural networks are used the outcome is  multivariate. In this section  we discuss two simple extensions of the quantile loss to the multivariate response setting. Our experiments section will contain empirical evaluations of the proposals here.

	\subsubsection{Geometric quantiles}
	
	We start by considering  geometric quantiles. These were introduced by \cite{chaudhuri1996geometric}   to generalize quantiles to multivariate settings. Specifically,   suppose  that 	we are given  data  $\{(x_i,y_i)\}_{i=1}^n   \subset \mathbb{R}^d \times \mathbb{R}^p$, with $p>1$. Furthermore, consider the Euclidean unit ball $\mathbb{ R}^p$, namely  $B^{(p)} = \{  u   \in  \mathbb{ R}^p\,:\,  \|u\|\leq 1 \}$.  \cite{chaudhuri1996geometric}  defines the function   $ \Psi(  \cdot,\cdot ) \,:\,  \mathbb{ R}^p\times \mathbb{ R}^p  \,\rightarrow  \mathbb{ R}$, 
	\[
	\Psi(u,x) \,=\,   \|x\|  +   x^{\top} u,
	\]
	and proposes to minimize  the empirical risk associated with this loss. Motivated by the geometric quantile framework, we define the geometric quantile based on $u \in B^{(p)}$ and  a  ReLU nerwork  class  $\mathcal{F}      \subset  \,\{  f\,:\,   \,\,\,\,\,f\,:\,\mathbb{ R}^d \rightarrow \mathbb{ R}^p    \}$, as 
	\[
	\hat{f}_u   \,=\,      \underset{f  \in  \mathcal{F}  }{\arg \min}\,\,   \sum_{i=1}^{n}\Psi(u,y_i-   f(x_i)   ) .
	\]
	Notice that when $u=0$,  $\hat{f}_u$
	becomes
	\begin{equation}
		\label{eqn:robust}
		\hat{f}_u   \,=\,      \underset{f  \in  \mathcal{F}  }{\arg \min}\,\,   \sum_{i=1}^{n}    \| y_i-   f(x_i)  \|.
	\end{equation}
	The latter  can be thought as an estimator  of the mean of $y_i$ conditioning on $x_i$. In fact, (\ref{eqn:robust})   is commonly known as the $L_1$-median, see \cite{vardi2000multivariate}. The $L_1$-median can  be interpreted as a robust version of the usual   least squares,
	\begin{equation}
		\label{eqn:ls}
		\underset{f  \in  \mathcal{F}  }{\arg \min}\,\,   \sum_{i=1}^{n}    \| y_i-   f(x_i)  \|^2.
	\end{equation}
	This is due to the fact that  replacing $\|\cdot\|^2$ with $\|\cdot\|$, as in (\ref{eqn:robust}), has the advantage that large residuals are not heavily penalized as in  (\ref{eqn:ls}).

	\subsubsection{Marginal quantiles}
	\label{sec:marginal_quantiles}

	Marginals quantile have perhaps the advantage over geometric quantiles in that they can produce  actual prediction intervals, and have probabilistic meaning. However, as their name suggests,  marginal quantiles  only produce  predication intervals  for each variable in the output marginally, and thus do not produce a prediction region for the output jointly.

	Let  $\tau \in (0,1)$,    and $ f_{\tau}^* \,:\,  \mathbb{ R}^d  \,\rightarrow \,\mathbb{ R}^p$,  $ f_{\tau}^*(x) = ( f_{\tau,1}^*(x) ,\ldots,f_{\tau,p}^*(x)   )^{\top} $, where
	\begin{equation}
		\label{eqn:marginal}
		f_{\tau,j }^*(x) ,\,=\, F_{Y_j|X=x}^{-1}(\tau),
	\end{equation}
	where  $Y =(Y_1,\ldots,Y_p)^{\top} \in \mathbb{ R}^p$. The functions   $f_{\tau,1}^*(x) ,\ldots,f_{\tau,p}^*(x)$ are the marginal quantiles of $Y_1,\ldots,Y_p$ respectively, conditioining  on $X$.
	Marginal quantiles have been studied in the literature \citep[c.f.][]{babu1989joint,abdous1992note}.  Given $\{  (x_i,y_i)\}_{i=1}^n \subset \mathbb{R}^d \times \mathbb{ R}^p$   independent copies of $(X,Y)$,
	the multivariate  function $f_{\tau}^*$  can be estimated with a multivariate output ReLU neural network architecture  $\mathcal{F}$  as
	\begin{equation}
		\label{eqn:qunatiles2}
		(\hat{f}_{\tau,1},\cdots,\hat{f}_{\tau,p})^{\top} \,=\,      \underset{f  =  (f_1,\ldots,f_p)^{\top}    \in  \mathcal{F}  }{\arg \min}\,\,     \sum_{j=1}^{p}   \sum_{i=1}^{n}   \rho_{\tau}(y_{i,j}-   f_j(x_{i})).
	\end{equation}
	To be specific, here  the class $\mathcal{F}$ consists of functions of the form (\ref{eqn:form1})  with $p_0 = d$ and  $p_{L+1} = p$.

	\section{Theory}
	
	We  now proceed to provide statistical guarantees for  quantile regression  with ReLU networks. Our theory is organized in three parts. First, we provide a general upper bound  on the mean squared error for estimating the quantile function. Second, we study a setting where the quantile function is a member of  a space of compositions of functions whose  coordinates are H\"{o}lder functions. Finally, we assume  that the quantile  function belongs to a  Besov space.

	\subsection{Notation}
	Throughout this section, for functions  $f,g\,:\,\mathbb{ R}^d  \rightarrow \mathbb{ R}$,     we define the function  $\Delta_n^2(f,g)$ as
	\begin{equation}
		\label{eqn:loss}
		\Delta_n^2(f,g)   \,:=\,  \frac{1}{n} \sum_{i=1}^{n}   D^2(  f(x_i) - g(x_i)  ),
	\end{equation}
	with $\{x_i\}_{i=1}^n$ the features and where
	\begin{equation}
		\label{eqn:loss2}
		D^2(t) := \min\left\{  \vert t \vert,  t^2 \right\}.
	\end{equation}
	This function was used as performance metric in a different quantile regression context in  \cite{padilla2020adaptive}.

	Furthermore,  		for bounded  functions  $f$ and  $g$  with  $f,g  \,:\,[0,1]^d\rightarrow \mathbb{ R}$,  we define $	\Delta^2(f,g) $ as
	\[
	\Delta^2(f,g) :=  \mathbb{E}\left(   D^2(f(X)-g(X))   \right),
	\]
	set  $\Delta(f,g) :=  \sqrt{\Delta^2(f,g)}$.  We also write 
	\[
	\|f-g\|_{\ell_2} :=  \sqrt{\mathbb{E} \left(  \left(   f(X) - g(X) \right)^2   \right)  },
	\]
	and
	\begin{equation}
		\label{eqn:mse_def}
		\displaystyle 	\|  f-g  \|_n^2  \,:=\,  \frac{1}{n}\sum_{ i=1}^n   (  f(x_i) -g(x_i)  )^2.
	\end{equation}
	
	For a matrix  $A \in \mathbb{ R}^{s \times t}$  we  define  
	\[
	\|A\|_0 \,=\, \vert \{  (i,j)\,:\,     A_{i,j} \neq 0, \,i\in\{1,\ldots,s\}, \, j\in\{1,\ldots,t\} \}\vert,\,\,\,\,\,\,  \|A\|_{\infty} \,=\,  \underset{i=1,\ldots,s,\,j=1,\ldots,t}{\max}\,\,\vert A_{i,j}\vert.
	\]
	We also  write $\floor{x}$  for the largest integer strictly smaller than $x$. The notation $\mathbb{N}_{+}$ and  $\mathbb{R}_{+}$ indicate the set of positive natural and real numbers respectively. For sequences  $a_n$  and  $b_n$ we write  $a_n= O(b_n)$ and  $a_n \lesssim b_n$ if there exist  $C>0$ and  $N>0$ such that $n \geq N$ implies  $a_n \,\leq  \,C b_n$. If $a_n  = O(b_n)$ and  $b_n = O(a_n)$ then we  write  $a_n \asymp b_n$.
	
	Finally,  we refer to the quantities $\varepsilon_i = y_i -   f_{\tau}^*(x_i)$ for  $i=1,\ldots,n$  as the errors.

	\subsection{General  upper bound}

	In this subsection we  focus on quantile regression ReLU estimators of the form  
	\begin{equation}
		\label{eqn:deep}
		\displaystyle   	\hat{f} \,= \, \underset{ f  \in  \mathcal{F}(W,U,L) ,    \,\,\,\|f\|_{\infty} \leq F }{\arg \min} \, \sum_{i=1}^{n}  \rho_{\tau}(y_i -  f(x_i) ),
	\end{equation}
	where   $\mathcal{F}(W,U,L)$  is the class of networks of the form (\ref{eqn:form1}) such that  the number of parameters in the network  is  $W$,  the number of nodes  is $U$, and the number of layers is $L$. Here, $F$ is  a fixed positive constant.

	Before  arriving at our first  result, we start by stating some assumptions regarding the generative model. Throughout, we  consider  $\tau \in (0,1) $  as fixed.

	\begin{assumption}
		\label{as1}  
		We write $  f_{\tau}^*(x_i) = F_{y_i |x_i}^{-1}(\tau) $ for $i=1,\ldots,n$.  Here  $F_{y_i |x_i} $  is cumulative distribution function of $y_i$ conditioning on $x_i$  for  $i=1,\ldots,n$. Also,   $y_1,\ldots,y_n \in \mathbb{ R}$ are assumed to be  independent.
	\end{assumption}

	Notice that Assumption \ref{as1}  simply requires that the different outcome  measurements  are independent conditioning on the design,  which for this subsection is assumed to be fixed.
	
	\begin{assumption} \label{as2} There exists  a constant  $L>0$  such that   for   $\delta \in \mathbb{ R}^n$  satisfying  $\|\delta\|_{\infty} \leq L$  we have that \[ \underset{i = 1,\ldots,n}{\min}\,\,p_{y_i|x_i}(  f_{\tau}^*(x_i)  +\delta_i ) \geq \underline{p}, \] for some  $\underline{p}>0$, and where   $p_{y_i|x_i}$  is the probability density function  of $y_i$ conditioning on $x_i$.  We also require that 
		\[
		\underset{t  \in \mathbb{ R}}{\sup}\,\,  p_{y_i|x_i}(t)  \,\leq \,  c,\,\,\,\text{a.s.},\,\,\,
		\]
		for some constant  $c>0$.
	\end{assumption}
	
	Assumption \ref{as2}  requires that  there exists a neighborhood around $f_{\tau}^*(x_i)$ in which   the probability density function of $y_i$ conditioning on $x_i$ is bounded by below. Related conditions appeared as D.1 \cite{belloni20111}, 
	Condition  2 in \cite{he1994convergence}, and Assumption A  in  \cite{padilla2020adaptive}.

	Next we define   $f_n$, the projection of the quantile function $f_{\tau}^*$ onto the network  class $\mathcal{F}(W,U,L)$ in the sense of the quantile risk.

	\begin{definition}
		\label{def1}
		We define the function $f_n$ as
		\[
		f_n  \,\in \,  \underset{f \in   \mathcal{F}(W,U,L),   \|f\|_{\infty} \leq  F }{\arg \min}\,\,\,    \mathbb{E}\left[ \sum_{i=1}^{n}   \rho_{\tau}(z_i -  f(x_i)    )     -  \sum_{i=1}^{n}  \rho_{\tau}(z_i -  f_{\tau}^*(x_i)    )  \right],
		\]
		where  $z \in \mathbb{ R}^n$ is an independent  copy of  $y$. We also define the approximation error as
		\[
		\text{err}_1 =  \mathbb{E}\left[ \frac{1}{n}\sum_{i=1}^{n}   \rho_{\tau}(z_i -  f_n(x_i)    )     -  \frac{1}{n} \sum_{i=1}^{n}  \rho_{\tau}(z_i -  f_{\tau}^*(x_i)  )  \right].
		\]
	\end{definition}

	Notice that when $f_{\tau}^* \in \mathcal{F}(W,U,L)$ and  $\|f\|_{\infty}^*\leq F$  then the approximation error  is zero. However, in general $f_{\tau}^* \notin\mathcal{F}(W,U,L)$ and  $\text{err}_1   \geq 0$.
	
	We are now ready to present our first theorem which exploits the VC dimension results from \cite{bartlett2019nearly}.

	\begin{theorem}
		\label{thm:risk}
		Suppose  that Assumptions \ref{as1}--\ref{as2} hold and   $n\geq C LW \log (U)$ for a large enough $C>0$. Then  $\hat{f}$ defined in  (\ref{eqn:deep}) satisfies
		\[
		\mathbb{E}\left[    \Delta_n^2(  f_{\tau}^*,\hat{f} ) \,\bigg| \,   x_1,\ldots,x_n\right] \leq      c_1 F \left[  \frac{   \left\{LW \log U\cdot \log n   \right\}   }{n} \right]^{1/2} +   c_1\text{err}_1, 
		\]
		with  $c_1>0$  a constant. Furthermore, it also holds that
		\[
		\mathbb{E}\left[  \|\hat{f}-f_{\tau}^*\|_n^2 \,\bigg| \,   x_1,\ldots,x_n\right] \leq      c_1    \max\{1,F\} F \left[  \frac{   \left\{LW \log U\cdot \log n    \right\}  }{n} \right]^{1/2} +   c_1\max\{1,F\}\text{err}_1.
		\]
	\end{theorem}

	Theorem \ref{thm:risk}  provides a general  bound on the mean squared error that depends on the sample size $n$, the parameters of the network, and the approximation error. For instance, if  $L$, $W$ and $U$ are constants in $n$, then the rate becomes  $n^{-1/2  }  +   \text{err}_1 $.  In the next two subsections we will consider  classes of ReLU network with more structure which will lead to rates that match minimax rates in nonparametric regression.

	\subsection{Space of  compositions based on  H\"{o}lder functions  }
	\label{sec:theory2}

	Next we provide  convergence rates for quantile regression with ReLU networks under the assumption that  the quantile function belongs to a class of functions based on H\"{o}lder spaces. Such class  of  functions,  defined below,  was studied in \cite{schmidt2017nonparametric}.   There,  the authors showed that  for such class, neural networks with ReLU  activation function attain minimax rates. However, the results in \cite{schmidt2017nonparametric}  hold under the assumption of Gaussian errors. We now show that it is possible to attain the same  rates under general error assumptions  by employing the quantile loss. Before arriving at such result we start by providing some definitions.

	\begin{definition}
		\label{def8}
		We define the class of ReLU neural networks  $\mathcal{G}(L,p,s,F)$ as
		\[
		\begin{array}{lll}
			\mathcal{G}(L,p,S,F)   &\,=\,&\bigg\{     f\,:\,    f  \,\,\text{is of form } \,\,\,(\ref{eqn:form1}),   \,\,\,\text{and}\,\,\,\,\, \sum_{j=0}^L   \left( \|A^{(l)}\|_0  + \|V_l\|_0    \right) \leq S, \,\, \| f\|_{\infty} \leq F,\, \\
			& &   \underset{j=0,1,\ldots,L}{\max}\,  \|A^{(j)}\|_{\infty}  \leq 1,\,\,\,\,\, \underset{j=1,\ldots,L}{\max}\,  \|V_j\|_{\infty}  \leq 1 \bigg\}.
		\end{array}
		\]
	\end{definition}

	With the notation in Definition \ref{def8}, we consider the estimator 
	\begin{equation}\label{eqn:estimator}
		\hat{f}  \,=\,      \underset{  f \in \mathcal{G}(L,p,S,F)    ,    \,\,   }{\arg  \min}\,\,   \sum_{i=1}^{n}  \rho_{\tau}( y_i -  f(x_i)  ),
	\end{equation}
	and  define a $\|\cdot\|_{\infty}$-projection of  $f_{\tau}^*$, the true  quantile function, onto  $\mathcal{G}(L,p,S,F)$ as
	\[
	f_n \in      \underset{f \in 	\mathcal{G}(L,p,S,F)  ,  \,    }{\arg\min}  \, \|f - f_{\tau}^*\|_{\infty}.
	\]
	A few comments are in order. First, notice that  we assume that all the parameters are bounded by one. As discussed in \cite{schmidt2017nonparametric}, this is standard and in practice  can be achieved by projecting the parameters in $[-1,1]$ after every iteration of stochastic subgradient descent. Second, we assume that the networks are sparse as   was the case in \cite{schmidt2017nonparametric} and \cite{suzuki2018adaptivity}.  See \cite{hassibi1993second,han2015learning,frankle2018lottery} and \cite{gale2019state} for different approaches to produce sparse networks.

	Before stating our main result of this subsection, we provide the definition of the function class that we consider. Such class requires that we introduce some notation that comes from \cite{schmidt2017nonparametric}.

	\begin{definition}
		\label{def9}
		For  $\beta>0 $    and  $r\in \mathbb{N}_{+}$	we define the  class of  H\"{o}lder functions  of exponent $\beta$  as
		\[
		\mathcal{C}_r^{\beta}(I,K)       =  \bigg\{   f\,:\,I \subset \mathbb{R}^r  \rightarrow \mathbb{R}  \,:\,      \underset{   \alpha  \,:\,  \|\alpha\|_1  <\beta  }{\sum}      \|\partial^{\alpha} f\|_{\infty}  +\underset{\alpha      \,:\, \|\alpha\|_1 = \floor{\beta}      }{\sum}       \underset{x\neq y, x,y\in I }{\sup}    \,\frac{\left\vert     \partial^{\alpha} f(x) - \partial^{\alpha} f(y)       \right\vert  }{\left\|  x-y        \right\|_{\infty}^{    \beta-   \floor{\beta}   }    } \leq K \bigg\},
		\]
		where $\partial^{\alpha} =  \partial^{\alpha_1} \cdots   \partial^{\alpha_r}$ with  $(\alpha_1,\ldots,\alpha_r )   \in \mathbb{ N}^r$.
	\end{definition}

	\begin{definition}
		\label{def10}
		For  $q \in \mathbb{ N}_{+ }$, $d = (d_0,\ldots,d_{q+1} )  \in  \mathbb{N}^{q+2}_{+}$, $   t=(t_0,\ldots,t_q) \in \mathbb{ N}^{q+1}_{+}$,   $\beta  =  (\beta_0,\ldots,\beta_q) \in \mathbb{R}^{q+1}_{+}$  and  $K \in \mathbb{R}_{+}$ we define  the class of  functions
		\[
		\begin{array}{lll}
			\mathcal{H}(q,d,t,\beta,K)&\,=\,&\bigg \{      f =     g_q \circ  \cdots  g_0\,\,\,:\,\,\,  g_i =(g_{i,j})_j\,:\,[a_i,b_i]^{d_i    }   \rightarrow       [a_{i+1},b_{i+1}]^{d_{i+1} } \\
			& &  \,\,g_{i,j}  \in \mathcal{C}_{t_i}^{\beta_i}\left([a_i,b_i]^{t_i    } ,K\right) ,\,\,\,\,\,\,\,\,\,\,\,\,\,   \text{and}\,\,\,\,\,\,\,\,\,\,\,\,|a_i|,|b_i|\leq K  \bigg\}.
		\end{array}
		\]
	\end{definition}

	With  Definitions  \ref{def9}--\ref{def10} in hand, we  now state an assumption on the  true  quantile  function regarding its smoothness.

	\begin{assumption}
		\label{as5} The quantile function $f_{\tau}^*$ satisfies   $f_{\tau}^* \in  \mathcal{H}(q,d,t,\beta,K)$ for some  $q \in \mathbb{ N}_{+ }$,  $d = (d_0,\ldots,d_{q+1} )\in  \mathbb{N}^{q+2}_{+},   t=(t_1,\ldots,t_q) \in \mathbb{ N}^{q}_{+}$,   $\beta  =  (\beta_0,\ldots,\beta_{q}) \in \mathbb{R}^{q+1}_{+}$  and  $K \in \mathbb{R}_{+}$.   We also require that $\|f_{\tau}^*\|_{\infty}\leq F$ where  $F$ is the same apearing in (\ref{eqn:estimator}). Moreover, we define  the  smoothness indices
		\[
		\beta^*_i   = \beta_i   \prod_{l=i+1}^{q } \min\{\beta_l,1\},
		\]
		for $i = 1,\ldots, q-1$ and $\beta^*_q  = \beta_q$.
	\end{assumption}
	
	Importantly, as Section 5 of  \cite{schmidt2017nonparametric} showed,  the class $\mathcal{H}(q,d,t,\beta,K)$ is challenging enough so that wavelet estimators are suboptimal for estimating $f_{\tau}^* \in  \mathcal{H}(q,d,t,\beta,K)$. Our main result in this section shows that, in contrast,  quantile regression with ReLU networks attains optimal rates.
	
	As for the  distribution of the data,  our next assumption requires  that  the covariates have  a probability density function  that is bounded  by above and below.

	\begin{assumption}
		\label{as3}
		We assume that  $\{(x_i,y_i)\}_{i=1}^n$  are independent copies of  $(X,Y)$, with $X$  having   a probability  density function  $g_X$  with support in $[0,1]^d$  and such that
		\[
		c_1\leq    \underset{x \in [0,1]^d}{\inf} g_X(x)    \leq     \underset{x \in [0,1]^d}{\sup} g_X(x)  \leq  c_2,
		\]
		for  some constants $c_1,c_2>0$.
	\end{assumption}

	We are now ready to state  our main result  of this subsection  that exploits the approximation results from \cite{schmidt2017nonparametric}.

	\begin{theorem}
		\label{thm5}
		Suppose  that Assumptions  \ref{as1}--\ref{as3} hold. In addition,  suppose that  for the class $	\mathcal{G}(L,p,S,F) $  the parameters are chosen to satisfy 
		\[
		\begin{array}{l}
			\sum_{i=0}^{q} \log_2\left(   4\max\{t_i,\beta_i\}  \right)\log_2(n)\leq  L     \lesssim  n \epsilon_n,\,\,\,\,\,\,\, \max\{1,K\}\leq F,\\
			n\epsilon_n      \lesssim      \underset{i=1,\ldots,L}{\min} p_i,\,\,\,\,\,\, S\asymp n \epsilon_n \log n,\,\,\,\,\,\,    \underset{i=1,\ldots,L}{\max} p_i  \lesssim n,
		\end{array}
		\]
		where
		\begin{equation}
			\label{eqn:rate}	\epsilon_n \,=\,  \underset{i=0,1,\ldots,q}{\max}   n^{   -  \frac{2\beta_i^*}{2\beta_i^*    +t_i  }  }.
		\end{equation}
		Then there exists a constant  $C>0$  such that with probability approaching  one, we have that 
		\[
		\max\left\{		\|\hat{f}- f_{\tau}^*\|_{ \ell_2  }^2 ,		\|\hat{f}- f_{\tau}^*\|_{n }^2  \right\}   \,\leq \, C  \epsilon_n    L   \log^2 n,
		\]
		where  $\hat{f}$  is the estimator defined in (\ref{eqn:estimator}).
		Hence, if in addition $L \asymp \log n$, then
		\[
		\max\left\{		\|\hat{f}- f_{\tau}^*\|_{ \ell_2  }^2 ,		\|\hat{f}- f_{\tau}^*\|_{n }^2  \right\}   \,\leq \, C  \epsilon_n       \log^3 n,
		\]
		with probability approaching  one.

	\end{theorem}

	Notice that  Theorem   \ref{thm5}  shows that the  ReLU network   based estimator  defined  in    (\ref{eqn:estimator})  attains the rate $\epsilon_n$ under  the mean squared error and  the $\ell_2$  metrics, ignoring $L$ and the log factors, for estimating  quantile functions in the  class $\mathcal{H}(q,d,t,\beta,K)$.   Importantly,    the rate  $\epsilon_n$ is minimax for estimating  functions in the  class $\mathcal{H}(q,d,t,\beta,K)$. Specifically, Theorem \cite{schmidt2017nonparametric} showed that if  
	$t_j\leq  \min\{d_0,\ldots,d_{j-1}\}$ for all $j$ then for a constant $c>0$ we have that
	\[
	\underset{\hat{f}}{\inf}\,\,\underset{  f _{0.5}^*  \in  \mathcal{H}(q,d,t,\beta,K)}{\sup}\,\,	\|\hat{f}- f_{0.5}^*\|_{ \ell_2  }^2  \,\geq \,  c\epsilon_n,
	\]
	where the  infimum is taken over all possible estimators, and with the assumption that the errors  are Gaussian and the covariates are uniformly distributed in $[0,1]^d$. Thus,  Theorem \ref{thm5} provides an upper bound  that  nearly matches the lower bound and that it allows for     heavy-tailed  error  distributions
	
	
	
	\subsection{Besov spaces}

	Next we   study    quantile regression with ReLU networks in the context  of Besov  spaces. Our main result from this subsection  will be similar in spirit to Theorem \ref{thm5}  but  under the assumption  that the  quantile function belongs to a  Besov space. To arrive at our main result, we first   introduce  some notation regarding the ReLU class of networks  that we consider.

	\begin{definition}
		\label{def5}
		For  $W ,L\in \mathbb{N}_{+}$,  $S ,B \in \mathbb{R}$  we define the class of sparse  networks  $\mathcal{I}(L,W,S,B)$   as 
		\[
		\begin{array}{lll}
			\mathcal{I}(L,W,S,B) &\,:=\, &\Bigg\{  (A^{(L)}  \phi(\cdot) + b^{(L)} )  \circ \ldots \circ    (A^{(1)}  x + b^{(1)} )\,\,:\,\,    A^{(l)}  \in \mathbb{R}^{W \times W},  \,\,b^{(l)} \in \mathbb{R}^{W}  \\
			& &     \displaystyle \sum_{l=1}^L     \left(\|A^{(l)}\|_0   +  \|b^{(l)}\|_0\right) \leq S, \,\,\,\underset{l }\max \,\max\{  \|A^{(l)} \|_{\infty} ,   \|b^{(l)} \|_{\infty}   \}   \leq B  \Bigg\}.
		\end{array}
		\]
	\end{definition}

	Notice that the space of networks  $	\mathcal{I}(L,W,S,B)$  is actually   similar to $	\mathcal{G}(L,p,S,F)$. The main difference is that the networks in the former class  have weight matrices of the same  size across the different layers. This minor differences are only necessary in order to achieve the  theoretical guarantees  under the  different classes  to which the quantile   function belongs.

	We the notation from  Definition \ref{def5},  we   focus on the estimator 
	\begin{equation}
		\label{eqn:estimator_2}
		\hat{f}  \,=\,      \underset{  f \in \mathcal{I}(L,W,S,B) ,    \,\,   \|f\|_{\infty}\leq  F   }{\arg  \min}\,\,   \sum_{i=1}^{n}  \rho_{\tau}( y_i -  f(x_i)  ),
	\end{equation}
	where  $F>0$ is fixed.

	Before  providing  a statistical guarantee   for  $\hat{f}$ in (\ref{eqn:estimator_2}),  we first  state the required assumptions  imposed on the generative model.
	
	\begin{assumption}
		\label{as4}
		The quantile function satifies  $f_{\tau}^* \in B_{p,q}^s([0,1]^d)$,  $\|f_{\tau}^*\|_{\infty} \leq F$,   where  for $0<p,q\leq \infty$, and   $0<s<\infty$  we have $s\geq d/p$. 	Furthermore, the exists  $m \in \mathbb{N}$  such that   $0<s<\min\{m ,m-1+1/p\}$.  Here,  $B_{p,q}^s([0,1]^d)$ is  a Besov  space in $[0,1]^d$ as Definition \ref{def4} in the Appendix.
	\end{assumption}

	We are now ready to state the main result concerning estimation of a quantile function that belongs to a Besov space.
	
	\begin{theorem}
		\label{thm3}
		Suppose  that Assumptions  \ref{as1}--\ref{as2}  and  \ref{as3}--\ref{as4}  hold. 
		In addition, suppose that  for the class $\mathcal{I}(L,W,S,B)$  the parameters are chosen as
		\[
		\begin{array}{l}
			L =  3+ 2\ceil{    \log_2\left(    \frac{3^{  \max\{d,m\}  }}{  \epsilon  c_{d,m} }\right)  +5 }\ceil{  \log_2 \max\{d,m\} },\,\,\,\,\,\,\, W = W_0 N,\\
			S = (L-1)W_0^2 N +N,\,\,\,\,\,\, B = O\left( N^{  (v^{-1}+   d^{-1})(  \max\{ 1,   (d/p-s)_{+}\} )   }\right),
		\end{array}
		\] 
		for  a constant $c_{d,m} >0 $  that depends on $d$ and $m$, a constant $W_0 >0$,  and where $v= (s-\delta)/\delta$,
		\[
		\delta = d/p,\,\,\,\,\,\,\,  \epsilon = N^{   -s/d-  (v^{-1}  +  d^{-1}) (d/p-s)_{+}    }  + \{\log N\}^{-1}, \,\,\,\,\,\,\,  N \asymp   n^{  \frac{d}{2s+d} }.
		\]
		Then there exists a constant  $C>0$  such that with probability approaching  one, we have that 
		\[
		\max\left\{		\|\hat{f}- f_{\tau}^*\|_{ \ell_2  }^2 ,			\|\hat{f}- f_{\tau}^*\|_{n  }^2 \right\}   \,\leq \, C   \frac{(\log n)^2}{n^{ \frac{2s}{2s+d}  }    },
		\]
		with  $\hat{f}$ defined as (\ref{eqn:estimator_2}).
	\end{theorem}
	
	Notably,  Theorem \ref{thm3} shows that  the neural network  based   quantile estimator  $\hat{f}$ attains the rate $n^{ -\frac{2s}{2s+d}  } $, ignoring  logarithmic factors,  for estimating the quantile  function. This results  generalizes Theorem  2 from  \cite{suzuki2018adaptivity}  to the quantile regression setting. In particular, Theorem \ref{thm3}  holds under general assumptions of the errors allowing for heavy-tailed distributions. Furthermore,  the rate $n^{ -\frac{2s}{2s+d}  } $ is minimax  for estimation, with Gaussian errors, of the conditional mean  when such function  belongs to a  fixed ball of the space  $B_{p,q}^s([0,1]^d)$, see \cite{donoho1998minimax} and \cite{suzuki2018adaptivity}.


\section{Experiments}
\label{sec:experiments}

We study the performance of ReLU networks for quantile regression across a suite of heavy-tailed synthetic and real-data benchmarks. The benchmarks include both univariate and multivariate responses. For univariate responses, we compare ReLU methods against quantile regression versions of random Forests \citep{meinshausen2006quantile} and splines \citep{koenker1994quantile,he1994convergence}. In the univariate synthetic benchmarks, ReLU networks are shown to outperform random Forests in all of the tested settings; splines outperform ReLU networks only when the true response function is smooth. For multivariate responses, we consider the two different loss functions for multivariate quantiles proposed in \cref{sub_sec:multivariate}. In both univariate and multivariate responses, the ReLU networks with quantile-based losses perform better when estimating the mean than using a squared error loss.

\subsection{Univariate response}
\label{sec:uni_simulations}

We assess the performance of quantile regression with ReLU networks (Quantile Networks) on five different generative models. Each model involves a set of covariates and a univariate response target. The covariates determine the location of the response and a zero-mean, symmetric function with heavy tails is used as the noise distribution; we focus here on Student's t and Laplace distributions.

We compare Quantile Networks with three other nonparametric methods: (i) mean squared error regression with ReLU networks (SqErr Networks), as in Problem \ref{eqn:mse}; (ii) quantile regression with natural splines \citep[Quantile Splines,][]{koenker1994quantile,he1999quantile}; and quantile regression with random Forests \citep[Quantile Forests,][]{meinshausen2006quantile}. For the two neural network methods, we train the models using stochastic gradient descent (SGD) as implemented in PyTorch~\citep{paszke:etal:2019:pytorch} with Nesterov momentum of $0.9$, starting learning rate of $0.1$, and stepwise decay $0.5$. The neural network models also use the same architecture: two hidden layers of $200$ units each, with dropout rate of $0.1$ and batch normalization in each layer. For the other two nonparametric methods, we choose parameters to be flexible enough to capture a large number of nonlinearities while still computationally feasible on a laptop for moderate-sized problems. For Quantile Splines, we use a natural spline basis with $3$ degrees of freedom; we use the implementation available in the \texttt{statsmodels} package.\footnote{\url{https://www.statsmodels.org}} For Regression Forests, we use $100$ tree estimators and a minimum sample count for splits of $10$; these are defaults in the \texttt{scikit-garden} package.\footnote{\url{https://scikit-garden.github.io/}}

We assess the performance of all methods using the mean squared error (MSE) between the estimated and true quantile functions. In each experiment, the methods are estimated at different training sample sizes $n$, $n\in \{100,1000,10000\}$, and different quantile levels $\tau$,  $\tau\in\{0.05,0.25,0.50,0.75,0.95\}$. Since the SqErr Network only estimates the mean, we only evaluate it at $\tau=0.50$, which is equivalent to the mean in all benchmarks. For each benchmark, we generate $25$ datasets independently from the same generative model and evaluate performance using $10000$ sampled covariates with the corresponding true quantile. In each scenario the data are generated following the same location-plus-noise template,
\[
\begin{array}{lll}
y_i  &=& f_0(x_i)  +\epsilon_i,  \,\,\, i= 1,\ldots,n,\\
x_i  &  \overset{\text{ind}}{\sim}& [0,1]^d,
\end{array}
\]
where  $\epsilon_i \sim G_i$  for a distribution  $G_i$  in  $\mathbb{R}$, and  with $f_0 \,:\,[0,1]^d \rightarrow \mathbb{R}$  for  a choice  of $d$  that  is scenario dependent. 
We consider 5 different scenarios following this template:

\paragraph{Scenario 1.}  We set 
\[
\begin{array}{lll}
f_0(q)  & = &  g_2\circ g_1(q), \,\,\,\, \forall q\in \mathbb{ R}^2\\
g_1(q)  & = &  (\sqrt{q_{1}} +   q_{1}\cdot q_{2},  \mathrm{cos}(2\pi q_{2})  )^{\top}  ,\,\,\,\,\, \forall  q\in \mathbb{ R}^2,  \\  
g_2(q)  &=  & \sqrt{q_1 +   q_2^2 }   +   q_1^2 \cdot q_2,    ,\,\,\,\,\, \forall  q\in \mathbb{ R}^2,  \\   
\end{array}
\]
and   $\epsilon_i  =  v_i g_3(x_i) $ where
\[
g_3(q)  =   \|q-(1/2,1/2)^{\top}\|,  \,\,\,\,\, \forall  q\in \mathbb{ R}^2, 
\] 
with  $v_i  \overset{ \mathrm{ind} } { \sim} t(2)$, for  $i=1,\ldots,n$, where  $t(2)$ is the t-distribution with  2 degrees of freedom. 


\paragraph{Scenario 2.} In this scenario we specify 
\[
f_0(q) =   q_1^2+q_2^2,\,\,\,\,q\in [0,1]^2,
\]
and generate $\epsilon_i  \overset{\text{ind}}{\sim} \mathrm{Laplace}(0,2)$ for  $i=1,\ldots,n$.

\paragraph{Scenario 3.} This is constructed by defining $f_0 \,:\, [0,1]^2 \rightarrow \mathbb{ R}$ as
\[
\begin{array}{lll}
f_0(q) &=   & \begin{cases}
\sqrt{q_1+q_2} +1  & \text{if}\,\,\,\,\,\,   q_1<0.5,\\
\sqrt{q_1+q_2}  & \text{otherwise},   \\
\end{cases}\\
\end{array}
\]
and setting
\[
\epsilon_i = \sqrt{x_i^{\top} \beta} \nu_i,
\]
where  $\beta =  (1,1/2)^{\top}$ and  $\nu_i  \overset{\text{ind}}{\sim}   t(2)$ for  $i=1,\ldots,n$.

\paragraph{Scenario 4.} The function $f_0$ is chosen as
\[
\begin{array}{lll}
f_0(q)   &  = &   \sqrt{q_1+q_2+q_3+q_4+q_5},  \,\, q \in [0,1]^d,\\
\end{array}
\]
with   $d=5$, and the errors as $\epsilon_i  \overset{\text{ind}}{\sim} \mathrm{Laplace}(0,2)$ for $i=1,\ldots,n$. Here,  $\mathrm{Laplace}(0,2)$  is the Laplace distribution with  mean zero and scale parameter $2$.

\begin{figure}
	\begin{subfigure}{.33\textwidth}
		\caption{Quantile Forests, $\tau =0.05$.}
		\includegraphics[width=2.2in,height=2in]{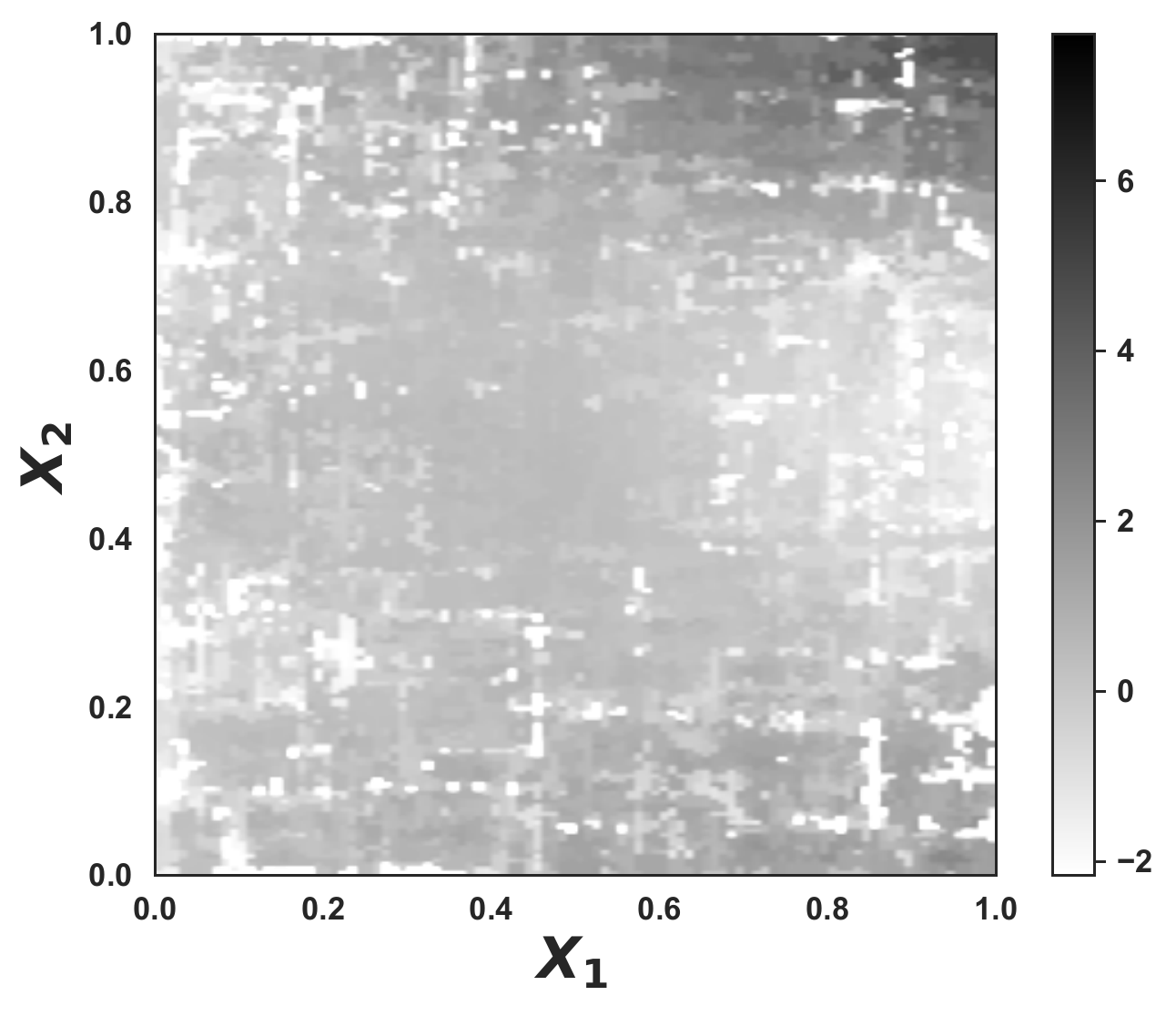}
	\end{subfigure}
	\begin{subfigure}{.33\textwidth}
		\caption{Quantile Network, $\tau =0.05$.}
		\includegraphics[width=2in,height=2in]{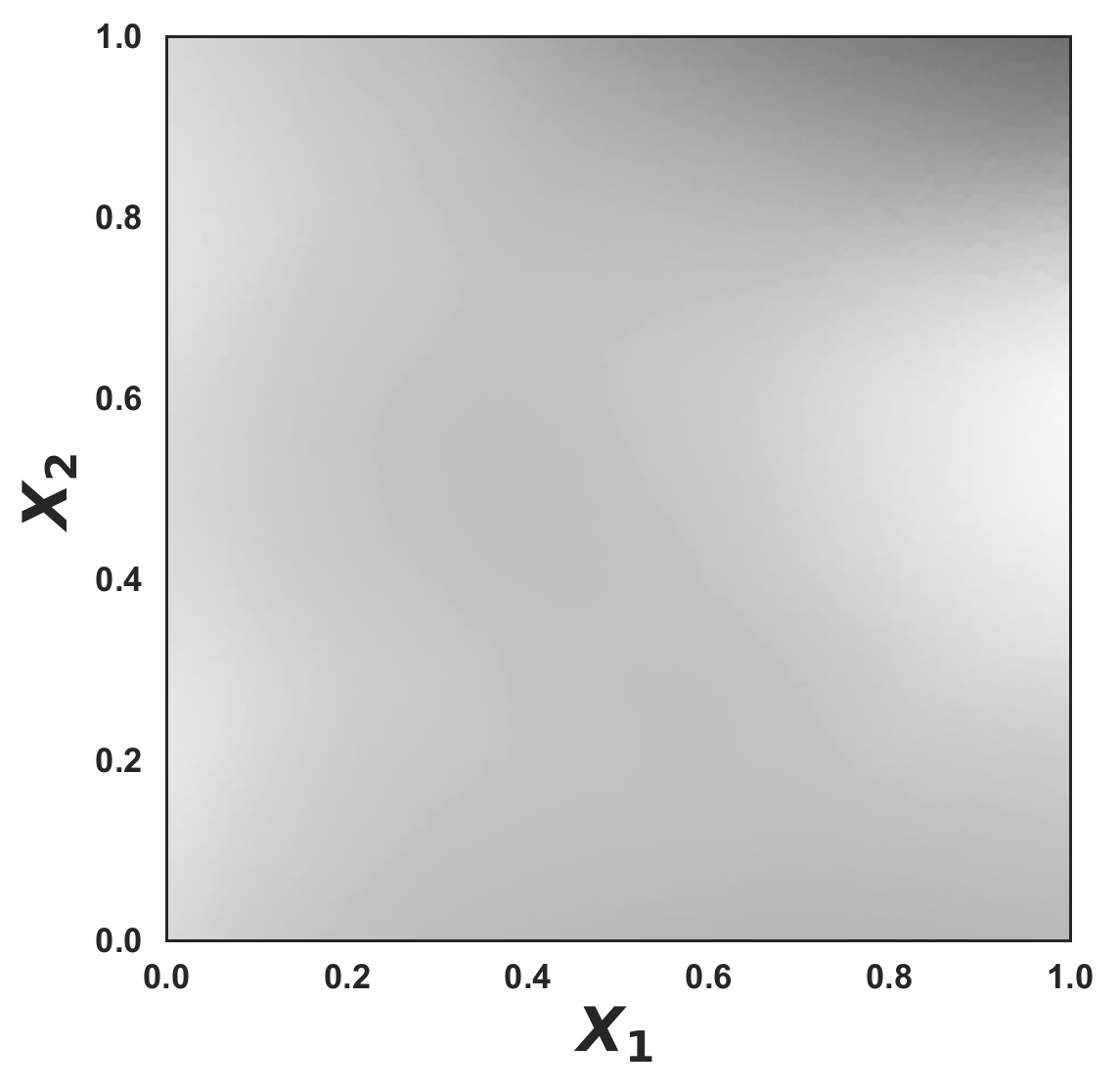}
	\end{subfigure}
	\begin{subfigure}{0.33\textwidth}
		\caption{True $f_{\tau}^*$, $\tau =0.05$.}%
		\includegraphics[width=2in,height=2in]{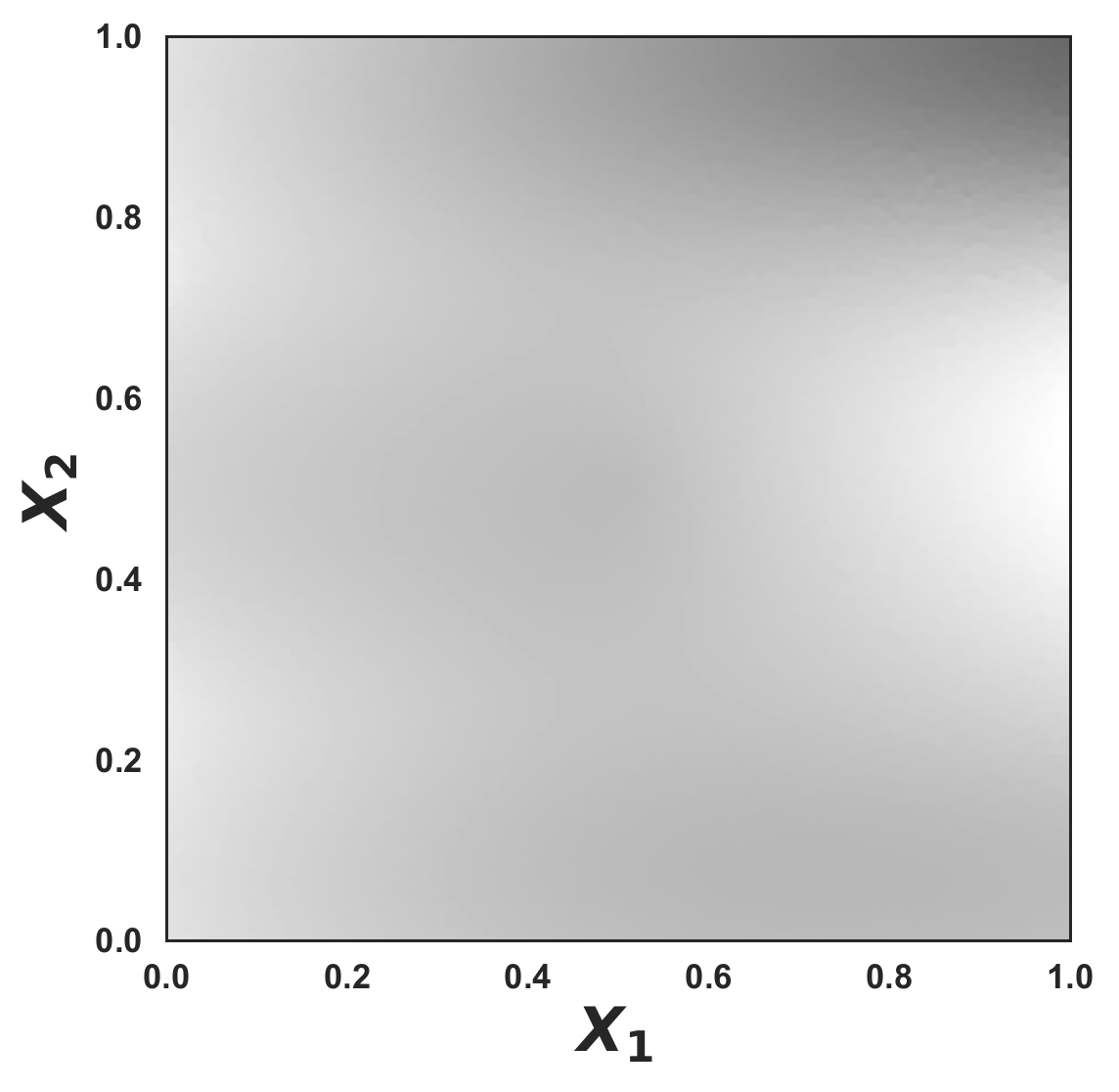}
	\end{subfigure}\\
	\begin{subfigure}{.33\textwidth}
		\caption{Quantile  Forests, $\tau =0.50$.}
		\includegraphics[width=2.2in,height=2in]{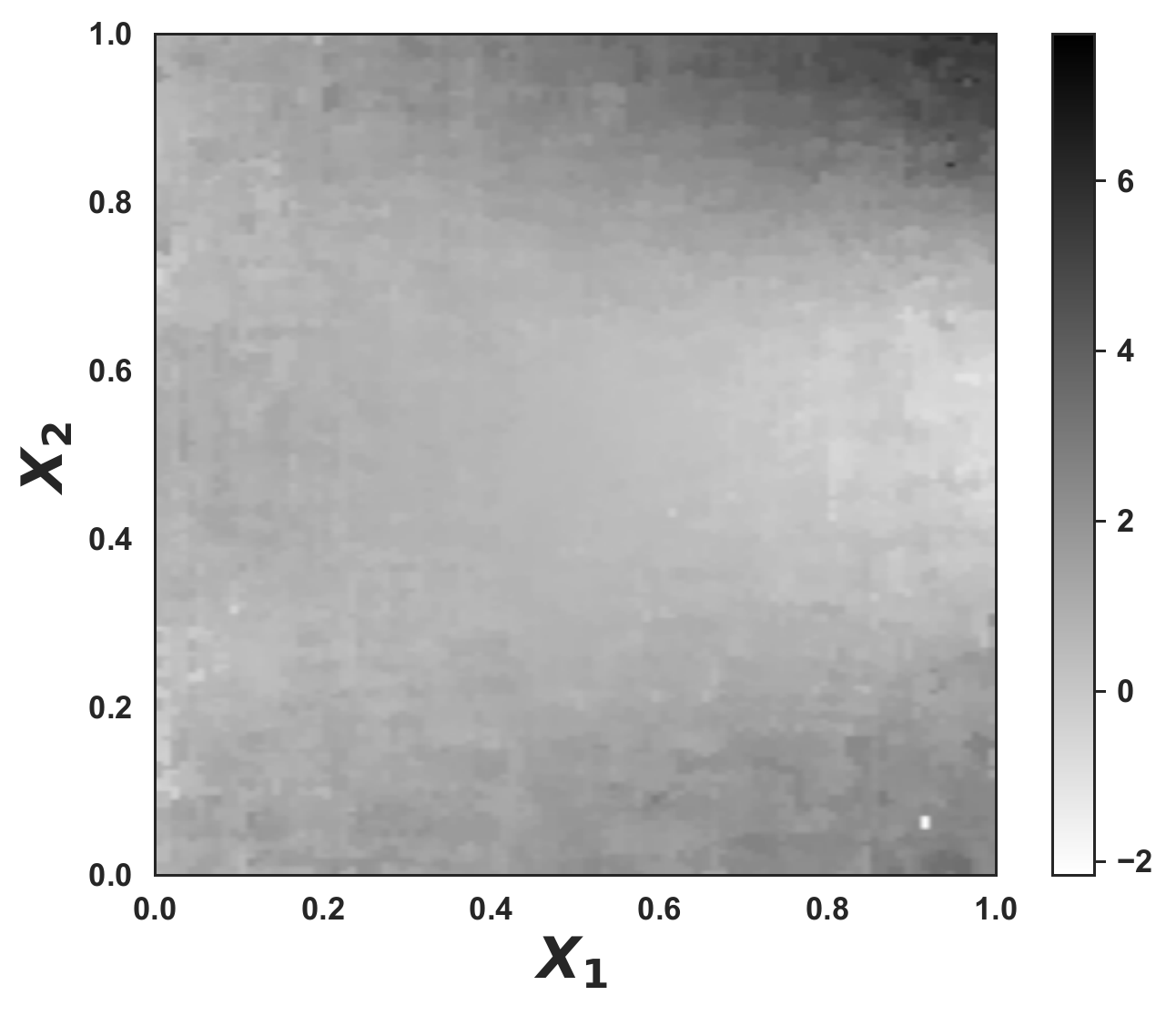}
	\end{subfigure}
	\begin{subfigure}{.33\textwidth}
		\caption{Quantile Network, $\tau =0.50$.}
		\includegraphics[width=2in,height=2in]{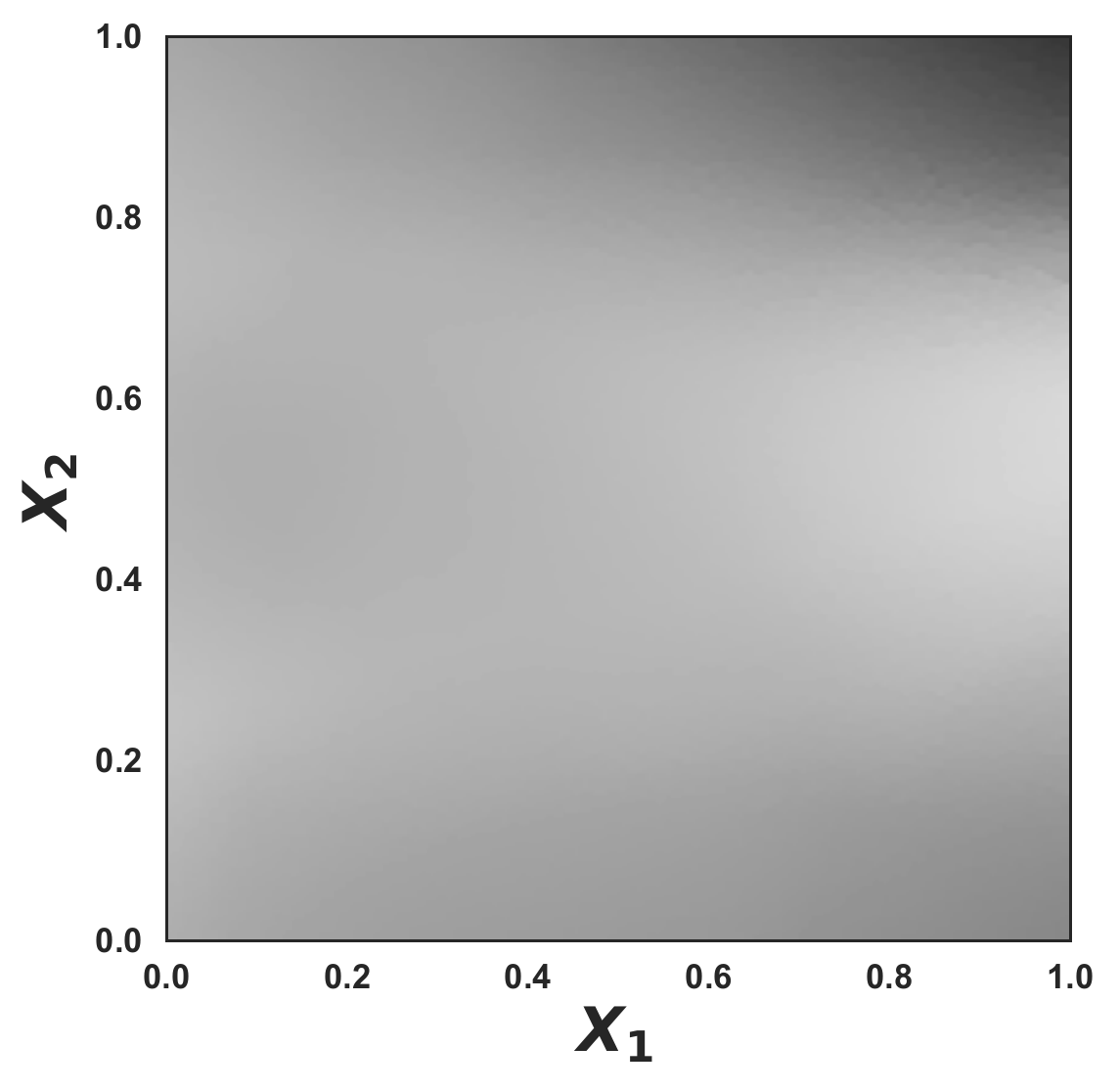}
	\end{subfigure}
	\begin{subfigure}{0.33\textwidth}
		\caption{True $f_{\tau}^*$, $\tau =0.50$.}%
		\includegraphics[width=2in,height=2in]{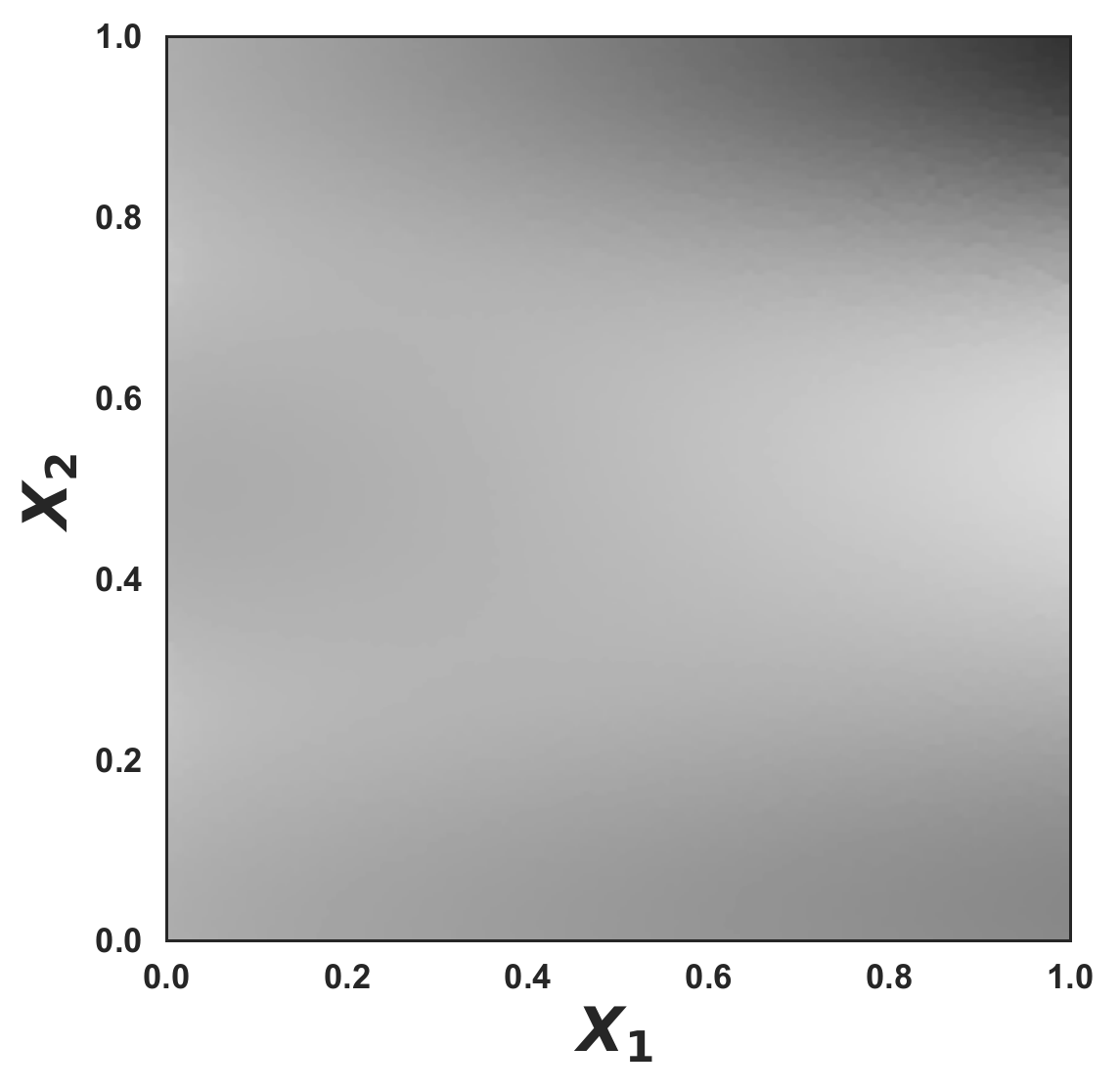}
	\end{subfigure}\\
	\begin{subfigure}{.33\textwidth}
		\caption{Quantile Forests, $\tau =0.95$.}
		\includegraphics[width=2.2in,height=2in]{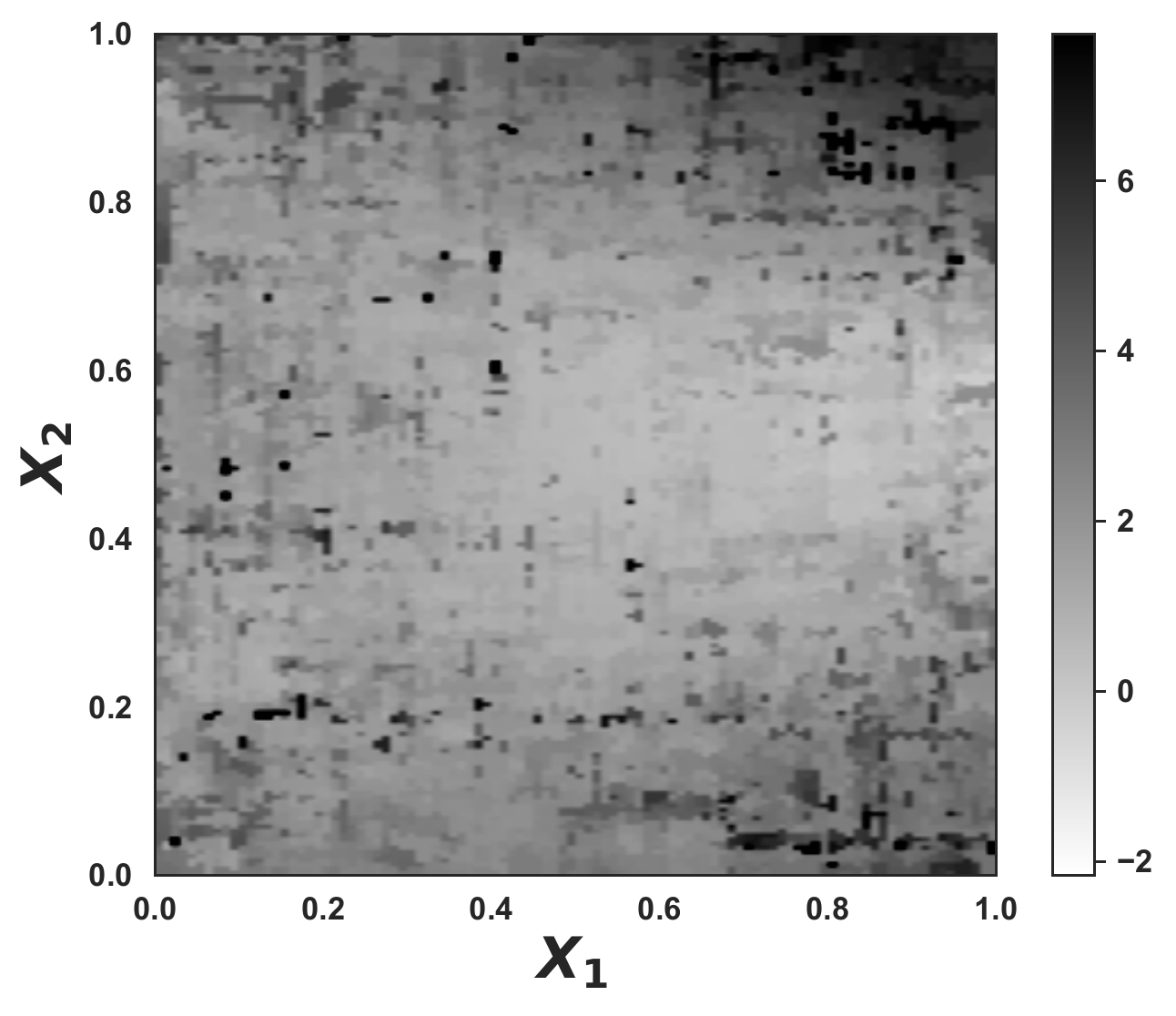}
	\end{subfigure}
	\begin{subfigure}{.33\textwidth}
		\caption{Quantile Network, $\tau =0.95$.}
		\includegraphics[width=2in,height=2in]{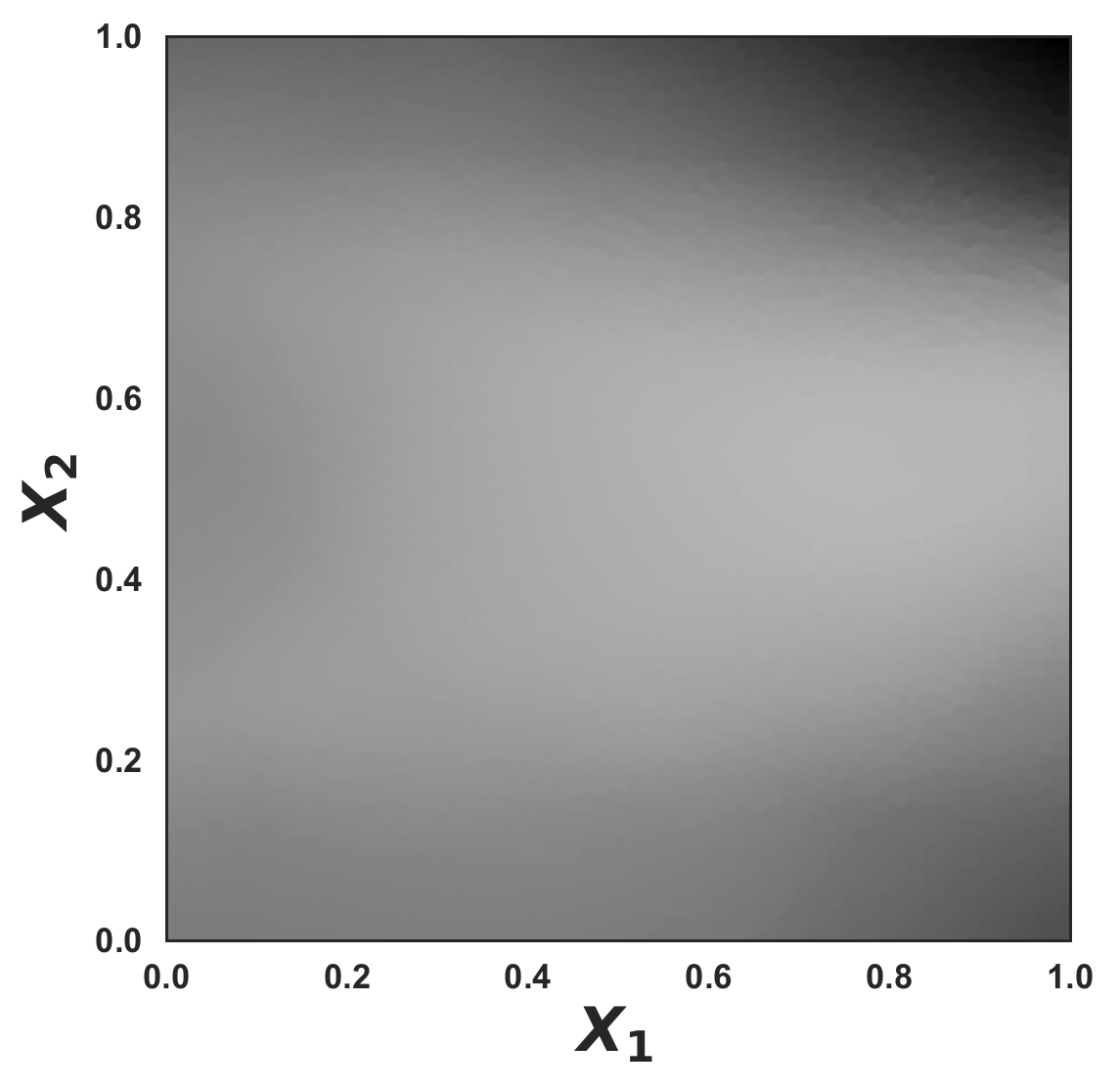}
	\end{subfigure}
	\begin{subfigure}{0.33\textwidth}
		\caption{True $f_{\tau}^*$, $\tau =0.95$.}%
		\includegraphics[width=2in,height=2in]{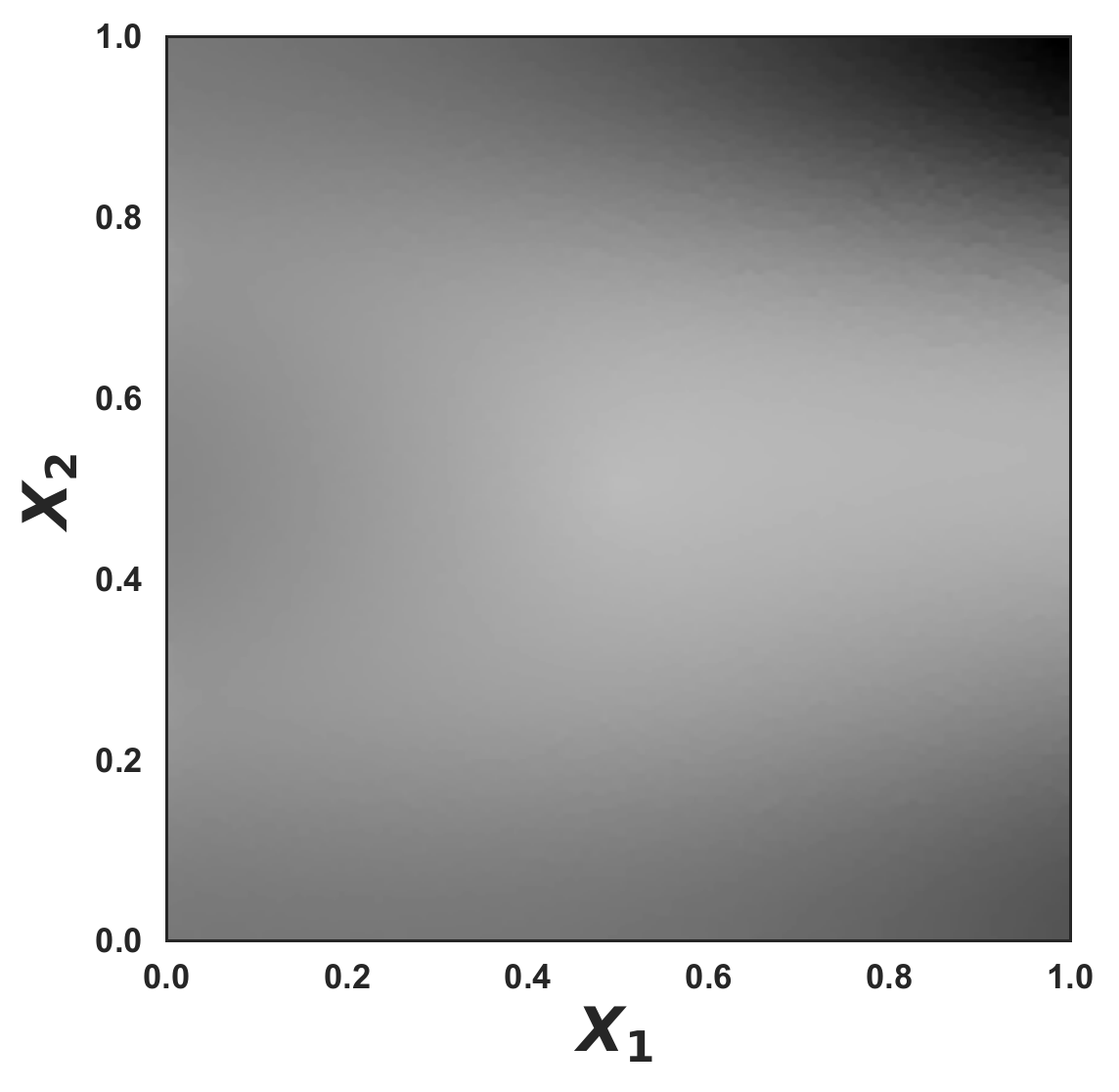}
	\end{subfigure}\\
	%
	%
	
	\caption{	\label{fig1} One instance of the true quantile function with $\tau \in\{0.05,0.50,0.95 \}$  and  its corresponding estimates based on Quantile Network and Quantile Forests. Here $n=10000$ and the data are generated under Scenario 1.}
\end{figure}

\paragraph{Scenario 5.} The function $f_0 \,:\,[0,1]^{10} \rightarrow \mathbb{R}$ is defined as  $f_0(q) = g_3  \circ g_2 \circ g_1(q)$ where 
\[
\begin{array}{lll}
g_1(q) & =& (   \sqrt{q_1^2 +    \sum_{j=2}^{10} q_j  },  (\sum_{j=1}^{10} q_j )^3 )^{\top},  \,\, q \in [0,1]^{10},\\
g_2(q)   &  = &  (\vert q_1\vert,q_2 \cdot q_1  )^{\top} ,  \,\, q \in [0,1]^2,\\
g_3(q) & =&  q_1 + \sqrt{q_1+q_2} ,  \,\, q \in [0,1]^2,\\
\end{array}
\]
where  $\epsilon_i  \overset{\text{ind}}{\sim}  t(3)$ for  $i=1,\ldots,n$, and $t(3)$  denotes the $t$-distribution with $3$ degrees of freedom.


\begin{figure}[t!]
	\centering
	\begin{subfigure}{0.43\textwidth}
		\caption{ $f_{\tau}^*$ }%
		\includegraphics[width=2.4in,height=2.1in]{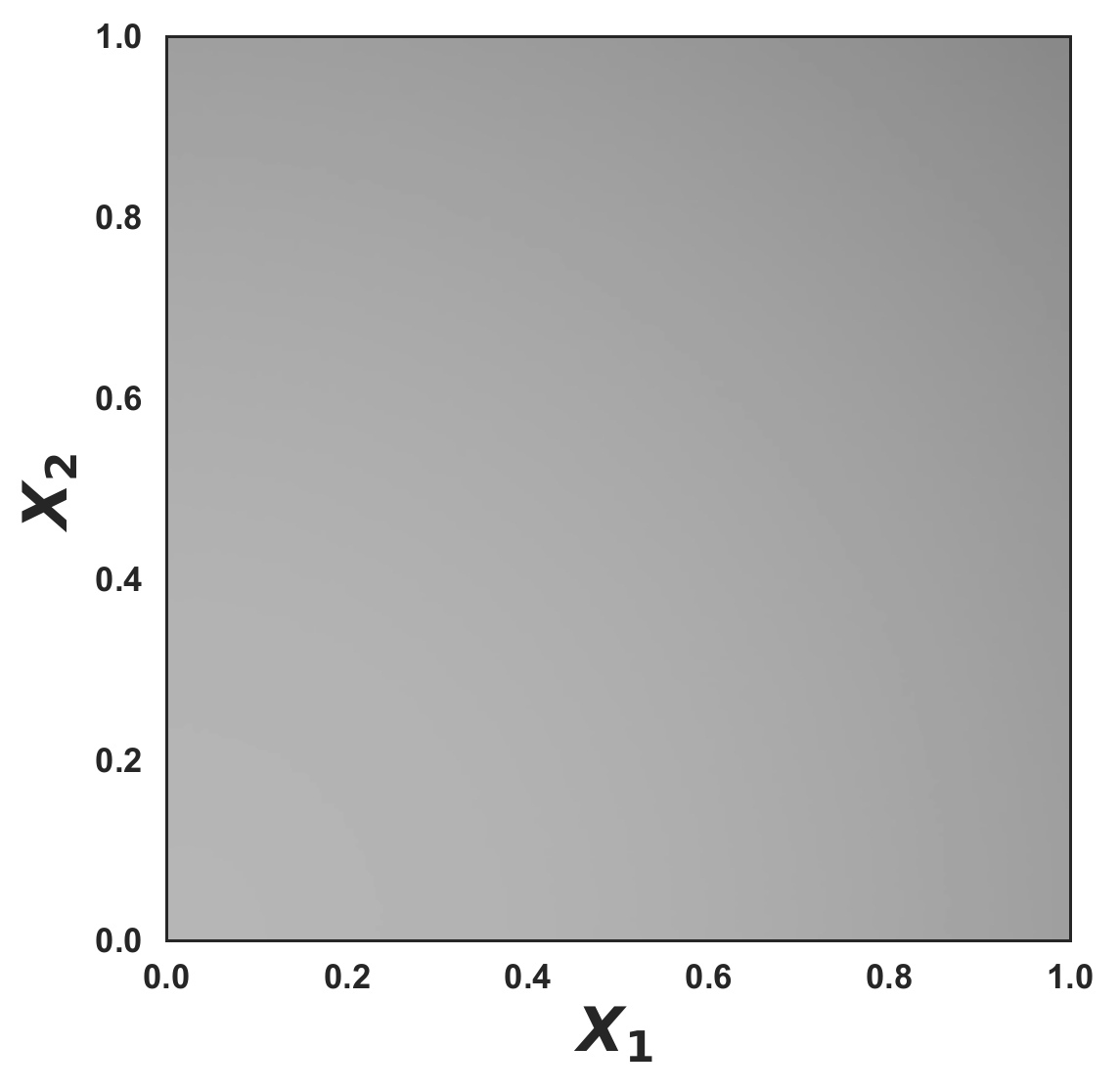}
	\end{subfigure}
	\begin{subfigure}{.43\textwidth}
		\caption{Quantile  Forests}
		\includegraphics[width=2.75in,height=2.1in]{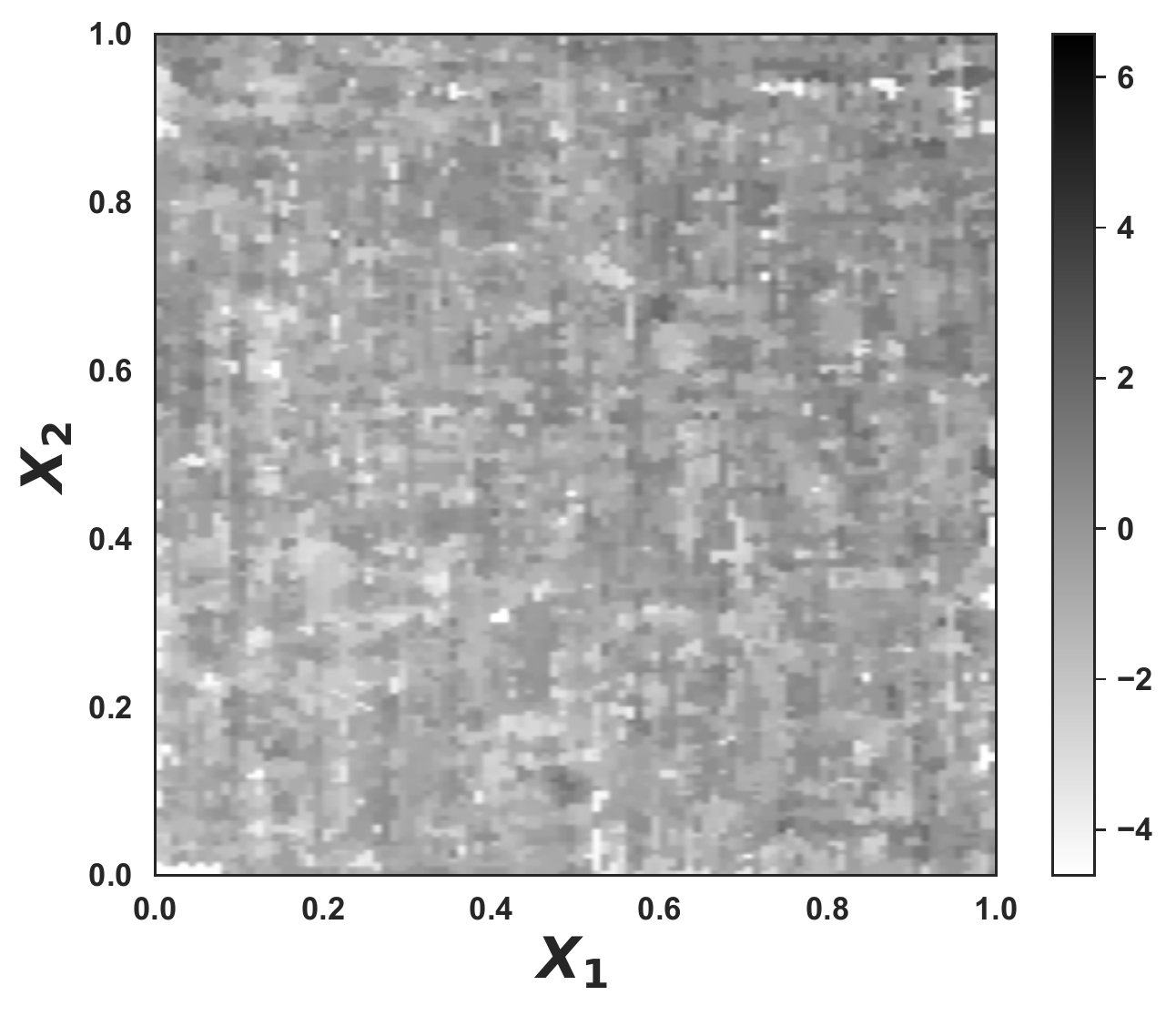}
	\end{subfigure}
	\begin{subfigure}{.43\textwidth}
		\caption{Quantile Network}
		\includegraphics[width=2.4in,height=2.1in]{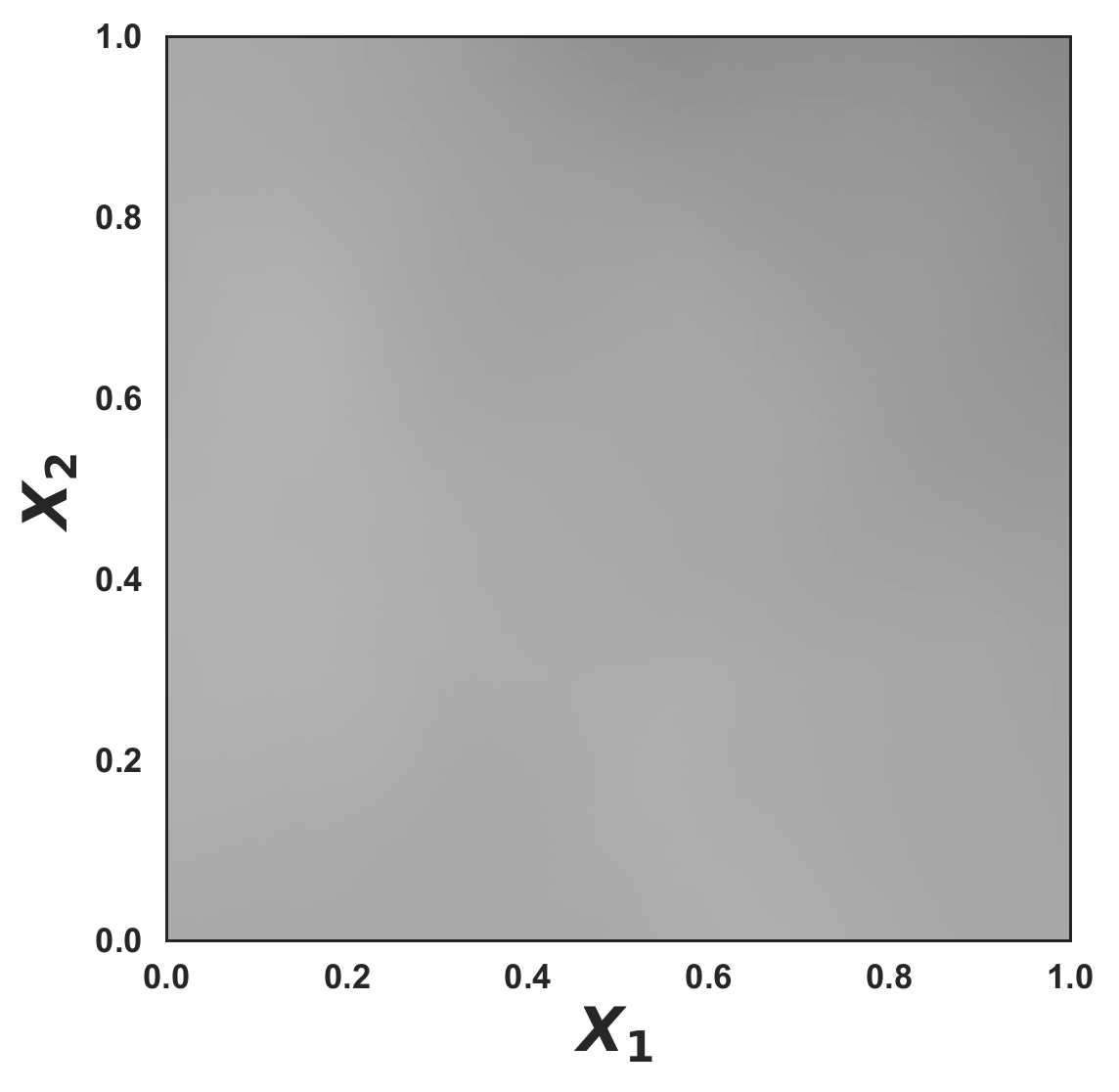}
	\end{subfigure}
	\begin{subfigure}{.43\textwidth}
		\caption{Quantile Spline}
		\includegraphics[width=2.4in,height=2.1in]{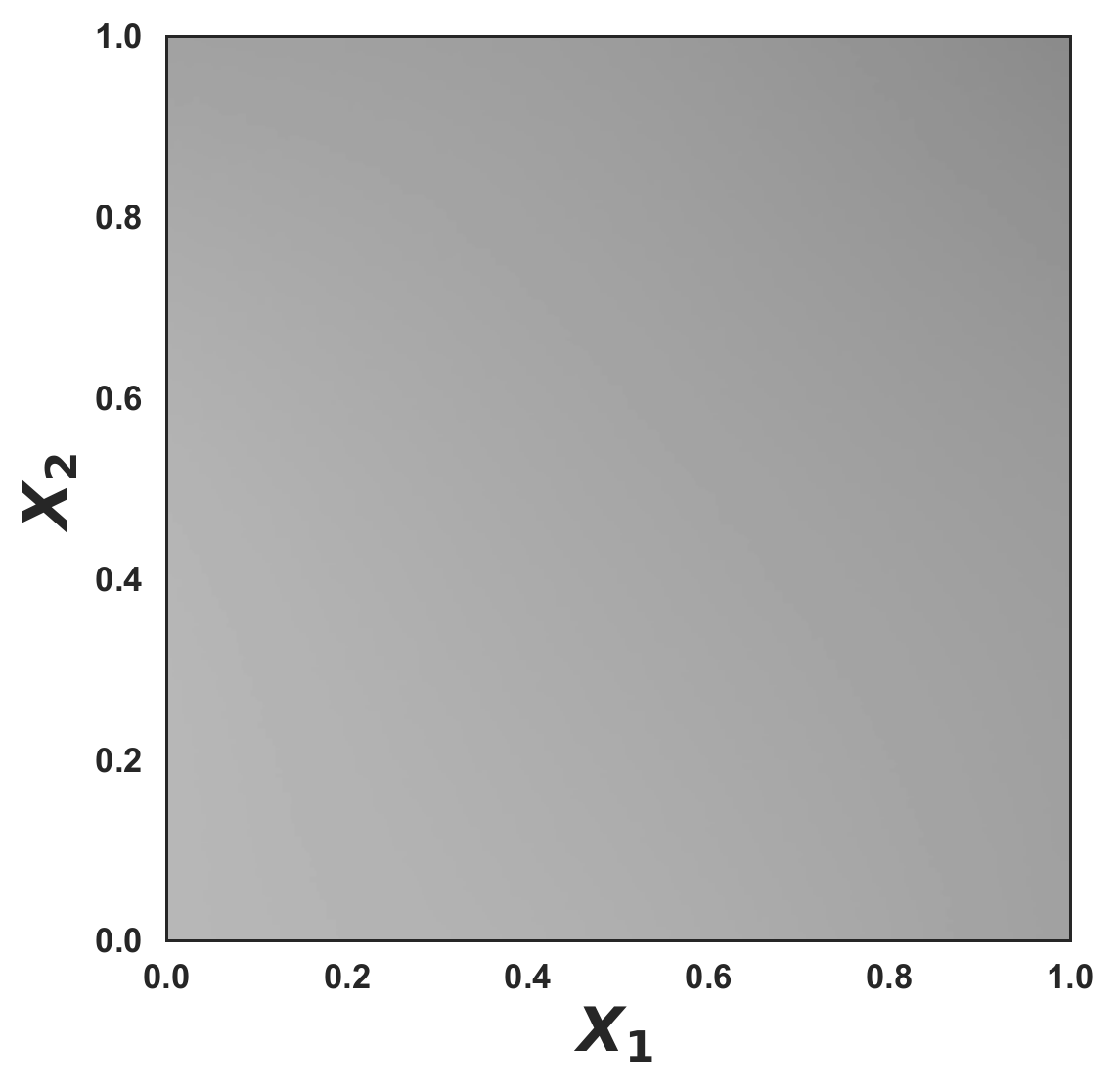}
	\end{subfigure}
	\caption{\label{fig2} One instance of the true quantile function with $\tau=0.25$  and  its corresponding estimates obtained from different methods. Here $n=10000$ and the data are generated under Scenario 2.}
\end{figure}

\begin{figure}[t!]
	\begin{subfigure}{.33\textwidth}
		\caption{Quantile Forests, $\tau =0.50$}
		\includegraphics[width=2in,height=1.7in]{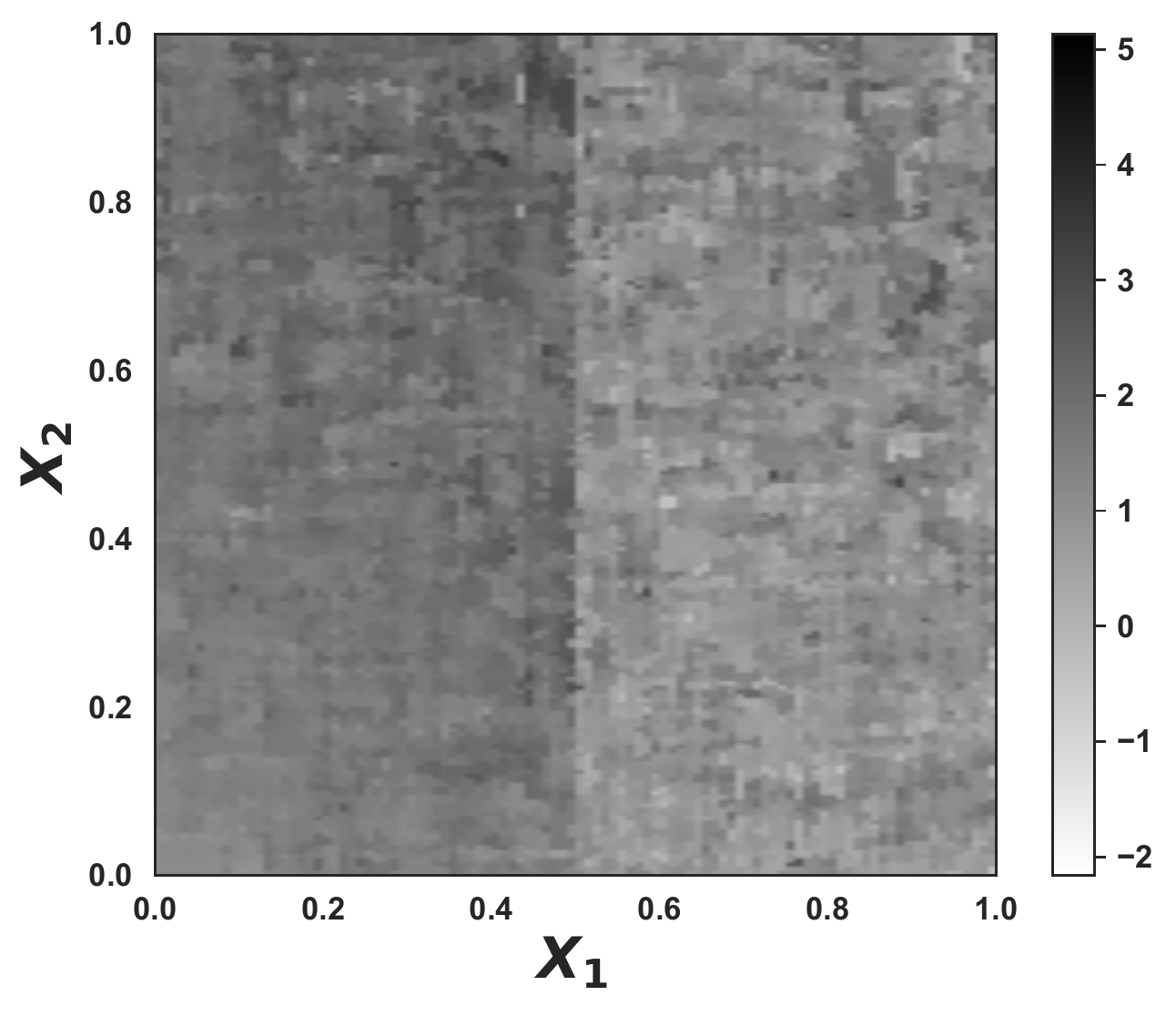}
	\end{subfigure}
	\begin{subfigure}{.33\textwidth}
		\caption{Quantile Network, $\tau =0.50$}
		\includegraphics[width=1.8in,height=1.7in]{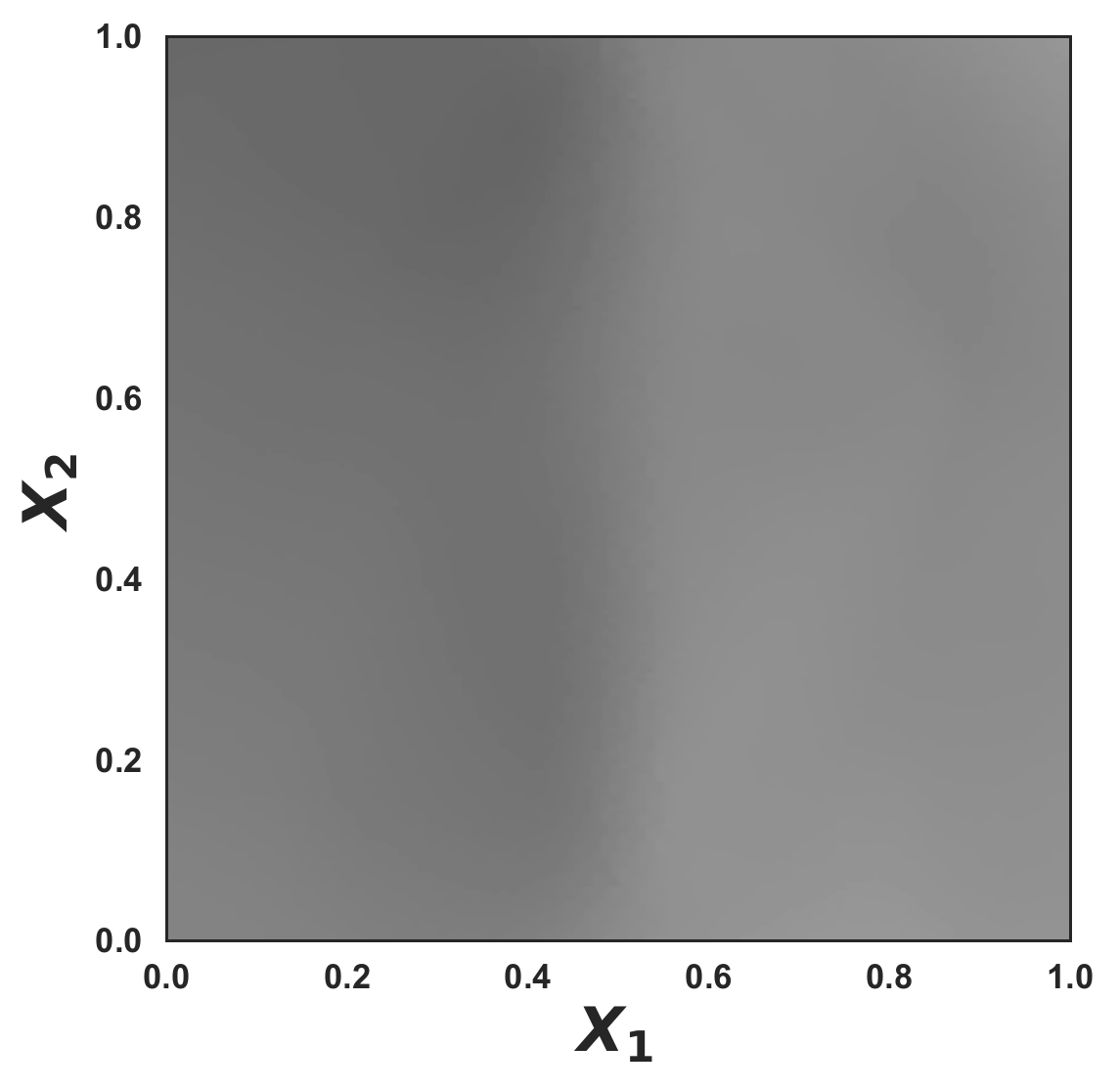}
	\end{subfigure}
	\begin{subfigure}{0.33\textwidth}
		\caption{True $f_{\tau}^*$,  $\tau =0.50$}%
		\includegraphics[width=1.8in,height=1.7in]{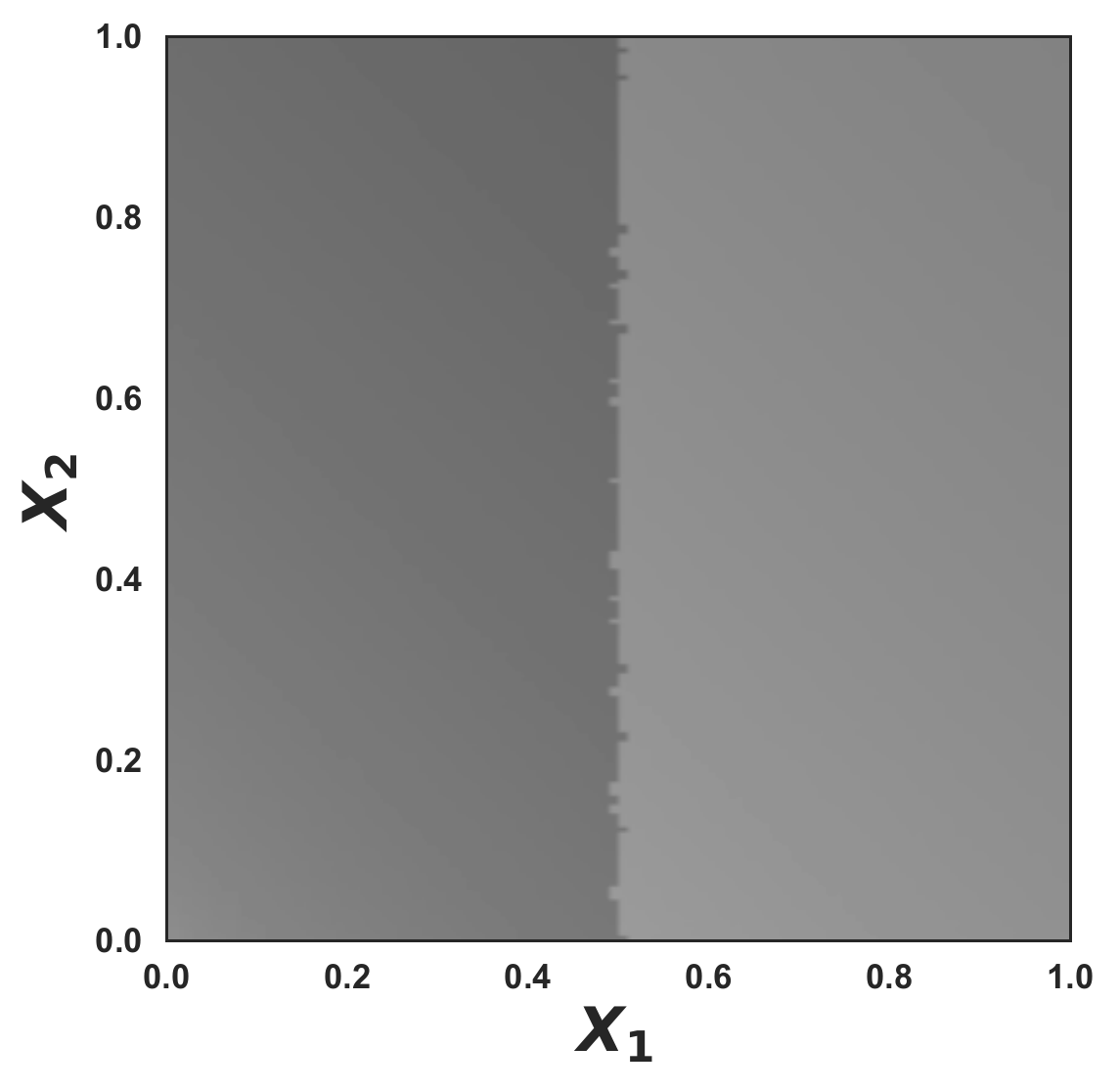}
	\end{subfigure}\\
	\vspace{0.1in}\\
	\begin{subfigure}{.33\textwidth}
		\caption{Quantile Spline, $\tau =0.50$}
		\includegraphics[width=2in,height=1.7in]{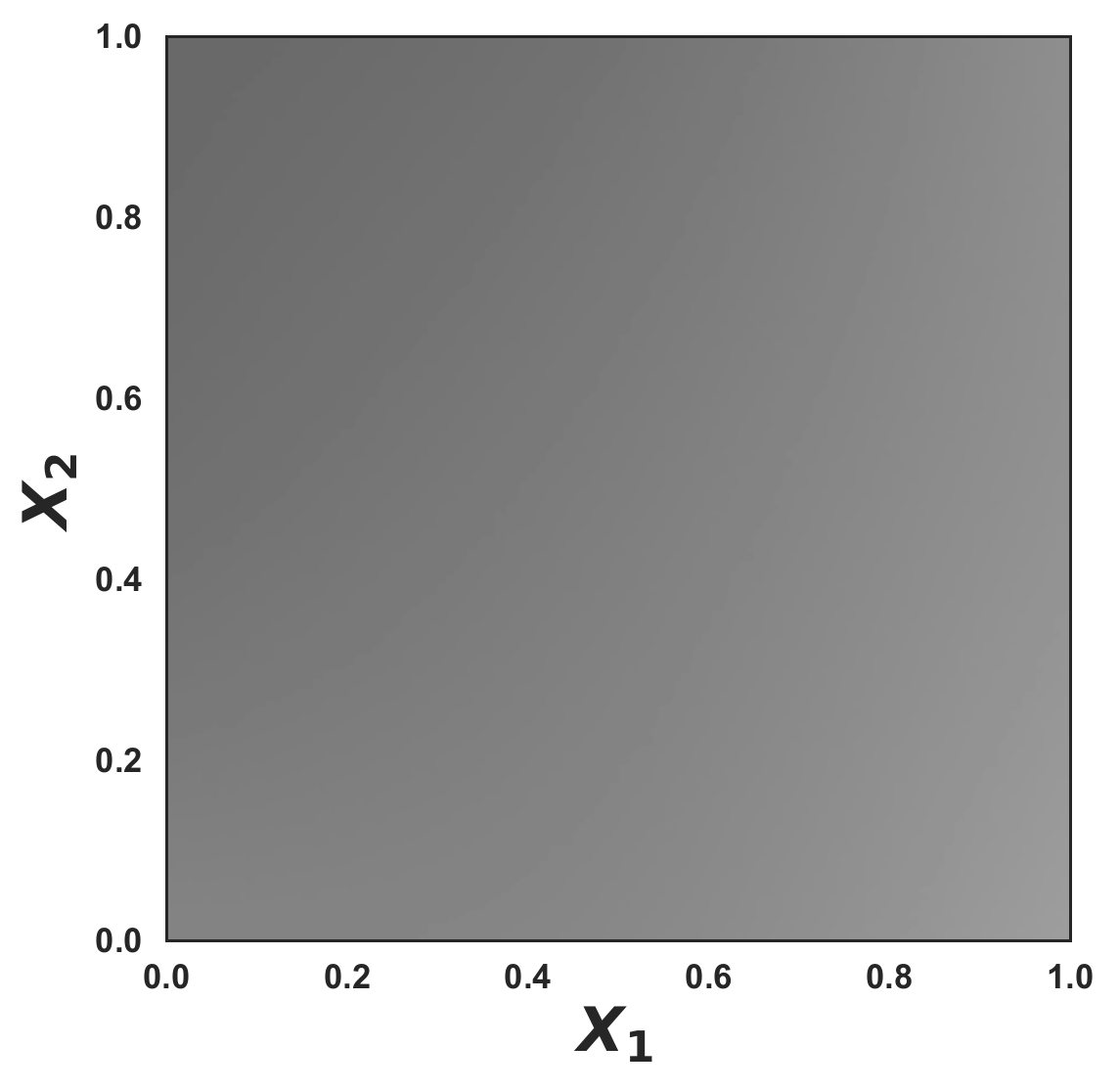}
	\end{subfigure}
	\begin{subfigure}{.33\textwidth}
		\caption{SqErr Network, $\tau =0.50$}
		\includegraphics[width=1.8in,height=1.7in]{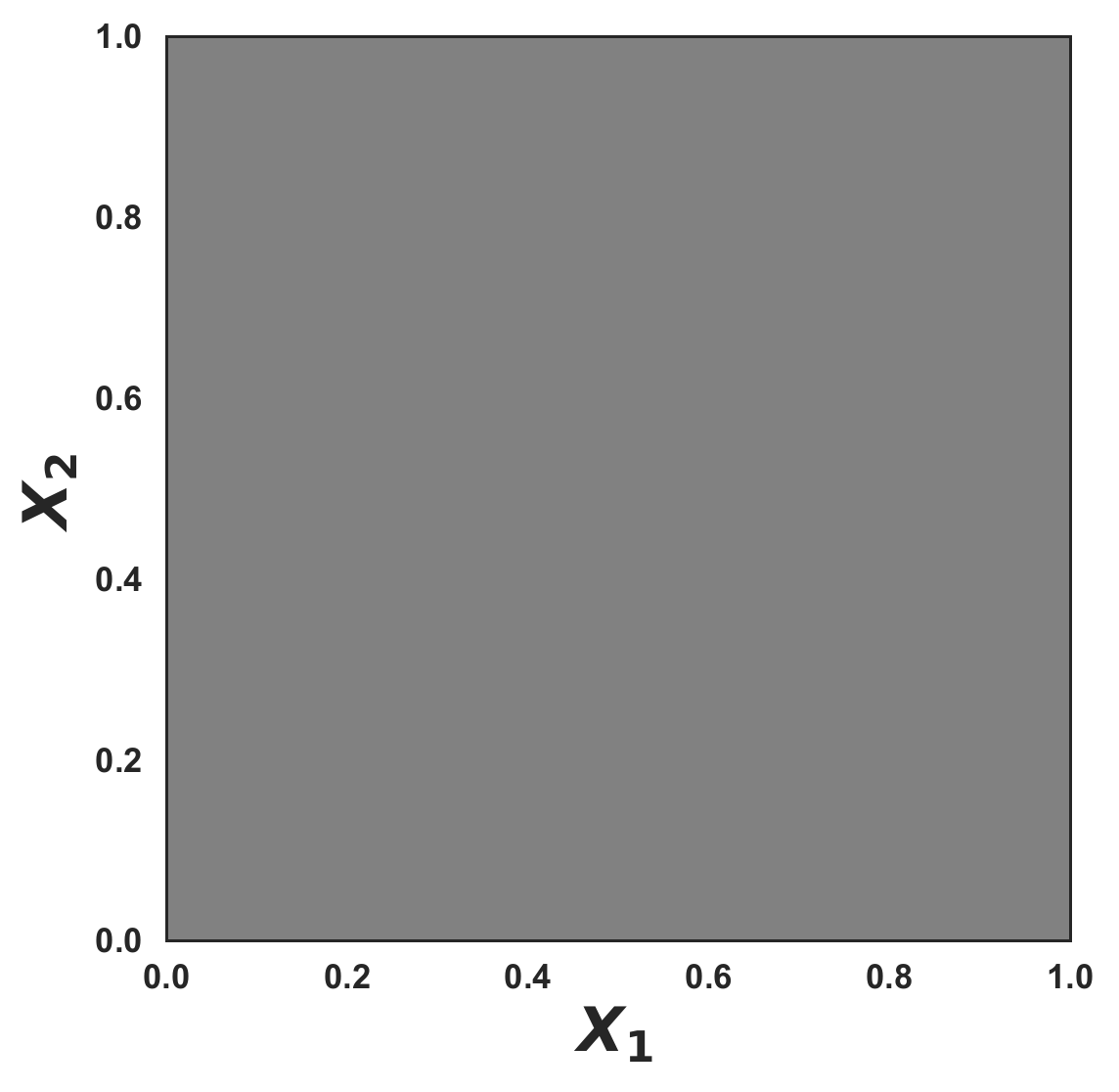}
	\end{subfigure}
	\begin{subfigure}{.33\textwidth}
		\caption{Quantile Network, $\tau=0.75$}
		\includegraphics[width=1.8in,height=1.7in]{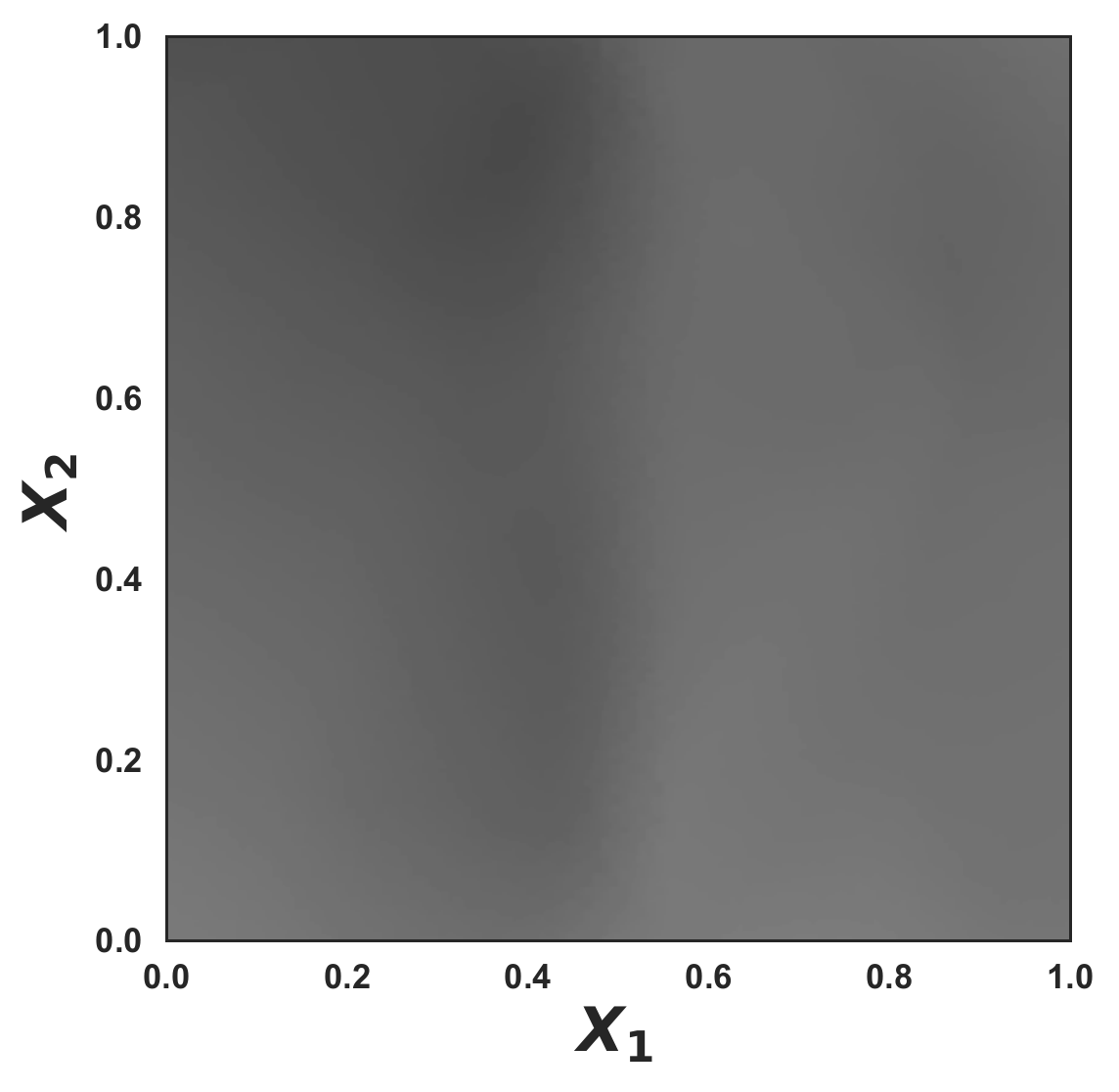}
	\end{subfigure}\\
	\begin{subfigure}{.33\textwidth}
		\caption{Quantile Forests, $\tau =0.75$.}
		\includegraphics[width=2in,height=1.7in]{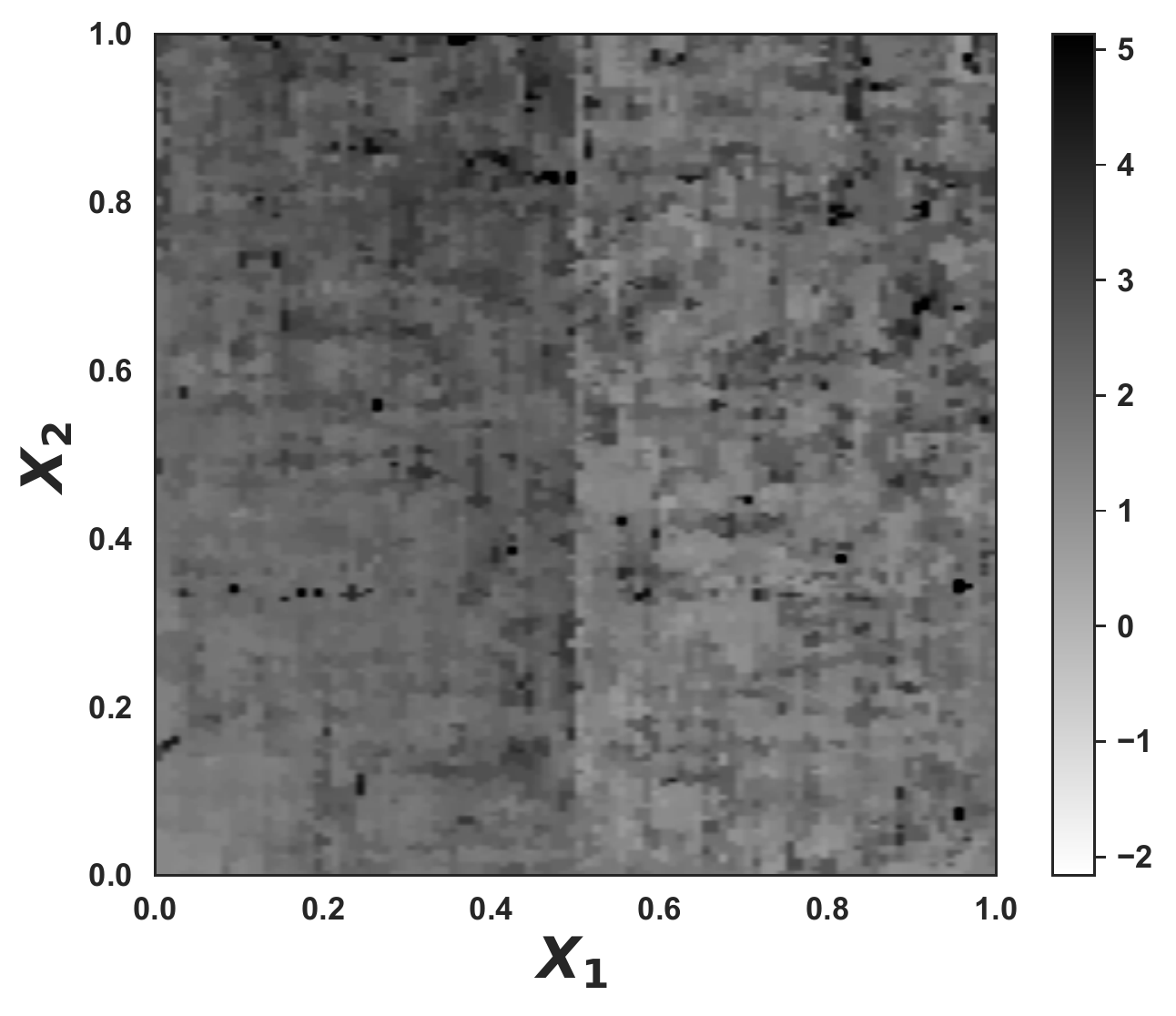}
	\end{subfigure}
	\begin{subfigure}{.33\textwidth}
		\caption{Quantile Spline, $\tau =0.75$.}
		\includegraphics[width=	1.8in,height=1.7in]{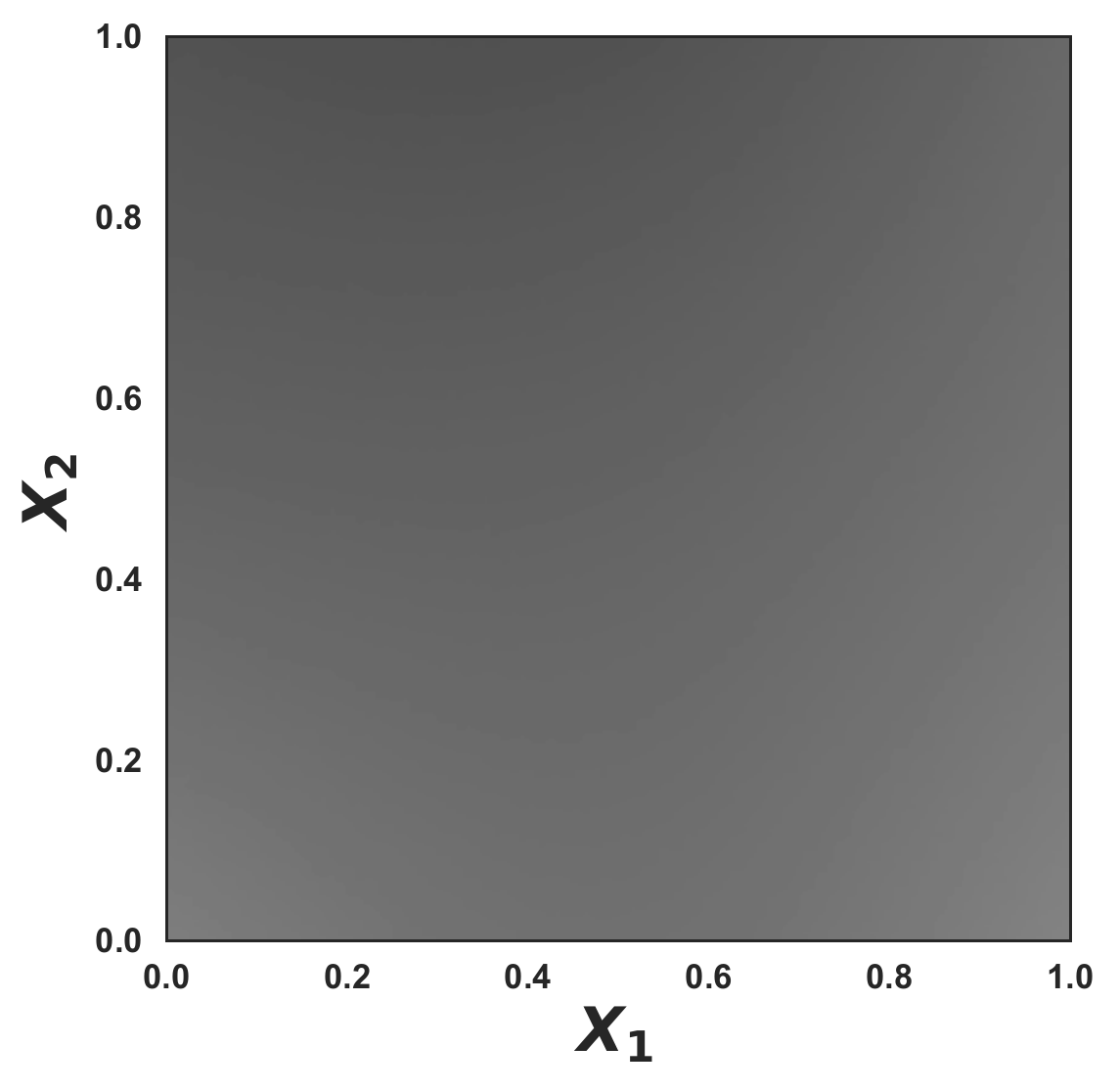}
	\end{subfigure}
	\begin{subfigure}{0.33\textwidth}
		\caption{True $f_{\tau}^*$, $\tau =0.75$}%
		\includegraphics[width=1.8in,height=1.7in]{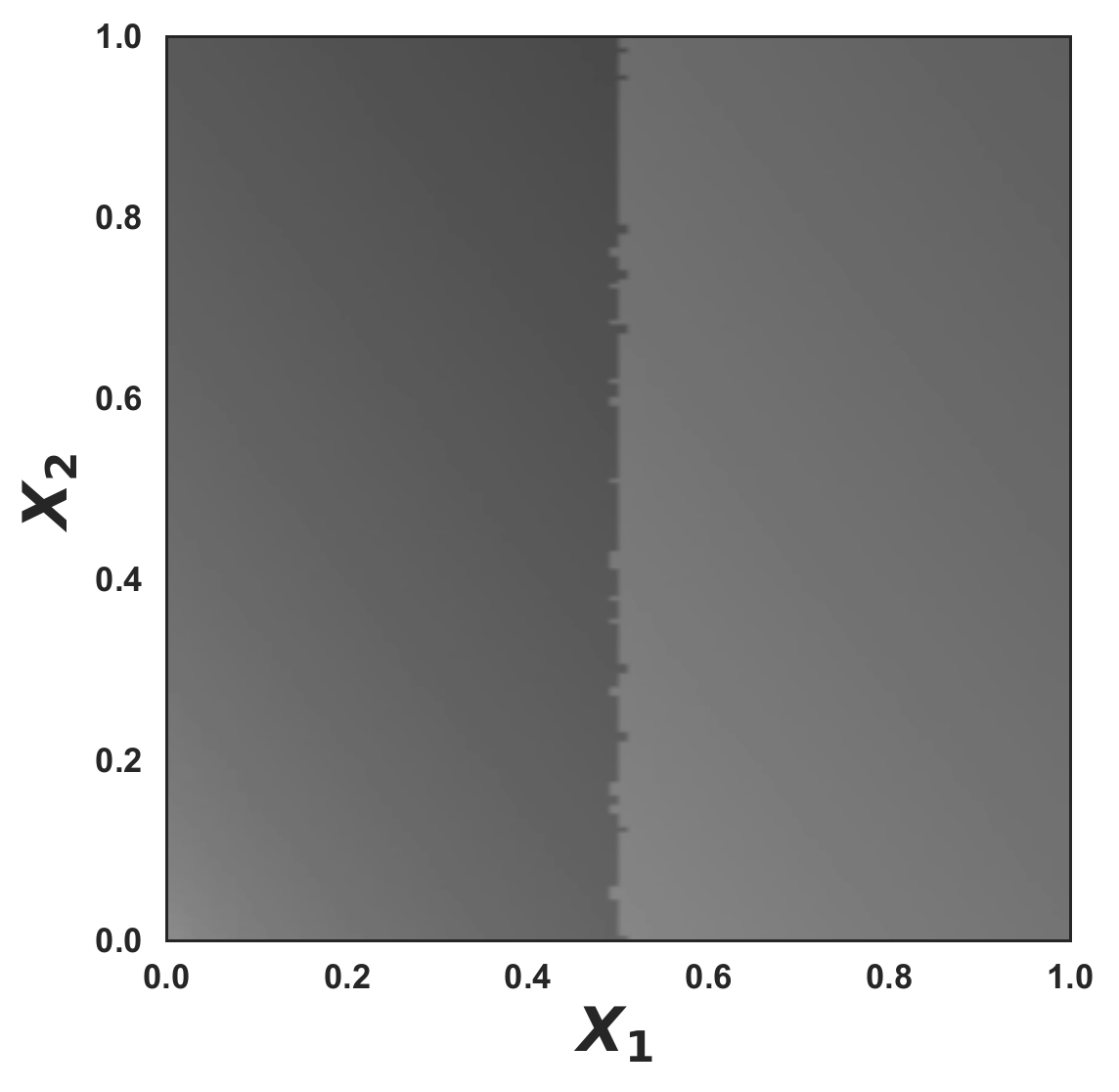}
	\end{subfigure}\\
	
	\caption{\label{fig3}One instance of the true quantile function with $\tau\in \{0.50,0.75\}$  and  of the corresponding estimates  obtained with the different methods. Here $n=10000$ and the data are generated under Scenario 3.}
\end{figure}


\begin{table}[t!]
	\centering
	\begin{small}
		\begin{tabular}{llccccc}
			\toprule
			\multicolumn{7}{c}{\textbf{Scenario 1}} \\
			\midrule
			$n$ & Method & $\tau=0.05$ & $\tau=0.25$ & $\tau=0.50$ & $\tau=0.75$ & $\tau=0.95$ \\
			\midrule
			\multirow{4}{*}{$100$} 
			& SqErr Network        & * & * & 0.98 & * & * \\
			& Quantile Network     & 0.60 & 0.50 & 0.46 & 0.68 & 1.85 \\
			& Quantile Spline      & 1.79 & 0.19 & 0.14 & 0.17 & 3.36 \\
			& Quantile Forests      & 1.69 & 0.54 & 0.34 & 0.69 & 2.37 \\         			
			\midrule
			\multirow{4}{*}{$1000$} 
			& SqErr Network        & * & *& 0.95 & * & * \\
			& Quantile Network     & 0.36 & 0.06 & 0.04 & 0.06 & 0.18 \\
			& Quantile Spline      & 0.17 & 0.07 & 0.06 & 0.07 & 0.17 \\
			& Quantile Forests      & 1.76 & 0.20 & 0.07 & 0.22 & 1.77 \\
			
			\midrule
			\multirow{4}{*}{$10000$} 
			& SqErr Network        & * & * & 0.94 & * & * \\
			& Quantile Network     & 0.04 & 0.01 & 0.01 & 0.01 & 0.03 \\
			& Quantile Spline      & 0.07 & 0.06 & 0.06 & 0.06 & 0.07 \\
			& Quantile Forests      & 2.90 & 0.16 & 0.08 & 0.67 & 13.89 \\
			\midrule
			\multicolumn{7}{c}{\textbf{Scenario 2}} \\
			\midrule
			$n$ & Method & $\tau=0.05$ & $\tau=0.25$ & $\tau=0.50$ & $\tau=0.75$ & $\tau=0.95$ \\
			\midrule
			\multirow{4}{*}{$100$} 
			& SqErr Network        & * & * & 0.28 & * & * \\
			& Quantile Network     & 3.82 & 2.37 & 2.21 & 2.64 & 4.43 \\
			& Quantile Spline      & 5.51 & 1.28 & 0.75 & 1.28 & 7.27 \\
			& Quantile Forests      & 4.13 & 1.68 & 1.10 & 1.66 & 4.42 \\
			\midrule
			\multirow{4}{*}{$1000$}
			& SqErr Network        & * & * & 0.26 & * &* \\
			& Quantile Network     & 0.41 & 0.23 & 0.21 & 0.24 & 0.48 \\
			& Quantile Spline      & 0.72 & 0.11 & 0.04 & 0.11 & 0.53 \\
			& Quantile Forests      & 4.08 & 1.24 & 0.68 & 1.18 & 3.62 \\
			\midrule
			\multirow{4}{*}{$10000$} 
			& SqErr Network        & *& * & 0.19 & * & *\\
			& Quantile Network     & 0.07 & 0.03 & 0.03 & 0.04 & 0.08 \\
			& Quantile Spline      & 0.06 & 0.01 & 0.00 & 0.01 & 0.06 \\
			& Quantile Forests      & 3.47 & 1.03 & 0.56 & 1.04 & 3.55 \\
			\midrule
			\multicolumn{7}{c}{\textbf{Scenario 3}} \\
			\midrule
			$n$ & Method & $\tau=0.05$ & $\tau=0.25$ & $\tau=0.50$ & $\tau=0.75$ & $\tau=0.95$ \\
			\midrule
			\multirow{4}{*}{$100$}
			& SqErr Network        & * & * & 0.24 & *& * \\
			& Quantile Network     & 1.33 & 0.80 & 0.93 & 1.57 & 2.83 \\
			& Quantile Spline      & 6.87 & 0.39 & 0.30 & 0.43 & 4.75 \\
			& Quantile Forests      & 5.49 & 0.81 & 0.37 & 0.81 & 4.42 \\
			\midrule
			\multirow{4}{*}{$1000$}
			& SqErr Network        & *& * & 0.25 & *& * \\
			& Quantile Network     & 0.21 & 0.09 & 0.08 & 0.09 & 0.32 \\
			& Quantile Spline      & 0.50 & 0.10 & 0.08 & 0.09 & 0.53 \\
			& Quantile Forests      & 11.26 & 0.93 & 0.20 & 0.77 & 9.96 \\
			\midrule
			\multirow{4}{*}{$10000$} 
			& SqErr Network        & * & * & 0.17 & * & * \\
			& Quantile Network     & 0.05 & 0.03 & 0.03 & 0.03 & 0.07 \\
			& Quantile Spline      & 0.09 & 0.07 & 0.07 & 0.07 & 0.11 \\
			& Quantile Forests      & 17.75 & 0.67 & 0.15 & 0.49 & 12.48 \\
			\bottomrule

		\end{tabular}
	\end{small}
	\caption{\label{tab1} Univariate Responses Tasks. Performances of different methods in Scenarios 1 -- 3, in terms of squared error from the true quantile averaged over 25 independent trials.}
\end{table}

\begin{table}[t!]
	\centering
	\begin{small}
		\begin{tabular}{llccccc}
			\multicolumn{7}{c}{\textbf{Scenario 4}} \\
			\midrule
			n & Method & $\tau=0.05$ & $\tau=0.25$ & $\tau=0.50$ & $\tau=0.75$ & $\tau=0.95$ \\
			\midrule
			\multirow{4}{*}{$100$} 
			& SqErr Network        & *& * & 0.30 & *&* \\
			& Quantile Network     & 5.96 & 3.06 & 3.44 & 4.02 & 5.13 \\
			& Quantile Spline      & * & *& * & * & * \\
			& Quantile Forests      & 3.65 & 1.85 & 1.11 & 1.60 & 3.55 \\
			\midrule
			\multirow{4}{*}{$1000$} 
			& SqErr Network        & * & *& 0.13 & * & *\\
			& Quantile Network     & 1.09 & 0.54 & 0.44 & 0.58 & 1.33 \\
			& Quantile Spline      & * & *& *&*& * \\
			& Quantile Forest      & 2.53 & 0.76 & 0.37 & 0.82 & 2.69 \\                   
			\midrule
			\multirow{4}{*}{$10000$} 
			& SqErr Network        & *& * & 0.05 & * &* \\
			& Quantile Network     & 0.20 & 0.09 & 0.07 & 0.10 & 0.25 \\
			& Quantile Spline      &* &* & *& *& * \\
			& Quantile Forest      & 1.71 & 0.42 & 0.18 & 0.42 & 1.71 \\
			\midrule
			\multicolumn{7}{c}{\textbf{Scenario 5}} \\
			\midrule
			$n$ & Model & $\tau=0.05$ & $\tau=0.25$ & $\tau=0.50$ & $\tau=0.75$ & $\tau=0.95$ \\
			\midrule
			\multirow{4}{*}{$100$} 
			& SqErr Network        & * & *& 30.59 & * & * \\
			& Quantile Network     & 5.05 & 4.37 & 3.61 & 4.26 & 6.96 \\
			& Quantile Spline      & * & *& * & *& *\\
			& Quantile Forest      & 41.86 & 21.94 & 15.36 & 20.69 & 53.68 \\
			\midrule
			\multirow{4}{*}{$1000$} 
			& SqErr Network        & 37.42 & 31.85 & 30.97 & 31.26 & 35.60 \\
			& Quantile Network     & 1.61 & 1.15 & 0.65 & 0.95 & 2.05 \\
			& Quantile Spline      & * & * & * & * & * \\
			& Quantile Forest      & 29.01 & 12.53 & 7.33 & 11.05 & 33.29 \\
			\midrule
			\multirow{4}{*}{$10000$} 
			& SqErr Network        & 35.97 & 31.13 & 30.60 & 31.24 & 36.31 \\
			& Quantile Network     & 0.28 & 0.22 & 0.17 & 0.24 & 0.39 \\
			& Quantile Spline      & * & * & * &* & * \\
			& Quantile Forest      & 17.28 & 7.03 & 3.70 & 5.77 & 18.50 \\
			\bottomrule
		\end{tabular}
	\end{small}
	\caption{\label{tab2} Univariate Responses Tasks: Performance of different methods in Scenarios 4 and 5, in terms of squared error from the true quantile averaged over 25 independent trials.}
\end{table}

We visualize the performances different approaches and true quantile functions in Figures \ref{fig2}--\ref{fig3}. There, we can see that  Quantile Network is a better estimate for the quantile functions in general, compared to the other methods.

We report the results for Scenarios 1-3 in Table \ref{tab1} and Scenarios 4-5 in Table \ref{tab2}. From Table \ref{tab1}  we can see that in Scenario 1,  the Quantile Network method outperforms the competitors for  most  quantiles and  sample sizes. The advantage becomes more evident as the sample size grows. The closest competitor is Quantile Spline  which  is the best method in some small sample problems.  Furthermore, in Scenario 2  the best  method is Quantile Spline, with Quantile Network as second best. This is not surprising since  Scenario 2  consists of very smooth quantile functions defined in a low dimensional domain ($d=2$).  In contrast, Scenario 3 consists of  a  quantile function with discontinuities and heteroscedastic errors. In this more challenging setting, Quantile Network outperforms others for larger values of $n$.


We do not compare against Quantile Splines in Scenarios 4--5  as such method does not scale up to 5 dimensional problems or above. In Table \ref{tab2}, we show that for  Scenarios 4--5   the clear best method is  Quantile Network (for $n>100$), with Quantile Forests as the second best. 



Overall, the results in Tables \ref{tab1}-\ref{tab2} demonstrate a clear advantage of the Quantile Network method. This method generally outperforms SqErr Network in all examples, presumably due to the heavy-tailed  or heteroscedastic error distributions. At the same time,  Quantile Network  also outperforms the other competitors with larger sample size or more complicated quantile functions.

\subsection{Multivariate response}  

We explore the performance of different quantile ReLU network approaches for multivariate responses as discussed in Section \ref{sub_sec:multivariate}.  We refer to  the estimator in equation (\ref{eqn:robust}) as Geometric Quantile, and the estimator in equation (\ref{eqn:qunatiles2}) as Quantile Network. As a benchmark, once again, we consider the estimator based on the squared error loss and as defined in equation (\ref{eqn:ls}). For all these estimators, the corresponding  network class is chosen as in \cref{sec:uni_simulations}.

We  conduct simulations in two different scenarios. In each scenario, we evaluate performance based on mean squared error defined as
\[
\frac{1}{np}\sum_{j=1}^{p}\sum_{i=1}^n  \left(    f_{\tau,j }^*(x_i) -  \hat{f}(x_i)\right)^2,
\]
where  the quantile functions   $  f_{\tau,j }^*(\cdot)$, $j=1,\ldots,p$, are defined in (\ref{eqn:marginal}), and $\tau= 0.5$.

We consider two multivariate response generative models as the follows. 

\paragraph{Scenario 6. }    
\[
\begin{array}{lll}
y_i &= &   g_2 \circ g_1 (x_i)   + \epsilon_i\\
g_1(q)   &  = &  (\vert q_1\vert,q_2 \cdot q_1  )^{\top} \\
g_2(q) & =& (   \sqrt{q_2^2 + q_1 },  (q_1+q_2)^3 )^{\top}\\
x_i  & \overset{\text{ind}}{\sim} &     U[0,1]^2,   \,\,\, i = 1,\ldots,n, \\
\epsilon_i  &\overset{\text{ind}}{\sim} &  Mt_3(0,  I_2),  \,\,\, i = 1,\ldots,n, 
\end{array}
\]
where $Mt_3(0,  I_2)$  is the multivariate $t$-distribution with  $3$ degrees of freedom  and scale matrix identity  $I_2 \in \mathbb{ R}^{2\times 2}$.

\paragraph{Scenario 7. }   
\[
\begin{array}{lll}
y_i &= &   f_0(x_i)   + \epsilon_i\\
f_0(q)   &  = &  (  \sqrt{q_1^2+ q_2^2},\sqrt{q_3^2+q_4^2}    )^{\top}\\
x_i  &\overset{\text{ind}}{\sim}&     U[0,1]^4,\\
\epsilon_{i,j}  & \overset{\text{ind}}{\sim} &  \mathrm{Laplace}(0,2).
\end{array}
\]

\begin{table}[t!]
	\centering
	\begin{small}
		\begin{tabular}{llc}
			\toprule
			\multicolumn{3}{c}{\textbf{Scenario 6}} \\
			\midrule
			n & Method & $\tau=0.50$ \\
			\midrule
			\multirow{3}{*}{$100$} & SqErr Network                       & 0.92 \\
			& Quantile Network                    & 0.49 \\
			& Geometric Quantile        & 0.54 \\
			
			\midrule
			\multirow{3}{*}{$1000$} & SqErr Network                       & 0.91 \\
			& Quantile Network                    & 0.06 \\
			&Geometric Quantile        & 0.07 \\
			\midrule
			\multirow{3}{*}{$10000$} & SqErr Network                       & 0.89 \\
			& Quantile Network                    & 0.01 \\
			& Geometric Quantile        & 0.01 \\
			
			\midrule
			\multicolumn{3}{c}{\textbf{Scenario 7}} \\
			\midrule
			n & Method & $\tau=0.50$ \\
			\midrule
			\multirow{4}{*}{$100$} & SqErr Network                       & 0.24 \\
			& Quantile Network                    & 1.67 \\
			& Geometric Quantile       & 2.36 \\
			
			\midrule
			\multirow{4}{*}{$1000$} & SqErr Network                       & 0.11 \\
			& Quantile Network                    & 0.23 \\
			& Geometric Quantile       & 0.14 \\
			
			\midrule
			\multirow{4}{*}{$10000$} & SqErr Network                       & 0.09 \\
			& Quantile Network                    & 0.03 \\
			& Geometric Quantile      & 0.03 \\

			\bottomrule
		\end{tabular}
	\end{small}
	\caption{\label{tab3} Multivariate Responses Tasks: Performance of different methods in Scenarios 6 and 7.}
	\small
	Note: We measure performance using the averaged mean squared error based on 25 Monte Carlo simulations for the two synthetic multivariate benchmarks.
\end{table}

Table \ref{tab3}  illustrates the performance of different methods in Scenarios 6 and 7 with sample sizes $n=\{100,1000,10000\}$. We can see that both  Quantile Network and Geometric  Quantile 
outperform the  $\ell_2$-based approach SqErr Network. This corroborates the results earlier in Section \ref{sec:uni_simulations}. Particularly, it demonstrates that both Quantile Network and Geometric Quantile are robust estimators when dealing with heavy-tailed distributions.

\section{Conclusion and Future Work}	
	
	In this paper  we have  studied, both theoretically and empirically, the  statistical performance   of ReLU networks for  quantile regression. Our main theorems  establish  minimax estimation rates under general classes of functions and distributions of the errors. These results 
rely on the approximation theory from \cite{schmidt2017nonparametric} and \cite{suzuki2018adaptivity}. Future work can extend these results to other function classes provided that  the corresponding  approximation theory is used or developed. 
Empirically, experiments in both univariate and multivariate  response quantile regression with ReLU networks show an advantage over other  quantile regression methods. Quantile regression networks were also shown to outperform $\ell_2$-based regression with the same neural network architecture when the error distribution is heavy-tailed. In the case of multivariate responses, quantile networks were also shown to perform well in benchmarks. A theoretical study of statistical rates of convergence for multivariate response neural quantile regression is left for future work.

		
		
			\appendix
\section{Notation}

For  an $\epsilon>0$ and  a metric  $\mathrm{dist}(\cdot,\cdot)$ on the class of functions  $\mathcal{F}$, we define the covering number $\mathrm{N}(\epsilon,\mathcal{F},\mathrm{dist}(\cdot,\cdot))$ as the minimum number of  balls of the form $\{ g \,:\,  \mathrm{dist}(g,f)\leq \epsilon  \}$,  with $f \in \mathcal{F}$,  needed to  cover  $\mathcal{F}$.

We also write
\[
\text{B}(f, \|\cdot\|_{\ell_2},   r ) \,=\,   \,\{   g   \,:\,    \|f-g\|_{\ell_2}  \leq r      \}.
\]
Furrthermore,  if  $a_n$  and $b_n$ are positive sequences, we say that  $a_n \lesssim  b_n$ if there exists $m$ such that $n\geq m$ implies  $a_n \leq  c b_n$ for a constant $c>0$.

\section{Theorem  \ref{thm:risk} }

Througouth this   section we write $\mathcal{F}$ to refer to $\mathcal{F}(W,U,L)$.

\subsection{Auxiliary results }

Before  stating our first result we first state some definitions and an auxiliary lemma.

\begin{definition}
	\label{def1.2}
	We define the empirical loss function \[
	\displaystyle	\hat{M}_n(\theta) =    \sum_{i=1}^{n} \hat{M}_{n,i}(\theta),
	\]
	where
	\[
	\hat{M}_{n,i}(f) = \frac{1}{n}  \left(\rho_{\tau}(y_i -  f(x_i)    )    -  \rho_{\tau}(y_i -   f_n(x_i)   )\right),
	\]
	with 
	\[
	f_n  \,\in \,  \underset{f \in \mathcal{F}   }{\arg \min}\,\,\,   \mathbb{E }\left[  \frac{1}{n} \sum_{i=1}^{n}   \rho_{\tau}(y_i -  f(x_i)    )     -   \frac{1}{n} \sum_{i=1}^{n}  \rho_{\tau}(y_i -  f_{\tau}^*(x_i)    )  \right].
	\]
	We also set
	$$M_{n,i}(f) =   \frac{1}{n} \mathbb{E }[\rho_{\tau}(z_i -  f(x_i)   )   -  \rho_{\tau}(z_i -  f_n(x_i)   ) ],$$ 
	where  $z \in \mathbb{ R}^n$ is an independent  copy of  $y$.  
\end{definition}

\begin{lemma}
	\label{lem1}
	Suppose  that Assumption \ref{as1}--\ref{as2} hold. Then there exists  a constant   $c_{\tau}$ such that  for all  $\delta \in \mathbb{ R}^n$, we have
	\[
	\displaystyle    \sum_{i=1}^{n } \mathbb{E}\left[\rho_{\tau}(z_i -  f_{\tau}^*(x_i)   -\delta_i ) \right]  -  \sum_{i=1}^{n } \mathbb{E}\left[\rho_{\tau}(z_i -  f_{\tau}^*(x_i)    ) \right]       \geq  c_{\tau}  n \Delta_n^2(\delta),
	\]
	where $z \in \mathbb{ R}^n$ is an independent  copy of  $y$.
\end{lemma}

\begin{proof}
	See Lemma 8 in \cite{padilla2020adaptive}.
\end{proof}

\begin{definition}
	\label{def2}
	Let $\mathcal{H}$ be   a class of
	functions from $\mathcal{X}$ to $\mathbb{ R}$. We  define the  pseudodimension of  $\mathcal{H}$, denoted as  $\mathrm{Pdim}(\mathcal{H})$, as  the  largest integer  $m$ for  which there exist   $(a_1,b_1)\ldots,(a_m,b_m) \in \mathcal{X} \times \mathbb{ R}$   such  that for all  $\eta \in \{0,1\}^m$ there exists  $f \in \mathcal{H}$  such that
	\[
	f(a_i)  >  b_i\,\,\,\,\iff\,\,\,\,    \eta_i,
	\]
	for  $i=1 ,\ldots,m$.
\end{definition}

\begin{theorem}[Theorem  7 from  \cite{bartlett2019nearly}]
	\label{thm2}
	With the notation from before, we have that
	\[
	\mathrm{Pdim}\left(\mathcal{F}(W,U,L)  \right)	  = O(LW \log( U)    ).
	\]
\end{theorem}

\subsection{Proof of  Theorem  \ref{thm:risk} }

\begin{proof}
	Throughout this proof the  covariates  $x_1,\ldots,x_n$  are  fixed.  Let $\hat{\delta}_i =  \hat{f}(x_i) - f_{\tau}^*(x_i)$  for $i=1,\ldots,n$.
	Notice that
	\begin{equation}
		\label{eqn:first}
		\begin{array}{lll}
			\displaystyle 	    \mathbb{E } \left[ \frac{1}{n}\sum_{i=1}^n D^2\left\{  f_{\tau}^*(x_i) - \hat{f}(x_i)   \right\}\right]   & \leq& \displaystyle  \frac{1}{c_{\tau} n}\mathbb{E } \left( \sum_{i=1}^{n } \mathbb{E } \left[\rho_{\tau}\{z_i -  f_{\tau}^*(x_i)   -\hat{\delta}_i \} \right]  -  \sum_{i=1}^{n } \mathbb{E } \left[\rho_{\tau}\{z_i -  f_{\tau}^*(x_i)    \} \right]    \right)\\
			&=&  \displaystyle \frac{1}{c_{\tau}}  E\left\{  M_n(\hat{f})    \right\}   + \frac{1}{c_{\tau}}   \text{err}_1,
		\end{array}
	\end{equation}
	where the inequality follows from Lemma \ref{lem1}.
	
	Next	we proceed  to bound  $  E\{  M_n(\hat{f})  \}$. To that end, notice that for a constant  $C>0$,
	\[
	\begin{array}{lll}
		\mathbb{E } \left\{  M_n(\hat{f})    \right\}  & \leq &\displaystyle  4 \mathbb{E } \left\{ \underset{f\in   \mathcal{F}    }{\sup}\,    \frac{1}{n}\sum_{i=1}^{n}  \xi_i  f(x_i)   \right\} \\
		& \leq& \displaystyle   F \mathbb{E } \left\{ \underset{f \in   \mathcal{F}    }{\sup}\,    \frac{1}{n}\sum_{i=1}^{n}  \xi_i  \frac{f(x_i)}{F}   \right\} \\
		& \leq& \displaystyle \frac{C F}{  \sqrt{n} }\,\int_{0}^{2}   \sqrt{  \log  \mathrm{N}\left( \mu, \mathcal{F}/F, \|\cdot\|_n   \right)  } \,d\mu\\
		& \leq&  \displaystyle \frac{C F}{\sqrt{n} }\, \int_{0}^{2}   
		\sqrt{  \log\left(    \left(     \frac{  2 \cdot e\cdot n  }{\mu  \cdot \mathrm{Pdim}\left(\mathcal{F}  \right)   }    \right)^{  \mathrm{Pdim}\left(\mathcal{F}  \right) }\right)  }\,d\mu
		\\
		& \leq&  \displaystyle \tilde{C} F   \sqrt{  \frac{   LW \log  U  \cdot \log n  }{n}  }
	\end{array}
	\]
	for some constant  $\tilde{C}>0$,
	where the first inequality follows by simmetrization and  Talgrand's inequality (\cite{ledoux2013probability}) similarly to  Theorem 12 in \cite{padilla2020adaptive},  the  third inequality follows from  Dudley's theorem,  the  fourth  holds because of Lemma 4 in \cite{farrell2018deep}, and the last from Theorem \ref{thm2}.
	
\end{proof}





\section{Theorem \ref{thm3}   }

The proof is in the spirit of the proof  of Theorem 1 in \cite{farrell2018deep} combined with  results and ideas from \cite{padilla2020adaptive} and \cite{suzuki2018adaptivity}.
\subsection{Notation}

Througout we  let  $\mathcal{X} = [0,1]^d$.  For  $p > 0$ and  $f \,:\, \mathcal{X} \,\rightarrow\, \mathbb{R}$ we let 
\[
\displaystyle   \| f\|_{p} \,:=\,   \left(   \int_{\mathcal{X}}    (f(x))^p dx\right)^{1/p}
\]
and
\[
L^p(\mathcal{X}) \,=\,\left\{  f\,:\, f \,:\, \mathcal{X} \,\rightarrow\, \mathbb{R},\,\,\,\,\text{and}\,\,\,\, \| f\|_{p} <\infty    \right\}.
\]	

\begin{definition}
	\label{def3}
	For a  function $f \in L^p(\mathcal{X})$  and  $p \in (0,\infty]$  we define the $r$-modulus of continuity  as
	\[
	w_{r,p}(f,t) = \underset{\|   h\|_2  \leq  t    }{\sup}\,\|  \Delta_h^r(f)\|_p,
	\]
	with 
	\[
	\Delta_h^r(f) =   \begin{cases}
		\sum_{j=0}^{r}  \frac{r!}{j! \,(j-r)!}        (-1)^{r-j} f(x+hj)    &  \text{if}\,\,\,\,\,      x\in \mathcal{X},  \,\,\, x+rh \in \mathcal{X},\\
		0 & \text{otherwise}.
	\end{cases}
	\]
\end{definition}

\begin{definition}
	\label{def4}
	For   $0<p,q\leq \infty$,  $\alpha>0$,   $r =  \floor{\alpha} +1$, we define  the Besov space $B^{\alpha}_{p,q}(\mathcal{X})$ as
	\[
	B^{\alpha}_{p,q}(\mathcal{X})    =    \left\{   f \in  L^p(\mathcal{X})\,:\,      \|f\|_{ 	B^{\alpha}_{p,q}(\mathcal{X})    }    < \infty 	\right\},
	\]
	where 
	\[
	\|f\|_{ 	B^{\alpha}_{p,q}(\mathcal{X})    }   =  \|f\|_{ p  }  +      \vert f\vert_{ 	B^{\alpha}_{p,q}(\mathcal{X})    } , 
	\]
	with
	\[
	\vert f\vert_{ 	B^{\alpha}_{p,q}(\mathcal{X})    }     =      \begin{cases}
		\left(\int_{0}^{\infty}     (t^{-\alpha} w_{r,p} (f,t))^q t^{-1} dt\right)^{\frac{1}{q}}       &\,\,\,\text{if } \,\,\,\,  q<\infty,\\
		\underset{t>0}{\sup}  \,\,\, t^{-\alpha} w_{r,p}(f,t)  &\,\,\,\text{if } \,\,\,\,  q=\infty.
	\end{cases}
	\]
\end{definition}


Throughout we denote  by $f_n \in \mathcal{I}(L,W,S,B)$  a function satisfying 
\[
f_n \in      \underset{f \in \mathcal{I}(L,W,S,B) ,  \,  \|f\|_{\infty} \leq F  }{\arg\min}  \, \|f - f_{\tau}^*\|_{\infty}.
\]

We also write
\[
\tilde{\mathcal{I}}(L,W,S,B)  = \left\{f \in \mathcal{I}(L,W,S,B) \,:\,   \,  \|f\|_{\infty} \leq F \right\}.
\]

\subsection{Auxiliary lemmas}

\begin{lemma}
	\label{lem2}
	Suppose  that  $\|  f_n -f_{\tau}^*\|_{\infty} \leq  c$ for a small enough constant $c$.  With the  notation in (\ref{eqn:loss2}) we have that
	\[
	\Delta^2(f,f_n)  \leq   \frac{1}{c_{\tau}} \left[\mathbb{E}\left(   \rho_{\tau}(Y- f(X))  -  \rho_{\tau}(Y- f_n(X))    \right)+    \| f_n-f_{\tau}^*\|_{\infty}  \Delta(f,f_n ) \sqrt{F  }    \right],
	\]
	and
	\[
	\|f-f_n\|_{\ell_2}^2  \leq   \frac{2F}{c_{\tau}} \left[\mathbb{E}\left(   \rho_{\tau}(Y- f(X))  -  \rho_{\tau}(Y- f_n(X))    \right)+    \| f_n-f_{\tau}^*\|_{\infty}  \|f-f_n\|_{\ell_2}   \sqrt{F  }    \right],
	\]
	for any $f \in \tilde{\mathcal{I}}(L,W,S,B)$   and for some constant $c_{\tau}$. 
\end{lemma}

\begin{proof}
	Notice that  by Equation   B.3 in  \cite{belloni20111},
	\[
	\begin{array}{lll}
		\rho_{\tau}(Y -  f(X)  )-\rho_{\tau}(Y -  f_n(X)  )& =& \displaystyle  -(f(X) -  f_n(X)) (\tau -  1{   \{Y\leq f_n(X)\} } ) +    \\
		&&\displaystyle \int_{0}^{ f(X)-f_n(X)  }   \left[   1\{ Y\leq f_n(X) +z\}  -1{\{ Y\leq f_n(X) \} }\right]dz\\
		& = & \displaystyle  -(f(X) -  f_n(X)) (\tau -  1{   \{Y\leq f_{\tau}^*(X)\}  } )      - \\
		& &\displaystyle (f(X) -  f_n(X)) (1{  \{Y\leq f_{\tau}^*(X)\}  }  -  1{   \{Y\leq f_n(X)\} } ) + \\
		& &\displaystyle \int_{0}^{  f(X)-f_n(X)  }   \left[   1\{ Y\leq f_n(X) +z\}  -1{\{ Y\leq f_n(X) \} }\right]dz.\\
	\end{array}
	\]
	Hence, taking expectations and using Fubini's theorem,
	\[
	\begin{array}{lll}
		\mathbb{E}\left( \rho_{\tau}(Y -  f(X)  )-\rho_{\tau}(Y -  f_n(X)  )\right)& =&\displaystyle   \mathbb{E}\left(    -(f(X) -  f_n(X)) \mathbb{E}\left((\tau -  1{   \{Y\leq f_{\tau}^*(X)\}  } ) \bigg|   X\right)  \right)     - \\
		& &\displaystyle \mathbb{E}\left ( (f(X) -  f_n(X))\mathbb{E}\left(   (1{  \{Y\leq f_{\tau}^*(X)\}  }  -  1{   \{Y\leq f_n(X)\} } )  \bigg|    X\right)        \right)+ \\
		& &\displaystyle     \mathbb{E}\bigg(       \int_{0}^{  f(X)-f_n(X)  }   \bigg[   \mathbb{E}\left(1\{ Y\leq f_n(X) +z\}\bigg|  X \right) -  \\
		& & \,\,\,\,\,\,\,\,\,  \mathbb{E}\left(1\{ Y\leq f_n(X) \}\bigg|  X \right) \bigg]dz \bigg)\\
		& \geq & -c_1    \mathbb{E}\left[     \vert f(X) -  f_n(X)  \vert \cdot\vert  f_{\tau}^*(X) -  f_n(X)   \vert  \right] + \\
		& & c_{\tau} \mathbb{E}(        D^2(f(X)- f_n(X)  )      )\\
		& \geq&-c_1   \sqrt{\mathbb{E}\left[     \vert f(X) -  f_n(X)  \vert^2 \right]} \sqrt{\mathbb{E}\left[     \vert f_{\tau}^*(X) -  f_n(X)  \vert^2  \right]} +\\
		&&c_{\tau} \mathbb{E}(        D^2( f(X)- f_n(X)  )   )\\
		& \geq&-c_1   \| f_n-f_{\tau}^*\|_{\infty}   \sqrt{F  \Delta^2(f,f_n ) }\\
		&&c_{\tau} \mathbb{E}(        D^2( f(X)- f_n(X)  )      )\\
	\end{array}
	\]
	for a constant $c_1>0$, where the first inequality holds
	since the cumulative distribution function of  $Y$ conditioning on $X$  is Lipchitz around $f_{\tau}^*(X)$  by Assumption \ref{as2}, and by the same argument from the proof of Lemma 8  in \cite{padilla2020adaptive}. 
	
\end{proof}

\begin{definition}
	\label{def7}
	
	We define the empirical loss function \[
	\displaystyle	\hat{M}_n(\theta) =    \sum_{i=1}^{n} \hat{M}_{n,i}(\theta),
	\]
	where
	\[
	\hat{M}_{n,i}(f) =  \frac{1}{n}\left[\rho_{\tau}\{y_i -  f(x_i)    \}    -  \rho_{\tau}\{y_i -   f_n(x_i)   \}\right] ,
	\]
	and we set
	$$M_{n}(f) =      \mathbb{E}[\rho_{\tau}\{Y-  f(X)   \}   -  \rho_{\tau}\{Y -  f_n(X)   \} ].$$ 
\end{definition}

\begin{lemma}
	\label{lem3} 	Suppose  that  $\|  f_n -f_{\tau}^*\|_{\infty} \leq  c$ for a small enough constant $c$.  The estimator  $\hat{f}$   defined in  (\ref{eqn:estimator_2}) satisfies
	\[
	\Delta^2(\hat{f},f_n)  \leq    \frac{1}{c_{\tau}}\left[M_n(\hat{f})  -  \hat{M}_{n}(\hat{f})     +     \| f_n-f_{\tau}^*\|_{\infty}  \Delta(\hat{f},f_n ) \sqrt{F  }  \right].
	\]
	Furthermore,
	\[
	\|\hat{f}-f_n\|_{\ell_2}^2  \leq    \frac{2F   }{c_{\tau}}\left[M_n(\hat{f})  -  \hat{M}_{n}(\hat{f})     +     \| f_n-f_{\tau}^*\|_{\infty}\|\hat{f}-f_n\|_{\ell_2} \sqrt{F  }  \right].
	\]
\end{lemma}
\begin{proof}
	By Lemma  \ref{lem2} we have that
	\[
	\begin{array}{lll}
		\Delta^2(\hat{f},f_n)  &\leq  &  \displaystyle   \frac{1}{c_{\tau}} \left[\mathbb{E}\left(   \rho_{\tau}(Y- \hat{f}(X))  -  \rho_{\tau}(Y- f_n(X))    \right) +     \| f_n-f_{\tau}^*\|_{\infty}  \Delta(\hat{f},f_n ) \sqrt{F  }\right] \\ 
		&\leq & \displaystyle  \frac{1}{c_{\tau}} \bigg[  \mathbb{E}\left(   \rho_{\tau}(Y- \hat{f}(X))  -  \rho_{\tau}(Y- f_n(X))    \right) +     \| f_n-f_{\tau}^*\|_{\infty}  \Delta(\hat{f},f_n ) \sqrt{F  }-   \\
		&&	\displaystyle   \frac{1}{n}\sum_{i=1}^n  \rho_{\tau}(y_i -  \hat{f}(x_i))  +\sum_{i=1}^n  \rho_{\tau}(y_i -  f_n(x_i)) 
		\bigg],
	\end{array}
	\]
	where the last inequality follows by the optimality of $\hat{f}$.
\end{proof}

\begin{lemma}
	\label{lem4}
	Suppose that
	\begin{equation}
		\label{eqn:cond1}
		3\mathbb{E}\left(\underset{f \in   \tilde{\mathcal{I}}(L,W,S,B), \, \, \|f-f_n\|^2_{\ell_2}\leq r^2 }{\sup}\,  \frac{1}{n}\sum_{i=1}^{n}\xi_i ( f(x_i)-  f_n(x_i))^2    \right) \leq  r^2, 
	\end{equation}
	for  $\{\xi_i\}_{i=1}^n$  Rademacher variables independent of $\{(x_i,y_i)\}_{i=1}^n$, 	and 
	\begin{equation}
		\label{eqn:cond2}
		\max\left\{ 4F\sqrt{   \frac{ \gamma}{n} } , 4F   \sqrt{\frac{\gamma }{3n}   }\right\}  \leq r.
	\end{equation}
	Then with probability at least  $1- e^{-\gamma}$,  
	$\|f-f_n\|_{\ell_2}^2 \leq   r^2$  with  $f  \in \tilde{\mathcal{I}}(L,W,S,B)$ implies
	\[
	\|f-f_n\|_n^2 \,\leq  \,  (2r)^2.
	\]
\end{lemma}

\begin{proof}
	First notice that 
	\[
	\vert( f(x)-f_n(x))^2 \vert  \leq   2\left[F^2 +   \|f_n\|_{\infty}^2 \right],
	\]
	for all $x$. Hence, 
	\[
	\begin{array}{lll}
		\mathbb{E}\left(  \left(  f(X)-f_n(X)\right)^4  \right) 
		&\leq & 2(F^2   +    \|f_n\|_{\infty}^2 ) \mathbb{E}\left(  \left(f(X)-f_n(X)\right)^2  \right) \\
		& \leq& 4 F^2\|f-f_n\|_{\ell_2}^2 
	\end{array}
	\]
	
	Then by Theorem 2.1 in \cite{bartlett2005local} with probability at least $1-\exp(-\gamma)$,
	\begin{equation}
		\begin{array}{l}
			\underset{f \in   \tilde{\mathcal{I}}(L,W,S,B),  \,\,   \|f-f_n\|_{\ell_2}^2\leq r^2}{\sup}\, \left\{\|f-f_n\|_n^2 - \|f-f_n\|_{\ell_2}^2 \right\}\\
			\leq    \displaystyle 3\mathbb{E}\left(\underset{f \in   \tilde{\mathcal{I}}(L,W,S,B),  \,\, \|f-f_n\|_{\ell_2}^2\leq r^2 }{\sup}\,  \frac{1}{n}\sum_{i=1}^{n}\xi_i ( f(X_i)-  f_n(X_i))^2    \right)\\
			\displaystyle   +    4rF\sqrt{   \frac{ \gamma}{n} } +  \frac{16F^2   \gamma  }{3n}\\
		\end{array}
	\end{equation}
	and the claim  follows.
\end{proof}

\begin{lemma}
	\label{lem5}
	Suppose that $\| \hat{f}-f_n\|_{\ell_2}     \leq  r_0$, with  $r_0$  satisfying (\ref{eqn:cond1})-(\ref{eqn:cond2}) and Assumption \ref{as4} holds. Also, with the notation of Assumption \ref{as3}, suppose that  for the class $\mathcal{I}(L,W,S,B)$  the parameters are chosen as
	\[
	\begin{array}{l}
		L =  3+ 2\ceil{    \log_2\left(    \frac{3^{  \max\{d,m\}  }}{  \epsilon  c_{d,m} }\right)  +5 }\ceil{  \log_2 \max\{d,m\} },\,\,\,\,\,\,\, W = W_0 N,\\
		S = (L-1)W_0^2 N +N,\,\,\,\,\,\, B = O\left( N^{  (v^{-1}+   d^{-1})(  \max\{ 1,   (d/p-s)_{+}\} )   }\right),
	\end{array}
	\] 
	for  a constant $c_{d,m} $  that depends on $d$ and $m$, a constant $W_0$,  and where $v= (s-\delta)/\delta$,
	\[
	\delta = \frac{d}{p},\,\,\,\,\,\,\, \,\,\,\,\,\,\,  N \asymp   n^{  \frac{d}{2s+d} }.
	\]
	Then for some positive constant $C_0$ it holds that 
	\[
	\begin{array}{lll}
		\,	\|\hat{f}-f_n\|^2_{\ell_2}    & \leq&  C_0 \bigg[   r_0F^{2.5}\sqrt{   \frac{ \gamma}{n} } +  \frac{ F^{2.5}   \gamma  }{n}  + \\
		& &    r_0F \sqrt{   \frac{N  (\log N)^2 }{n}  } +       r_0F\sqrt{\frac{N \left[(\log N)^2+  \log r_0^{-1} + \log n\right] }{  n }        }  + N^{-s/d} r_0 F^{1.5}\bigg]
	\end{array}
	\]
	with probability at least $1-\exp(-\gamma)$, where  $N \asymp   n^{  \frac{d}{2s+d}  }$.
\end{lemma}
\begin{proof}
	Let\[
	\mathcal{G} = \left\{    g \,:\,    g(x,y) = \rho_{\tau}(y-f(x)) - \rho_{\tau}(y-f_n(x))  ,\,\,\,    f\in \tilde{\mathcal{I}}(L,W,S,B),\,\, \|f-f_n\|_{\ell_2}   \leq  r_0\right\}.
	\]
	Then for $\xi_1,\ldots,\xi_n$  independent Rademacher  variables independent of $\{(x_i,y_i)\}_{i=1}^n$, we have that
	\begin{equation}
		\label{eqn:e1}
		\begin{array}{lll}
			c_{\tau}\,	\|\hat{f}-f_n\|_{\ell_2}^2  &\leq & 2F[M_n(\hat{f})  -  \hat{M}_{n}(\hat{f})      +     \| f_n-f_{\tau}^*\|_{\infty}\|\hat{f}-f_n\|_{\ell_2}  \sqrt{F } ]\\
			& \leq &  \displaystyle   2F\underset{ g\in  \mathcal{G}  }{\sup}  \left\{ \mathbb{E}(g(X,Y)) - \frac{1}{n}\sum_{i=1}^{n}g(x_i,y_i)     \right\}  +      \| f_n-f_{\tau}^*\|_{\infty} \|\hat{f}-f_n\|_{\ell_2}  F^{3/2}\\
			& \leq&   \displaystyle 12F\mathbb{E}\left(    \underset{g \in    \mathcal{G} }{\sup} \,    \frac{1}{n} \sum_{i=1}^{n}\xi_i  g(x_i,y_i)  \bigg| (x_1,y_1),\ldots,(x_n,y_n) \right)\\
			& &\displaystyle +  4 r_0F^{2.5}\sqrt{   \frac{\gamma}{n} } +  \frac{100 F^{2.5}   \gamma  }{3n} +     \| f_n-f_{\tau}^*\|_{\infty} \|\hat{f}-f_n\|_{\ell_2}  F^{1.5},\\
		\end{array}\, 
	\end{equation}
	where  the first inequality  follows from   Lemmas \ref{lem3}, and the third happens with probability at least $1-e^{-\gamma}$ and holds by 
	Theorem 2.1 in \cite{bartlett2005local}.

	Next, notice that for a constant  $C>0$,
	\begin{equation}
		\label{eqn:entropy}
		\begin{array}{lll}
			\displaystyle  \mathbb{E}_{\xi} \left(    \underset{g \in    \mathcal{G} }{\sup} \,    \frac{1}{n} \sum_{i=1}^{n}\xi_i  g(x_i,y_i)   \right) & \leq &\displaystyle  \mathbb{E}_{\xi} \left(    \underset{f \in    \mathcal{I}(L,W,S,B)  ,\,\,\|f\|_{\infty} \leq  F\,\,\|f-f_n\|_{\ell_2}  \leq  r_0    }{\sup} \,    \frac{1}{n} \sum_{i=1}^{n}   \xi_i (  f(x_i) - f_n(x_i)  ) \right)\\
			& \leq & \displaystyle  \mathbb{E}_{\xi} \left(    \underset{f \in    \mathcal{I}(L,W,S,B)  ,\,\,\|f\|_{\infty} \leq  F,\,\,  \|f_n-f\|_n  \leq  2r_0    }{\sup} \,    \frac{1}{n} \sum_{i=1}^{n}   \xi_i (  f(x_i) - f_n(x_i)  ) \right)\\
			& \leq&  \displaystyle     \underset{0 < \alpha < 2r_0  }{\inf} \left\{   4\alpha  +      \frac{12}{  \sqrt{n} }   \int_{\alpha}^{2  r_0 } \sqrt{   \log  \mathcal{N}(     \delta, \tilde{\mathcal{I}}(L,W,S,B), \|\cdot\|_n  )   }  d\delta\right\}\\
			& \leq&   \displaystyle \underset{0 < \alpha < 2r_0  }{\inf} \left\{   4\alpha  +      \frac{12}{  \sqrt{n} }   \int_{\alpha}^{2  r_0 } \sqrt{   \log  \mathcal{N}(     \delta, \tilde{\mathcal{I}}(L,W,S,B), \|\cdot\|_{\infty}  )   }  d\delta\right\}\\
			& \leq &   \displaystyle \underset{0 < \alpha < r_0  }{\inf} \left\{   4\alpha  +  \frac{24   r_0 }{  \sqrt{n} }\sqrt{   \log  \mathcal{N}(    \alpha, \tilde{\mathcal{I}}(L,W,S,B), \|\cdot\|_{\infty}  )   }   \right\},\\
			& \leq &   C\displaystyle \underset{0 < \alpha <  r_0  }{\inf} \left\{   \alpha  +  r_0  \sqrt{\frac{N \left[(\log N)^2+    \log \alpha^{-1}\right] }{  n }        }\right\},\\
		\end{array}
	\end{equation}
	where the first  inequality  follows 	by Talagrand's inequality \citep{ledoux2013probability}, and the second holds with probability at least  $1-\exp(-\gamma)$ by Lemma \ref{lem4},  the third  by Dudley's chaining inequality,  and the last  by the proof Theorem 2 in \cite{suzuki2018adaptivity} . The latter  theorem also gives 
	$N  \asymp n^{  \frac{d}{2s+d} }$, and
	\[
	\|f_n-f_{\tau}^*\|_{\infty}   \leq  C_1 N^{-s/d},
	\]
	for some constant  $C_1>0$.
	Hence,   taking  
	\[
	\alpha =    r_0 \sqrt{   \frac{  N  (\log N)^2 }{n}  },
	\]
	(\ref{eqn:e1}) and (\ref{eqn:entropy}) imply
	\[
	\begin{array}{lll}
		\,	\|\hat{f}-f_{\tau}^*\|_{\ell_2}^2    & \leq&  \frac{1}{ c_{\tau}}\bigg[   4r_0F^{2.5}\sqrt{   \frac{ \gamma}{n} } +  \frac{100 F^{2.5}   \gamma  }{3n}  + \\
		& &  C r_0 F\sqrt{   \frac{ N  (\log N)^2 }{n}  } +     C r_0F\sqrt{\frac{N \left[(\log N)^2+  \log r_0^{-1} + \log n\right] }{  n }        }+ C_1 r_0  N^{-s/d}F^{1.5} \bigg]
	\end{array}
	\]
	with probability at least $1-2e^{-\gamma}$.
\end{proof}

\begin{lemma}
	\label{lem6}
	Let $r^*$ be defined as
	\[
	r^* = \inf \left\{ r>0\,:\,  	3\mathbb{E}\left(\underset{f \in   \tilde{\mathcal{I}}(L,W,S,B),  \,\,  \|f-f_n\|_{\ell_2}  \leq s }{\sup}\,  \frac{1}{n}\sum_{i=1}^{n}\xi_i ( f(x_i)-  f_n(x_i))^2   \right)   <s^2 ,\,\,\forall s \geq r   \right\},
	\]
	for  $\{\xi_i\}_{i=1}^n$  Rademacher variables independent of $\{(x_i,y_i)\}_{i=1}^n$.
	Then  under the conditions of Lemma \ref{lem5},
	\[
	r^*  \leq    \displaystyle    \tilde{C} \left[  \sqrt{   \frac{ N  (\log N)^2 }{n}  } +      \sqrt{\frac{N \left[(\log N)^2+ \log n\right] }{  n }        } \right],  
	\]
	for a constant $\tilde{C}>0$ and with $N$  satisfying  $N  \asymp n^{  \frac{d}{2s+d} }$.
\end{lemma}

\begin{proof}
	Consider the set \[
	\mathcal{G}_{r^*} =  \left\{  f \in    \tilde{\mathcal{I}}(L,W,S,B)\,:\,    \|f-f_n\|_{\ell_2}^2 \leq  (r^* )^2 \right\},
	\]
	and the define the event
	\[
	E    \,=\,\left\{     \underset{f \in   \mathcal{G}_{r^*}  }{\sup}  \|f-f_n\|_{n}^2 \leq   (2 r^*)^2 \right\}.
	\]
	If  $r^*$ satisfies (\ref{eqn:cond2})  with $\gamma= \log n$   then we have that  $\mathbb{P}(E) \geq  1-1/n$  by Lemma \ref{lem4}.
	
	Also,
	\begin{equation}
		\label{eqn:bound}
		\begin{array}{lll}
			\displaystyle  (r^*) ^2  & \leq &  \displaystyle 3\mathbb{E}\left(\underset{f \in  \tilde{ \mathcal{I}}(L,W,S,B),  \,\, \|f-f_n\|_{\ell_2}^2\leq (r^*)^2 }{\sup}\,  \frac{1}{n}\sum_{i=1}^{n}\xi_i ( f(x_i)-   f_n(x_i))^2    \right)   \\ 
			& \leq& \displaystyle 3\mathbb{E}\left(   \mathbb{E }  \left(  \underset{f \in   \tilde{\mathcal{I}}(L,W,S,B),  \,\,\|f-f_n\|_{n}^2\leq (2r^*)^2 }{\sup}\,  \frac{1}{n}\sum_{i=1}^{n}\xi_i (f(x_i)-   f_n(x_i)  )^2   \bigg|  x_1,\ldots,x_n\right)  1_E \right)   \,+\,\\
			& &    \displaystyle \frac{12 F^2 }{n}\\
			& \leq& \displaystyle 3\mathbb{E}\left(   \mathbb{E }  \left(  \underset{f \in   \tilde{\mathcal{I}}(L,W,S,B),  \,\,\|f-f_n\|_{n}^2\leq (2r^*)^2 }{\sup}\,  \frac{1}{n}\sum_{i=1}^{n}\xi_i (f(x_i)-   f_n(x_i)  )^2   \bigg|  x_1,\ldots,x_n\right)  1_E \right)   \,+\,\\
			& &    \displaystyle \frac{12 F^2 }{n}.\\  
			& & 
		\end{array}
	\end{equation}
	Next, let
	\[
	\mathcal{G} = \left\{    g \,:\,    g(x) =   (f(x) -f_n(x))^2  ,\,\,\,\text{for some}\,\,    f\in \tilde{\mathcal{I}}(L,W,S,B),\,\,\,\,  \|f-f_n\|_{n}   \leq  2r^*\right\}.
	\]
	Notice that  if  $g_1,g_2\in \mathcal{G}$ with  $g_j =  f_j -f_n$,  $j=1,2$, then
	\[
	\vert g_1(x) -g_2(x) \vert \, = \, \vert   f_1(x) -f_2(x) \vert\cdot \vert   f_1(x) +f_2(x)   - 2f_n(x)\vert   \,\leq\, 4F\| f_1-f_2 \|_{\infty}.
	\]
	Hence,  combining this with (\ref{eqn:bound}), using Dudley's chaining  we obtain that
	\begin{equation}
		\label{eqn:ineq}
		\begin{array}{lll}
			\displaystyle  (r^*) ^2  & \leq &  \displaystyle  3  \mathbb{E}\left( \underset{0 < \alpha < 2r^*  }{\inf} \left\{   4\alpha  +      \frac{12}{  \sqrt{n} }   \int_{\alpha}^{2 r^* } \sqrt{   \log  \mathcal{N}(     \delta, \mathcal{G}, \|\cdot\|_n  )   }  \right\}\right) +\frac{12 F }{n}\\
			& \leq& \displaystyle  3  \mathbb{E}\left( \underset{0 < \alpha < 2r^*  }{\inf} \left\{   4\alpha  +      \frac{12}{  \sqrt{n} }   \int_{\alpha}^{2 r^* } \sqrt{   \log  \mathcal{N}(     \delta, \mathcal{G}, \|\cdot\|_{\infty}  )   }  \right\}\right) +\frac{12 F }{n}\\
			&\leq  &    C\displaystyle \underset{0 < \alpha <  2 r^*  }{\inf} \left\{   \alpha  + r^*  \sqrt{\frac{N \left[(\log N)^2+    \log \alpha^{-1}\right] }{  n }        }\right\} +   \frac{12 F }{n},
		\end{array}
	\end{equation}
	
	where the last inequality follows   from  Theorem 2 in \cite{suzuki2018adaptivity}, and with  $N$  satisfying  $N  \asymp n^{  \frac{d}{2s+d} }$. Hence,  if $r^*$  satisfies   (\ref{eqn:cond2})  then  for a constant $\tilde{C}>0$
	\[
	r^*  \leq    \displaystyle  \tilde{C} \left[  \sqrt{   \frac{ N  (\log N)^2 }{n}  } +     \sqrt{\frac{N \left[(\log N)^2+ \log n\right] }{  n }        } \right] ,  
	\]
	which follows from  (\ref{eqn:ineq}) by taking  
	$$\alpha =  r^*  \sqrt{   \frac{ N  (\log N)^2 }{n}  }. $$ The claim follows.
\end{proof}

\subsection{Proof of Theorem  \ref{thm3} }
\label{sec:actual_proof}

\begin{proof}
	Throughout we use the notation from  the proof of Lemma \ref{lem5}.
	Then we proceed as in \cite{farrell2018deep}. Specifically,  we divide the space $\tilde{\mathcal{I}}(L,W,S,B)$ into  sets of increasing radius
	\[
	\text{B}(f_n, \|\cdot\|_{\ell_2},   \overline{r}  )  , \, \text{B}(f_n,  \|\cdot\|_{\ell_2},  2 \overline{r}  )  \backslash\text{B}(f_n,  \|\cdot\|_{\ell_2},   \overline{r}  ) ,\ldots,  \text{B}(f_n,  \|\cdot\|_{\ell_2},  2^l \overline{r}  )  \backslash\text{B}(f_n,  \|\cdot\|_{\ell_2},   2^{l-1}\overline{r}  ), 
	\]
	where \[
	l  = \floor[\bigg]{  \log_2\left(    \frac{2F}{  \sqrt{(\log n )/n }   }  \right)  }.
	\]
	Next, if $\overline{r} >r^*$, the  by Lemma \ref{lem4},  with probability at least  $1-le^{-\gamma}$, we have that 
	\[
	\|f-f_n\|_{\ell_2}^2 \leq   2^j \overline{r}     \,\,\,\text{implies}\,\,\, \|f-f_n\|_{n}^2 \leq   2^{j+1} \overline{r}. 
	\]
	Then if  for some $j\leq  l$  it holds that
	\[
	\hat{f} \in  \text{B}(f_n,  \|\cdot\|_{\ell_2},  2^j \overline{r}  )  \backslash\text{B}(f_n,  \|\cdot\|_{\ell_2},   2^{j-1}\overline{r}  ),
	\]
	then by Lemma \ref{lem5}, with probability at least  $1-4e^{-\gamma}$, we have that 
	\[
	\begin{array}{lll}
		\|\hat{f}-f_n \|_{\ell_2}^2   & \leq& \tilde{C}\bigg[   2^j \overline{r}    F^{2.5} \sqrt{   \frac{ \gamma}{n} } +  \frac{ F^{2.5}   \gamma  }{n}  + \\
		& &   \cdot 2^j \overline{r} F\sqrt{   \frac{ N  (\log N)^2 }{n}  } +       \cdot2^j \overline{r}F\sqrt{\frac{N \left[(\log N)^2+ 2\log n\right] }{  n }        }    +     2^{j}   \overline{r}  N^{-s/d}  F^{1.5}\bigg]\\
		&\leq&2^{2j-2}  \overline{r}^2,
	\end{array}
	\]
	provided that 
	\[
	\tilde{C}\left[    F^{2.5} \sqrt{   \frac{ \gamma}{n} } + F \sqrt{   \frac{ N  (\log N)^2 }{n}  }   +  F\sqrt{\frac{N \left[(\log N)^2+ 2\log n\right] }{  n } }  + N^{-s/d} F^{1.5}\right] \leq    \frac{1}{8}   2^j\overline{r},
	\]
	and 
	\[
	\tilde{C}\frac{ F^{2.5}   \gamma  }{n}  \leq  \frac{1}{4}   2^{2j} \overline{r}^2,
	\]
	both of	which  for all  $j$  hold if 
	\begin{equation}
		\label{eqn:upper_bound}
		\overline{r}  =   	 	 8  \tilde{C}\left[    F^{2.5} \sqrt{   \frac{\gamma}{n} } + D\sqrt{   \frac{ N  (\log N)^2 }{n}  }   + D\sqrt{\frac{N \left[(\log N)^2+ 2\log n\right] }{  n } }  + N^{-s/d} F^{1.5}\right]  +   2\sqrt{	\frac{\tilde{C} F^{2.5}   \gamma  }{n}  }  + r^*.
	\end{equation}
	Therefore,  by Lemmas \ref{lem4}--\ref{lem5}, with probability at least  $1-e^{-\gamma}$, we have that 
	\[
	\|\hat{f}-f_n\|_{\ell_2}\leq  2^l \overline{r},\,\,\,\,\,\,\,\,\,\,\text{and}\,\,\,\,\,\,\,\,\,\,\|\hat{f}-f_n\|_{n}\leq  2^{l+1} \overline{r},
	\]
	which  by the previous argument implies that
	\[
	\|\hat{f}-f_n\|_{\ell_2}\leq  2^{l-1} \overline{r},\,\,\,\,\,\,\,\,\,\,\text{and}\,\,\,\,\,\,\,\,\,\, \|\hat{f}-f_n\|_{n}\leq  2^{l} \overline{r},
	\]
	and continuing recursively we arrive at 
	\[
	\|\hat{f}-f_n\|_{\ell_2}\leq  \overline{r},\,\,\,\,\,\,\,\,\,\,\text{and}\,\,\,\,\,\,\,\,\,\, \|\hat{f}-f_n\|_{n}\leq  2\overline{r}.
	\]
	The claim follows by noticing that Lemma \ref{lem6} and (\ref{eqn:upper_bound}) imply that
	\[
	\begin{array}{lll}
		\overline{r}   & \leq&    	 	  8  \tilde{C}\left[    F^{2.5} \sqrt{   \frac{\gamma}{n} } +F \sqrt{   \frac{ N  (\log N)^2 }{n}  }   + F\sqrt{\frac{N \left[(\log N)^2+ 2\log n\right] }{  n } }  + N^{-s/d} F^{1.5}\right]  +   2\sqrt{	\frac{\tilde{C} F^{2.5}   \gamma  }{n}  }  + \\
		& &  \displaystyle    \tilde{C}\left[  \sqrt{   \frac{ N  (\log N)^2 }{n}  } +      \sqrt{\frac{N \left[(\log N)^2+ \log n\right] }{  n }        } \right],  
	\end{array}
	\]  
	and the claim follows since  $\|f_n-f_{\tau}^*\|_{\infty} \leq N^{-s/d}.$
\end{proof}


\section{Proof of Theorem \ref{thm5} }

The proof is similar  to that of Theorem \ref{thm3}   relying in the lemmas  given next.

Just as in Lemma \ref{lem2}, we have the following result.

\begin{lemma}
	\label{lem7}
	Suppose  that  $\|  f_n -f_{\tau}^*\|_{\infty} \leq  c$ for a small enough constant $c$. Then
	\[
	\Delta^2(f,f_n)  \leq   \frac{1}{c_{\tau}} \left[\mathbb{E}\left(   \rho_{\tau}(Y- f(X))  -  \rho_{\tau}(Y- f_n(X))    \right)+    \| f_n-f_{\tau}^*\|_{\infty}  \Delta(f,f_n ) \sqrt{F  }    \right],
	\]
	and
	\[
	\|f-f_n\|_{\ell_2}^2  \leq   \frac{2F}{c_{\tau}} \left[\mathbb{E}\left(   \rho_{\tau}(Y- f(X))  -  \rho_{\tau}(Y- f_n(X))    \right)+    \| f_n-f_{\tau}^*\|_{\infty}  \|f-f_n\|_{\ell_2}   \sqrt{F  }    \right],
	\]
	for any $f \in 	\mathcal{G}(L,p,S,F) $, and for some constant $c_{\tau}$.
\end{lemma}

Similarly to Lemmas \ref{lem4}--\ref{lem6}, we have the following three lemmas.

\begin{lemma}
	\label{lem8} 	Suppose  that  $\|  f_n -f_{\tau}^*\|_{\infty} \leq  c$ for a small enough constant $c$.  The estimator  $\hat{f}$  satisfies
	\[
	\Delta^2(\hat{f},f_n)  \leq    \frac{1}{c_{\tau}}\left[M_n(\hat{f})  -  \hat{M}_{n}(\hat{f})     +     \| f_n-f_{\tau}^*\|_{\infty}  \Delta(\hat{f},f_n ) \sqrt{F  }  \right].
	\]
	Furthermore,
	\[
	\|\hat{f}-f_n\|_{\ell_2}^2  \leq    \frac{2F}{c_{\tau}}\left[M_n(\hat{f})  -  \hat{M}_{n}(\hat{f})     +     \| f_n-f_{\tau}^*\|_{\infty}\|\hat{f}-f_n\|_{\ell_2} \sqrt{F  }  \right].
	\]
\end{lemma}

\begin{lemma}
	\label{lem9}
	If
	\begin{equation}
		\label{eqn:cond3}
		3\mathbb{E}\left(\underset{f \in  \mathcal{G}(L,p,S,F),  \,\, \|f-f_n \|_{\ell_2}   \leq r }{\sup}\,  \frac{1}{n}\sum_{i=1}^{n}\xi_i( f(X_i)-   f_n(X_i))^2    \right) \leq  r^2, 
	\end{equation}
	for  $\{\xi_i\}_{i=1}^n$  Rademacher variables independent of $\{(x_i,y_i)\}_{i=1}^n$, 	and 
	\begin{equation}
		\label{eqn:cond4}
		\max\left\{ 4F\sqrt{   \frac{ \gamma}{n} } , 4F   \sqrt{\frac{\gamma }{3n}   }\right\}  \leq r.
	\end{equation}
	Then with probability at least  $1- e^{-\gamma}$,  
	$\|f-f_n\|_{\ell_2}^2 \leq   r^2$  with  $f  \in \mathcal{G}(L,p,S,F)$ implies
	\[
	\|f-f_n\|_n^2 \,\leq  \,  (2r)^2.
	\]
\end{lemma}

\begin{lemma}
	\label{lem10}
	Suppose that $\| \hat{f}-f_n\|     \leq  r_0$, with  $r_0$  satisfying (\ref{eqn:cond3})-(\ref{eqn:cond4}) and Assumption \ref{as5} holds. Also, with the notation of Definition \ref{def10}, suppose that  for the class $	\mathcal{G}(L,p,S,F) $  the parameters are chosen to satisfy 
	\[
	\begin{array}{l}
		\sum_{i=0}^{q} \log_2\left(   4\max\{t_i,\beta_i\}  \right)\log_2(n)\leq  L     \lesssim  n \epsilon_n,\,\,\,\,\,\,\, \max\{1,K\}\leq F,\\
		n\epsilon_n       \lesssim      \underset{i=1,\ldots,L}{\min} p_i,\,\,\,\,\,\, S\asymp n \epsilon_n \log n,\,\,\,\,\,\,    \underset{i=1,\ldots,L}{\max} p_i  \lesssim n,
	\end{array}
	\]
	where
	\[
	\epsilon_n \,=\,  \underset{i=0,1,\ldots,q}{\max}   n^{   -  \frac{2\beta_i^*}{2\beta_i^*    +t_i  }  }.
	\]
	Then 
	\[
	\begin{array}{lll}
		\,\|\hat{f}-f_n\|_{\ell_2}^2     	  & \lesssim&\displaystyle  r_0 \sqrt{F\epsilon_n   \log n    \left[\log L  +  L\log n    \right] }+r_0 \sqrt{F\epsilon_n   \log n    \left[\log L  +  L\log n+\log \left(n\right)    \right] }\\ 
		& & \displaystyle  +   r_0F^2\sqrt{   \frac{ \gamma}{n} }  +  \frac{ F^2   \gamma  }{n}+     r_0 F \underset{i=0,1,\ldots,q}{\max}   n^{   -  \frac{\beta_i^*}{2\beta_i^*    +t_i  }  }\\
	\end{array}
	\]
	with probability at least $1-2\exp(-\gamma)$.
\end{lemma}

\begin{proof}
	Let\[
	\mathcal{G} = \left\{    g \,:\,    g(x,y) = \rho_{\tau}(y-f(x)) - \rho_{\tau}(y-f_n(x))  ,\,\,\,    f\in \mathcal{G}(L,p,S,F) ,\,\,\,\,  \|f-f_n\|   \leq  r_0\right\}.
	\]
	Prooceding as in the proof of Lemma \ref{lem5},  we obtain that
	
	\begin{equation}
		\label{eqn:entropy2}
		\begin{array}{lll}
			\displaystyle  \mathbb{E}_{\xi} \left(    \underset{g \in    \mathcal{G} }{\sup} \,    \frac{1}{n} \sum_{i=1}^{n}\xi_i  g(x_i,y_i)   \right)
			& \leq&  \displaystyle     C \underset{0 < \alpha <  r_0  }{\inf} \left\{   4\alpha  +      \frac{12}{  \sqrt{n} }   \int_{\alpha}^{2  r_0 } \sqrt{   \log  \mathcal{N}(     \delta, \mathcal{G}(L,p,S,F) , \|\cdot\|_n  )   } d\delta \right\},\\
		\end{array}
	\end{equation}
	for some positive constant $C$.
	
	However, defining $V = \prod_{l=0}^{L+1} (p_l+1)$, then Lemma  5 in   \cite{schmidt2017nonparametric} and  (\ref{eqn:entropy2}) imply that 
	\begin{equation}
		\label{eqn:entropy3}
		\begin{array}{lll}
			\displaystyle  \mathbb{E}_{\xi} \left(    \underset{g \in    \mathcal{G} }{\sup} \,    \frac{1}{n} \sum_{i=1}^{n}\xi_i  g(x_i,y_i)   \right)    & \leq &\displaystyle    C \underset{0 < \alpha <  r_0  }{\inf} \left\{   4\alpha  +      \frac{12}{  \sqrt{n} }   \int_{\alpha}^{2   r_0 } \sqrt{   \log  \mathcal{N}(     \delta, \mathcal{G}(L,p,S,F) , \|\cdot\|_{\infty} )   } d\delta \right\}\\
			& \leq&4C\alpha    +    \frac{24 Cr_0}{\sqrt{n}}  \sqrt{      (S+1)    \log  \left(   2   \alpha^{-1}(L+1)V^2 \right)   },
		\end{array}
	\end{equation}
	where  $V = \prod_{l=0}^{L+1}(p_l+1)$. Therefore,  setting   
	\[
	\alpha =   r_0\sqrt{\frac{(S+1)\log ((L+1)  V^2)}{n}},
	\]
	(\ref{eqn:entropy3}) implies
	\[
	\displaystyle  \mathbb{E}_{\xi} \left(    \underset{g \in    \mathcal{G} }{\sup} \,    \frac{1}{n} \sum_{i=1}^{n}\xi_i  g(x_i,y_i)   \right)    \,\leq \,	4Cr_0\sqrt{\frac{(S+1)\log ((L+1)  V^2)}{n}}   +    \frac{24  r_0 C}{\sqrt{n}}  \sqrt{      (S+1)    \log  \left(   2(L+1)V^2    n\right)   }.
	\]
	Hence, as in  (\ref{eqn:e1}), we obtain that
	\begin{equation}
		\label{eqn:e2}
		\begin{array}{lll}
			\,	\|\hat{f}-f_n\|_{\ell_2}^2  &\leq &  \displaystyle \frac{  2\sqrt{F} }{	c_{\tau}}\bigg[	24r_0C\sqrt{\frac{F(S+1)\log ((L+1)  V^2)}{n}}   +    \frac{144  \sqrt{F}r_0C}{\sqrt{n}}  \sqrt{      (S+1)    \log  \left(   2(L+1)V^2    n \right)   }+\\
			& &   \displaystyle  +    4 r_0F^2\sqrt{   \frac{\gamma}{n} } +  \frac{100 F^2   \gamma  }{3n} +     \| f_n-f_{\tau}^*\|_{\infty} \|\hat{f}-f_n\|_{\ell_2}  F
			\bigg],\\
		\end{array}\, 
	\end{equation}
	with probability at least $1-e^{-\gamma}$.
	
	Furthermore, by Equation (26) in the proof of Theorem 1 in \cite{schmidt2017nonparametric} and the argument therein, we have that 
	\begin{equation}
		\label{eqn:linity}
		\|f_n -f_{\tau}^*\|_{\infty}    \,\leq \,  C^{\prime}   \underset{i=0,1,\ldots,q}{\max}     c^{ -\frac{\beta_i^*}{t_i}   } n^{   -  \frac{\beta_i^*}{2\beta_i^*    +t_i  }  },		
	\end{equation}
	for  positive  constants $c$  and  $C^{\prime}$. Hence,  combining (\ref{eqn:e2}) with (\ref{eqn:linity}) we arrive at
	\[
	\begin{array}{lll}
		\,	\|\hat{f}-f_n\|_{\ell_2}^2  &\leq &   \displaystyle \frac{  s\sqrt{F} }{	c_{\tau}}\bigg[	24r_0C\sqrt{\frac{F(S+1)\log ((L+1)  V^2)}{n}}   +    \frac{144  \sqrt{F}r_0C}{\sqrt{n}}  \sqrt{      (S+1)    \log  \left(   2(L+1)V^2    n \right)   }+\\
		& &   \displaystyle  +    4 r_0F^2\sqrt{   \frac{\gamma}{n} } +  \frac{100 F^3   \gamma  }{3n} +    r_0  F\,  C^{\prime}   \underset{i=0,1,\ldots,q}{\max}     c^{ -\frac{2\beta_i^*}{t_i}   } n^{   -  \frac{2\beta_i^*}{2\beta_i^*    +t_i  }  }
		\bigg],\\
		& \lesssim&\displaystyle  r_0 \sqrt{F\epsilon_n   \log n    \left[\log L  +  \log V    \right] }+r_0 \sqrt{F\epsilon_n   \log n    \left[\log L  +  \log V+\log \left(n\right)    \right] }\\
		& & \displaystyle  +   r_0F^2\sqrt{   \frac{ \gamma}{n} }  +  \frac{ F^2  \gamma  }{n}+     r_0  F  \underset{i=0,1,\ldots,q}{\max}   n^{   -  \frac{\beta_i^*}{2\beta_i^*    +t_i  }  }\\
		& \lesssim&\displaystyle  r_0 \sqrt{F\epsilon_n   \log n    \left[\log L  +  L\log n    \right] }+r_0 \sqrt{F\epsilon_n   \log n    \left[\log L  +  L\log n+\log \left(n\right)    \right] }\\ 
		& & \displaystyle  +   r_0F^2\sqrt{   \frac{ \gamma}{n} }  +  \frac{ F^2   \gamma  }{n}+     r_0 F \underset{i=0,1,\ldots,q}{\max}   n^{   -  \frac{\beta_i^*}{2\beta_i^*    +t_i  }  }\\
	\end{array}\, 
	\]
	where in the last four inequalities we have used the choice of the network parameters.
\end{proof}

\begin{lemma}
	\label{lem11}
	Let $r^*$ be defined as
	\[
	r^* = \inf \left\{ r>0\,:\,  	3\mathbb{E}\left(\underset{f \in   \mathcal{G}(L,p,S,F) ,  \,\,  \|f-f_n\|_{\ell_2}\leq s }{\sup}\,  \frac{1}{n}\sum_{i=1}^{n}\xi_i ( f(x_i)-  f_n(x_i))^2   \right)   < s^2 ,\,\,\forall s \geq r   \right\},
	\]
	for  $\{\xi_i\}_{i=1}^n$  Rademacher variables independent of $\{(x_i,y_i)_{i=1}^n\}$.
	Then  under the conditions of Lemma \ref{lem10},
	\[
	r^*  \leq    \displaystyle    \tilde{C}      \epsilon_n L \log^2 n, 
	\]
	for a constant $\tilde{C}>0$. 
\end{lemma}

\begin{proof}
	Proceeding as in the proof of Lemma \ref{lem6}, we obtain that  for a constant $C>0$,
	\[
	\begin{array}{lll}
		(r^* )^2  &\leq &  C\displaystyle \underset{0 < \alpha <  2 r^*  }{\inf} \left\{   \alpha  + \frac{r^* }{  \sqrt{n} } \sqrt{      \log \mathrm{N}\left( \alpha,\mathcal{G}(L,p,S,F) , \| \cdot\|_{\infty} \right)   }\right\} +   \frac{12 F^2 }{n}\\
		& \leq & C\displaystyle \underset{0 < \alpha <  2 r^*  }{\inf} \left\{   \alpha  + \frac{r^* }{  \sqrt{n} } \sqrt{     (S+1)\log\left( 2\alpha^{-1}     (L+1) V^2  \right) }\right\} +   \frac{12 F^2 }{n}\\
		&\lesssim&\displaystyle \underset{0 < \alpha <  2 r^*  }{\inf} \left\{   \alpha  + r^*  \sqrt{     \epsilon_n  \log n \left[L\log n  +  \log \alpha^{-1}\right]  }\right\} +   \frac{12 F^2 }{n}\\
	\end{array}
	\]
	where the second inequality follows from Lemma 5 ion \cite{schmidt2017nonparametric},  and the second by the choice of the parameters in the network. Hence, setting 
	\[
	\alpha    \,=\, r^* \sqrt{  \epsilon_n  L\log^2 n },
	\]
	we obtain that 
	\[
	r^*      \,\lesssim\, \sqrt{  \epsilon_n  L\log^2 n }.
	\]
	
\end{proof}

	\bibliographystyle{plainnat}
\bibliography{references}

\end{document}